\documentclass[a4paper,11pt,reqno]{amsart}
\usepackage{amsmath,graphicx,graphics,amssymb,psfrag,mathdots,mathtools,ulem,xcolor}
\usepackage[colorlinks=true]{hyperref}
\newtheorem{lem}{Lemma}
\newtheorem{con}{Conjecture}
\newtheorem{theo}{Theorem}

\newtheorem{prop}{Proposition}

\theoremstyle{remark}
\newtheorem{rem}{Remark}

\numberwithin{equation}{section}

\newcommand{\fd}{\Delta}
\newcommand{\m}{\operatorname{M}}

\newcommand{\M}{\operatorname{M}}

\newcommand{\ch}{\operatorname{q}}
\newcommand{\di}{\operatorname{d}}
\newcommand{\Z}{\mathbb Z}
\newcommand{\h}{\operatorname{H}}
\renewcommand{\mod}{\operatorname{mod}}

\newmuskip\pFqskip
\mathchardef\pFcomma=\mathcode`, 

\begin{document}

\title{Lozenge tilings of hexagons with removed core and satellites}

\author[Mihai Ciucu and Ilse Fischer]{\box\Adr}

\newbox\Adr
\setbox\Adr\vbox{ 
\centerline{ \large Mihai Ciucu} \vspace{0.3cm}
\centerline{Indiana University, Department of Mathematics}
\centerline{Bloomington, IN 47401, USA}
\centerline{{\tt mciucu@indiana.edu}} 
\vspace{0.5cm} \centerline{and}  \vspace{0.5cm}
\centerline{ \large Ilse Fischer} \vspace{0.3cm}
\centerline{Universit\"at Wien, Fakult\"at f\"ur Mathematik}
\centerline{Oskar-Morgenstern-Platz 1, 1090 Wien, Austria}
\centerline{{\tt ilse.fischer@univie.ac.at}} 
}

\thanks{The authors acknowledge support by the National Science Foundation, DMS grant 1501052 and Austrian Science Foundation FWF, START grant Y463 and SFB grant F50}

\begin{abstract} We consider regions obtained from 120 degree rotationally invariant hexagons by removing a core and three equal satellites (all equilateral triangles) so that the resulting region is both vertically symmetric and 120 degree rotationally invariant, and give simple product formulas for the number of their lozenge tilings. 
We describe a new method of approach for proving these formulas, and give the full details for an illustrative special case. As a byproduct, we are also able to generalize this special case in a different direction, by finding a natural counterpart of a twenty year old formula due to Ciucu, Eisenk\"olbl, Krattenthaler and Zare, which went unnoticed until now. The general case of the original problem will be treated in a subsequent paper. We then work out consequences for the correlation of holes, which were the original motivation for this study.
\end{abstract}


\maketitle

\section{Introduction}

The fact that not only the number of plane partitions that fit in a box (equivalently, lozenge tilings\footnote{ A lozenge is the union of two adjacent unit triangles on the triangular lattice; a lozenge tiling of a lattice region $R$ is a covering of $R$ by lozenges that has no gaps or overlaps.} of a hexagon), but also all the symmetry classes (a total of ten) are given by simple 
\begin{figure}[h]
\centerline{
\hfill
{\includegraphics[width=0.50\textwidth]{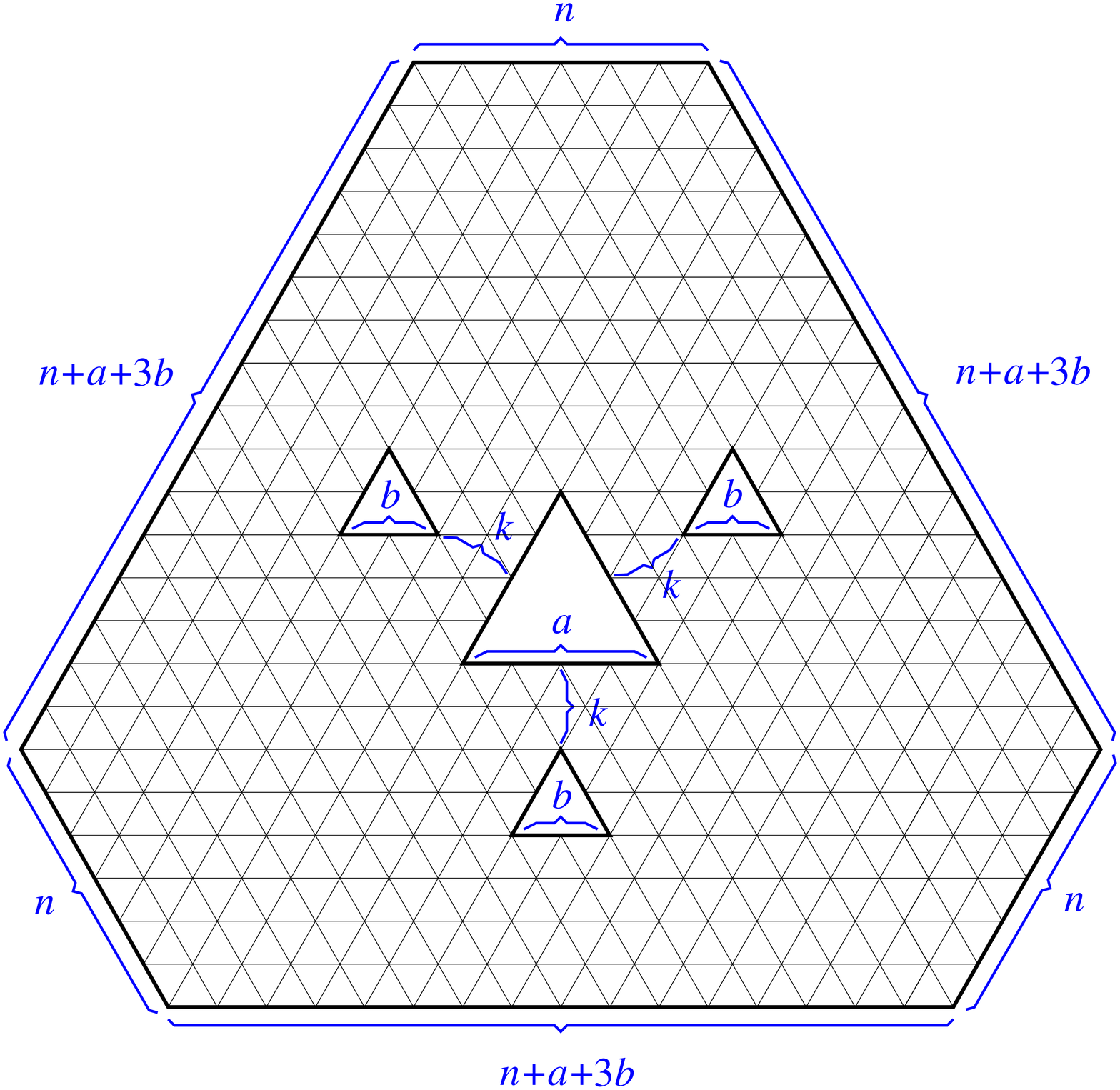}}
\hfill
{\includegraphics[width=0.50\textwidth]{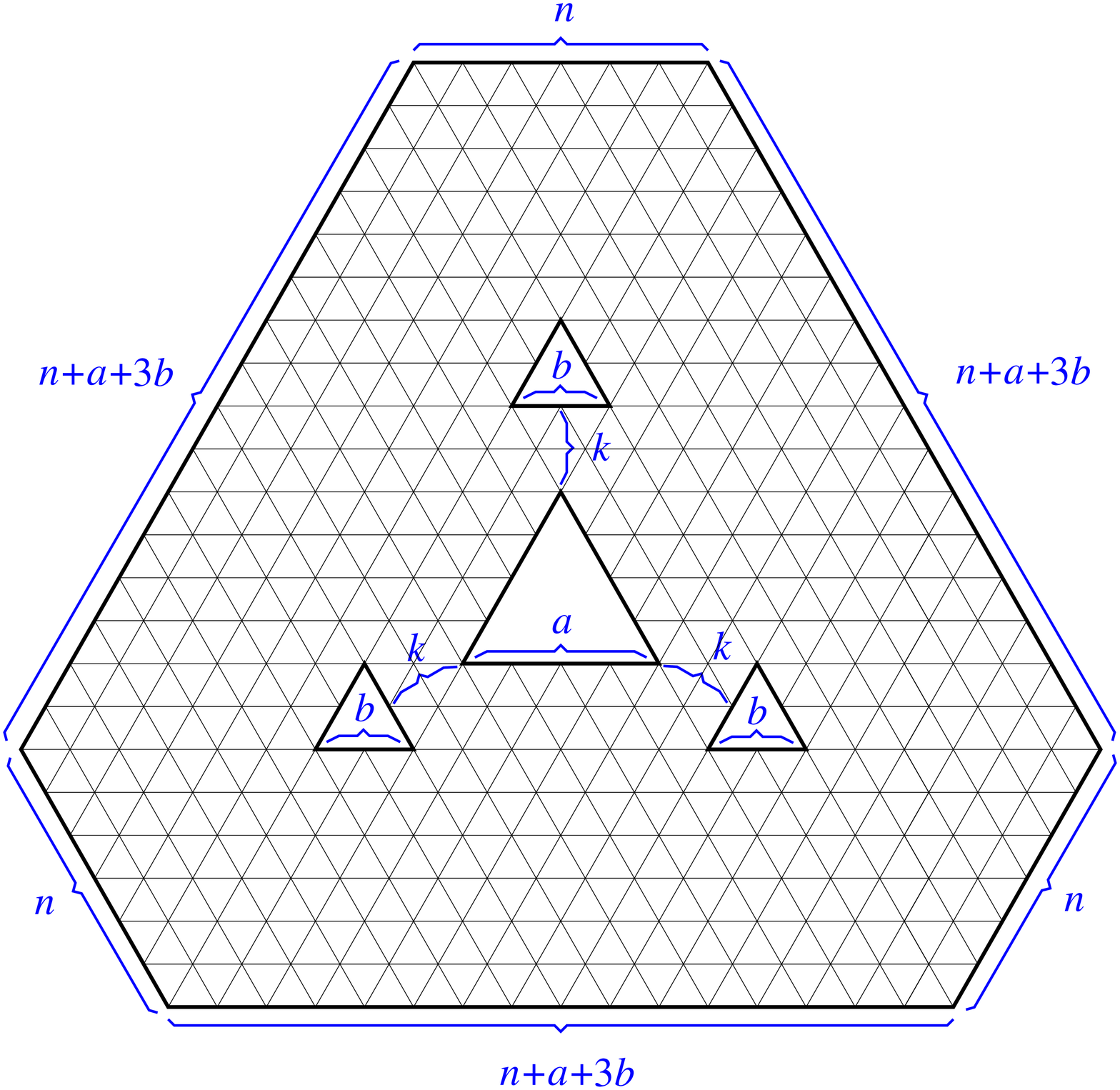}}
\hfill
}
\caption{\label{faa} The region $S_{n,a,b,k}$ (left)  and $S'_{n,a,b,k}$ (right) for $n=6$, $a=4$, $b=2$, $k=1$.}
\end{figure}
product formulas, is of singular beauty in enumerative combinatorics\footnote{ To specify just one of them, MacMahon proved \cite{MacM} --- in an equivalent formulation --- that the number of lozenge tilings of a hexagon of side lengths $a$, $b$, $c$, $a$, $b$, $c$ (in cyclic order) on the triangular lattice is equal to
\begin{equation}
\prod_{i=1}^a\prod_{j=1}^b\prod_{k=1}^c \frac{i+j+k-1}{i+j+k-2}.
\nonumber
\end{equation}
The other nine are only somewhat more complicated.}. This has been a rich source of inspiration for many researchers over the last four decades. Just to skim the surface, we mention \cite{And,StanPP,Kup,Ste,Bres,KKZ} and the survey \cite{kratt} for more recent developments. Works of the first author inspired by this include \cite{ppone,pptwo,sc,ec,ef,ov,ff,fv}. Probabilistic aspects were studied by Cohn, Larsen and Propp \cite{CLP}, Borodin, Gorin and Rains \cite{Borodin}, and Bodini, Fusy and Pivoteau \cite{Bodini}. Another extension was given by Vuleti\'c \cite{Vuletic}. 

\parindent15pt
This paper considers regions obtained from 120 degree rotationally invariant hexagons by removing a core and three equal satellites (all equilateral triangles) so that the resulting region is both vertically symmetric and 120 degree rotationally invariant (Figure \ref{faa} shows the two types of regions that are obtained; see Section 2 for the precise definitions), and gives simple product formulas for the number of their lozenge tilings.

The reader may find interesting the account of how these regions were found.

The special case of these regions when the core is empty was discovered by the first author in 1999, when he noticed that the number of its lozenge tilings seems to always factor fully into relatively small prime factors (such integers are sometimes referred to as ``round'').

This seemed a very hard result to prove (indeed, even guessing the precise product formula seemed exceedingly hard). 
Using the Lindstr\"om-Gessel-Viennot theorem \cite{Lindstr,GesselViennot} it is clearly possible to derive a determinant for the number of lozenge tilings in the case of even size satellites, but
the identification of factors method for evaluating determinants which had proved successful on many occasions before (see e.g. \cite{onethird,CEKZ,cutcor,free,kratt}) was not applicable due to the lack of a polynomial parameter.
Furthermore, it was the odd size that interested the first author most. The reason had to do with \cite{sc}, where he discovered that the distribution of gaps in random lozenge tilings is governed by Coulomb's law of two dimensional electrostatics: \cite{sc} could handle a multitude of even holes but only a single odd hole, and in order to support the conjecture that 2D Coulomb governs the distribution of holes for arbitrary holes (this conjecture was published in \cite{ov}) it was desirable to have an example with three odd holes; the fact that they were not collinear made this instance especially interesting. Having an exact, simple product formula for the number of tilings of the hexagon with three holes, the correlation of the holes can be worked out and its asymptotics determined, confirming thus the above mentioned ``electrostatic conjecture.''

The first author mentioned this observation to Christian Krattenthaler in 2003, and considered briefly a project to attempt proving it,
but the project was abandoned due to the above mentioned complications and limitations. 

An important step ahead was achieved in 2010, when the first author noticed that the round factorization persists if a fourth hole is added at the center. The reason this is so helpful is because it introduces a new parameter in the data, and the counts can be proved to be polynomials in this parameter. Then the data is not just integers that factor into relatively small (but otherwise mysterious) prime factors, but polynomials in the new parameter that factor fully into linear factors. This also gives a more objective measure of ``roundness'' than just factorization of integers into relatively small prime factors.


While the first author was working on enumerating the tilings of these regions in 2014, he showed them to Tri Lai (he was the first author's Ph.D. student at the time), who then in 2017 co-wrote the paper \cite{LR} which involves these regions.
To be precise, \cite{LR}
focuses on counting the lozenge tilings of these regions which are invariant under rotation by $120^\circ$, and also of those which are both vertically symmetric and invariant under this rotation (these follow, after a considerable amount of work, by
applying the factorization theorem \cite{FT} and Kuo condensation \cite{KuoOne}; see \cite{symffa,symffb} for earlier examples).
The straight count of lozenge tilings is
not mentioned in \cite{LR}. However, it turns out (see Conjecture~\ref{tbaa} below) that the straight count and the $120^\circ$-rotationally-invariant count are very closely related!

We are now finally presenting ourselves these regions found many years ago by the first author, and our work on
the problem of counting their lozenge tilings, a question that seems singularly hard in the circle of lozenge tiling problems. This is due to a large extent to the fact that it does not seem possible to extend this family so as to obtain a proof by applying Kuo's graphical condensation method, and also that it does not seem possible to deduce it from other results using standard combinatorial arguments.

It seemed very difficult even to find an explicit conjectural formula for the number of lozenge tilings of these regions, even with the great help that the extra parameter (the size of the core) brought in. The second author succeeded in finding one in 2016, and this is how this collaboration began.

\section{Statement of main results and conjectures}

The regions we present in this paper are hexagons on the triangular lattice\footnote{ Throughout this paper, with the exception of Section 3, we draw the triangular lattice so that one family of lattice lines is horizontal.} with one central and three satellite up-pointing triangular holes so that

\begin{enumerate}
\item[$(i)$] the hexagon with holes is both vertically symmetric and 120 degree rotationally invariant, and
\item[$(ii)$] the gap between each satellite and the core can be bridged by a string of whole lozenges lined up along their long diagonals.
\end{enumerate}

This common description leads to two families of different-looking regions, depending on whether the satellites point towards or away from the core. In the former case, condition $(ii)$ above amounts to the requirement that the side-length of the core is even (see the picture on the left in Figure \ref{faa} for an illustration), while in the latter the side-length of the satellites is required to be even (an example of this is shown on the right in  Figure \ref{faa}).

Assume therefore that $n$, $a$, $b$ and $k$ are non-negative integers with $a$ even, and define $S_{n,a,b,k}$ to be the region obtained from the hexagon $H_{n,n+a+3b}$ of side-lengths $n$, $n+a+3b$, $n$, $n+a+3b$,  $n$, $n+a+3b$ (clockwise from top) by removing a triangle of side $a$ from its center and three satellite triangular holes, each of side $b$, as indicated on the left in Figure \ref{faa} (we emphasize that $k$ is the length of a chain of lozenges that would bridge the gap between each satellite and the core; there are $2k$ lattice spacings between a satellite and the core). For non-negative integers  $n$, $a$, $b$ and $k$ with $b$ even, define $S'_{n,a,b,k}$ to be the region obtained from the same hexagon $H_{n,n+a+3b}$ by removing a triangle of side $a$ from its center and three satellite triangular holes of side $b$ as indicated on the right in Figure \ref{faa} ($k$ has the same significance as in the picture on the left in that figure). For both cases one must have $k\leq n/2$ in order for the satellites to be contained in the region.



Our original interest in these regions (and indeed the reason we found them) came from discovering (see \cite{sc}) that for quite general distributions of even\footnote{ With the exception of one, which could be odd.} triangular holes around the center of a very large hexagon, the number of lozenge tilings of the hexagon with holes varies with the position of the holes precisely\footnote{ In the double limit as first the enclosing hexagon becomes infinite, and then the separation between the holes approaches infinity.} as the exponential of the negative of the 2D electrostatic potential of a naturally corresponding system of electrical charges. This striking observation lended itself to generalization. We needed an example involving non-collinear holes, if possible of either even or odd side-lengths, for which we could work out the needed asymptotics. 

From this point of view, the more interesting family for us is $S_{n,a,b,k}$, as it can have three non-collinear odd charges ($S'_{n,a,b,k}$ can have at most one odd charge, a case already covered by \cite{sc}). The formula for the number of tilings of $S_{n,a,b,k}$ can then be used to determine the asymptotics of the correlation (see \eqref{ebi} for its definition) of the system of its four holes, providing thus the first example in the literature involving large non-collinear odd holes; we work this out in Theorem \ref{tbab} below (collinear holes of arbitrary size on the square lattice were treated in \cite{gd,fin}, and unit holes of arbitrary positions on the square lattice in \cite{Dub}; see also \cite{ec} for arbitrary holes of side two on the triangular lattice, and the extension \cite{gsp} to weighted doubly periodic planar bipartite lattices in the liquid phase of the Kenyon-Okounkov-Sheffield classification \cite{KOS} of the dimer models).


We therefore focus in this paper on the regions $S_{n,a,b,k}$. Analogous results to the ones we present below exist also for the regions $S'_{n,a,b,k}$, but due to the involved nature of the arguments and the fact that the $S_{n,a,b,k}$'s already provide us with the asymptotics we were after in the first place, we do not present them here.

Throughout this paper we define products according to the convention
\begin{equation}
\label{eba}
\prod _{k=m} ^{n-1}\text {\rm Expr}(k)=
\begin{cases}
\prod _{k=m}^{n-1} \text {\rm Expr}(k)&n>m,\\
\hskip0.478in
1&n=m,\\
\dfrac{1}{\prod _{k=n} ^{m-1}\text {\rm Expr}(k)}&n<m.
\end{cases}
\end{equation}

We recall that the Pochhammer symbol $(\alpha)_k$ is defined for any integer $k$ to be 
\[ 
(\alpha)_k:=\prod_{i=0}^{k-1}(\alpha+i),
\]
thus according to \eqref{eba}
\begin{equation}
\label{ebb}
(\alpha)_k:=
\begin{cases}
 \alpha(\alpha+1)\cdots(\alpha+k-1)&\text {if
}k>0,\\
1&\text {if }k=0,\\
1/((\alpha-1)(\alpha-2)\cdots(\alpha+k))&\text {if }k<0.
\end{cases}
\end{equation}
For half-integers $k$, define the Pochhammer symbol $(\alpha)_k$ by
\begin{equation}
\label{ebc}
(\alpha)_k:=\frac{\Gamma(\alpha+k)}{\Gamma(\alpha)}.
\end{equation}

Denote by $\M(R)$ the number of lozenge tilings of the region $R$ on the triangular lattice, and by $\M_r(R)$ the number of its lozenge tilings that are invariant under rotation by 120 degrees. 


Our main goal in this paper is to find a formula for $\M(S_{n,a,b,k})$. When $k=0$, the three satellites touch the core, and due to forced lozenges, removing the core and the three satellites is equivalent (as far as counting lozenge tilings of the leftover region) to removing just a larger core, of side $a+3b$. Therefore, the case $k=0$ follows by the main result of \cite{CEKZ} (see Theorem 1 there).

It turns out that there is a simple relationship between $\M(S_{n,a,b,k})$ and $\M_r(S_{n,a,b,k})$, the number of lozenge tilings of the region $S_{n,a,b,k}$ which are invariant under rotation by 120 degrees.  A formula for the latter was proved by Lai and Rohatgi in \cite{LR}. However, in Theorem 2 we provide a (rather radical) rewriting of their formula\footnote{In the case when $n$ is even; a similar rewriting holds for $n$ odd, but we do not need it in this paper.}, which has the advantage that it works for both even and odd satellite sizes (the original formulas were very different in the two cases; compare equations (2.8) and (2.9) of \cite{LR} with equations (2.10) and (2.11) of \cite{LR}). 


The simplest way to express our formula for $\M(S_{n,a,b,k})$ is to introduce the {\it normalized counts} $\overline{\M}(S_{n,a,b,k})$ and $\overline{\M}_r(S_{n,a,b,k})$ as follows:
\begin{align}
\label{ebca}
\overline{\M}(S_{n,a,b,k})&:=\frac{\M(S_{n,a,b,k})}{\M(S_{n,a,b,0})}\\
\nonumber
\\
\label{ebcb}
\overline{\M}_r(S_{n,a,b,k})&:=\frac{\M_r(S_{n,a,b,k})}{\M_r(S_{n,a,b,0})}
\end{align}

Then our formula for $\M(S_{n,a,b,k})$ follows (using also the two paragraphs above) from the following conjecture. 

\begin{con}
\label{tbaa}
For non-negative integers $n$, $a$, $b$ and $k$ with $a$ even we have
\begin{equation}
\label{ebcc}
\frac{\overline{\M}(S_{n,a,b,k})}{\overline{\M}_r(S_{n,a,b,k})^3}=\left[\prod_{i=1}^k\frac{(a+6i-4)(a+3b+6i-2)}{(a+6i-2)(a+3b+6i-4)}\right]^2.
\end{equation}
\end{con}

While this is still strictly speaking a conjecture, we mention that we do have a new approach to tackle it, which we are confident that will lead to a proof. 
We describe in Sections~\ref{detsec} and \ref{exsec} this new method, which uses the identification of factors method on a particularly convenient determinant, and give the details of the proof for the special case when $b=0$. This case corresponds to the cored hexagons treated in \cite{CEKZ}. As a byproduct, we are able to deal with a different generalization of this special case, which leads us to a new family of regions whose number of lozenge tilings is expressible by a product formula; see Theorem~\ref{newtheo} in Section~\ref{exsec} (in fact, using results from \cite{3b}, this can be generalized; see Remark~\ref{remgeneralize} in Section~\ref{exsec}). The relative briefness of the proof we present here compared to the original proof in \cite{CEKZ} illustrates the advantages of our new method. Some details still need to be worked out for the proof of the general case, which will be presented in a subsequent paper. 

The following two related results are also included in this paper:
\begin{itemize} 
\item In Section~\ref{withoutnsec}, we derive a determinantal formula  for 
$\m(S_{n,a,b,k})$ which is valid for even $b$, with the remarkable feature that the size of the underlying determinant depends neither on $n$ nor on $k$, and it can therefore be used to verify the formula for each particular choice of $a,b$ ($b$ even).
\item In Section~\ref{leadingsec}, we prove, using yet another determinantal expression, that for even $b$, $\M(S_{n,a,b,k})$ and our conjectured expression for it are polynomials in $a$ with the same leading coefficient. 
\end{itemize}

\begin{rem} It is amusing to note that the product on the right hand side of \eqref{ebcc} can be written as
\begin{equation*}
\frac
{\dfrac{\left(\frac{a}{6}+\frac13\right)_k}{\left(\frac{a}{6}+\frac23\right)_k}}
{\dfrac{\left(\frac{a+3b}{6}+\frac13\right)_k}{\left(\frac{a+3b}{6}+\frac23\right)_k}},
\end{equation*}
and that $a$ is the size of the core, while $a+3b$ is the size of the enlarged core arising in the special case $k=0$. It is remarkable that this ratio does not depend on $n$.
\end{rem}

The relationship between the un-normalized counts $\M(S_{n,a,b,k})$ and $\M_r(S_{n,a,b,k})$ is detailed in the following equivalent restatement of Conjecture 1.

\begin{con}
\label{tbab}
For non-negative integers $n$, $a$, $b$ and $k$ with $a$ even we have
\begin{align}
&
\frac{\M(S_{n,a,b,k})}{\M_r(S_{n,a,b,k})^3}
=
\prod_{i=1}^k\left[\frac{(a+6i-4)(a+3b+6i-2)}{(a+6i-2)(a+3b+6i-4)}\right]^2
\nonumber
\\
\nonumber
\\
&
\times
\begin{cases}
 \left[\dfrac{(a+3b+2)_{n/4,6}(a+3b+3n/2+1)_{n/4,6}}
            {(a+3b+4)_{n/4,6}(a+3b+3n/2+5)_{n/4,6}}\right]^2,&
\!\!\!\!\!\!\!\!\!\!\!\!\!\!\!\!\!\!\!\!\!\!\!\!n\ \mathrm{even}
\\
\\
\dfrac{(a+3b+3n+2)^2}{4(a+3b+3(n+1)/2-1)^2} 
\left[\dfrac{(a+3b+2)_{(n+1)/4,6}(a+3b+3(n+1)/2+1)_{(n+1)/4,6}}
            {(a+3b+4)_{(n+1)/4,6}(a+3b+3(n+1)/2+5)_{(n+1)/4,6}} \right]^2,&
\\
\ \ \ \ \ \ \ \ \ \ \ \ \ \ \ \ \ \ \ \ \ \ \ \ \ \ \ \ \ \ \ \ \ \ \ \ \ \ \ \ \ \ \ \ \ \ \ \ \ \ \ \ \ \ \ \ \ \ \ \ \ \ \ \ \ \ \ \ \ \ \ \ \ \ \ \ \ \ \ \ \ \ \ \ \ \ \ \ \ \ \ \ \ \ \ \ \ \ \ \ \ \ 
n\ \mathrm{odd}
\\
\end{cases}
\label{ebd}
\end{align}
where $(\alpha)_{k,m}:= m^k \left( \frac{\alpha}{m} \right)_k$ and the half-integer index Pochhammer symbols are defined by \eqref{ebc}.

\end{con}


%

We note that when the parameter $n$ is even in the above formula, it simplifies (writing the $n$- and $a$-parameters as $2n$ and $2a$, to spell out their evenness) to


%
\begin{equation}
\frac{\M(S_{2n,2a,b,k})}{\M_r(S_{2n,2a,b,k})^3}
=
\left[\dfrac
{\left(\dfrac{a}{3}+\dfrac{1}{3}\right)_k
\left(\dfrac{a}{3}+\dfrac{b}{2}+k+\dfrac{1}{3}\right)_{n/2-k} \left(\dfrac{a}{3}+\dfrac{b}{2}+\dfrac{n}{2}+\dfrac{1}{6}\right)_{n/2}}
{\left(\dfrac{a}{3}+\dfrac{2}{3}\right)_k
\left(\dfrac{a}{3}+\dfrac{b}{2}+k+\dfrac{2}{3}\right)_{n/2-k} \left(\dfrac{a}{3}+\dfrac{b}{2}+\dfrac{n}{2}+\dfrac{5}{6}\right)_{n/2}}
\right]^2.
\label{ebe}
\end{equation}
%


\begin{rem} In the special case when $a=b=0$ --- when our region becomes a regular hexagon $H_n$ of side $n$ --- the two branches of the formula in Conjecture \ref{tbab} unify to give
\begin{equation}
\label{ebfa}
\frac{\M(H_n)}{\M_r(H_n)^3}=\left[\frac
{\left(\dfrac13\right)_n}
{\left(\dfrac23\right)_n}\right]^2,
\end{equation}
a result that readily follows from the well-known formulas for symmetry classes of plane partitions (compare Cases 1 and 3 in \cite{StanPP}).
\end{rem}

Throughout the asymptotic analysis, we will focus on the case when the parameter $n$ is even. This will help keep its length manageable, while capturing the details of the asymptotics of our formulas. Analogous results exist for odd $n$.

We present now our (rather radical) rewriting of the formulas for $\M_r(S_{n,a,b,k})$ found by Lai and Rohatgi \cite{LR} (in line with the previous paragraph, we only treat here the case when the $n$-parameter is even; see also footnote 6). The new form has the advantage that it works for both even and odd satellite sizes (the original formulas were quite different in the two cases; compare equations (2.8) and (2.9) in \cite{LR} with equations (2.10) and (2.11) in \cite{LR}).

We emphasize that products with index limits out of order  are to be interpreted according to the formula we presented at the beginning of Section 2. 

\begin{theo} \cite{LR}
\label{tbb}
Let $n$, $a$, $b$ and $k$ be non-negative integers.

For even $n$ we have
\begin{align}
&
\M_r(S_{2n,2a,b,k})=
\dfrac{\left(\dfrac{a}{2}+\dfrac{k}{2}+\dfrac{1}{2}\right)_k \left(a+2n+\dfrac{3b}{2}+\dfrac12\right)_n}
                             {2^{n^2-n-k^2-k} \left(\dfrac{b}{2}+n-k+\dfrac12\right)_k \left(\dfrac{1}{2}\right)_{n-k}}
\nonumber
\\
&
\times
\left(\dfrac{a}{2}+\dfrac{b}{2}+\dfrac{k}{2}+\dfrac12\right)_k \left(\dfrac{a}{2}+b+n-\dfrac{k}{2}+\dfrac{1}{2}\right)_k 
\left(\dfrac{a}{2}+\dfrac{b}{2}+n-\dfrac{k}{2}+\dfrac12\right)_k
\nonumber
\\
&
\times
\left[
\prod_{i=1}^{n/2-k}\left(a+\dfrac{3b}{2}+3k+2i\right)_i \left(a+\dfrac{3b}{2}+3k+2i-1\right)_{i-1}
\right.
\nonumber
\\
&\ \ \ 
\times
\prod_{i=1}^{k} \left(\dfrac{a}{2}+\dfrac{i}{2}\right)_{i-1}
\prod_{i=1}^{n/2-1}\left(a+\dfrac{3b}{2}+\dfrac{3n}{2}+i+\dfrac12\right)_{2i}
\nonumber
\\
&
\left.
\times
\prod_{i=1}^{n-k-1}\dfrac{1}{\left(\frac{1}{2}\right)_i}
\prod_{i=1}^k \dfrac{\left(a+b+2i+k\right)_{n-i-k} (a+b+2n+k-2i+2)_{b-2k+4i-3}}
                   {\left(\frac{1}{2}\right)_{i} (2i)_{b-1} \left(i+ \frac{b-1}{2}\right)_{n-k}}
\right]^2,
\label{ebg}
\end{align}
%
while for odd $n$ we have
\begin{align}
&
\M_r(S_{2n,2a,b,k})=
\dfrac{\left(\dfrac{a}{2}+\dfrac{k}{2}+\dfrac{1}{2}\right)_k \left(a+2n+\dfrac{3b}{2}+\dfrac12\right)_n}
                             {2^{n^2-n-k^2-k} \left(\dfrac{b}{2}+n-k+\dfrac12\right)_k \left(\dfrac{1}{2}\right)_{n-k}}
\nonumber
\\
&
\times
\left(\dfrac{a}{2}+\dfrac{b}{2}+\dfrac{k}{2}+\dfrac12\right)_k \left(\dfrac{a}{2}+b+n-\dfrac{k}{2}+\dfrac{1}{2}\right)_k 
\left(\dfrac{a}{2}+\dfrac{b}{2}+n-\dfrac{k}{2}+\dfrac12\right)_k
\nonumber
\\
&
\times
\left[
\prod_{i=1}^{(n-1)/2-k}\left(a+\dfrac{3b}{2}+3k+2i\right)_i \left(a+\dfrac{3b}{2}+3k+2i+1\right)_{i}
\right.
\nonumber
\\
&\ \ \
\times
\prod_{i=1}^{k} \left(\dfrac{a}{2}+\dfrac{i}{2}\right)_{i-1}
\prod_{i=0}^{(n-1)/2-1}\left(a+\dfrac{3b}{2}+\dfrac{3n}{2}+i+1\right)_{2i+1}
\nonumber
\\
&\ \ \
\left.
\times
\prod_{i=1}^{n-k-1}\dfrac{1}{\left(\frac{1}{2}\right)_i}
\prod_{i=1}^k \dfrac{(a+b+2i+k)_{n-i-k} (a+b+2n+k-2i+2)_{b-2k+4i-3}}
                   {\left( \frac{1}{2} \right)_{i} (2i)_{b-1} \left(i + \frac{b-1}{2} \right)_{n-k}}
\right]^2.
\label{ebh}
\end{align}



\end{theo}

\begin{rem} It is worth mentioning that originally we worked out the above formulas in the case when $b$ is odd, patterned on the product formula for $\M(S_{n,a,b,k})$ that we discovered; it was this that led us to the expressions in \eqref{ebg} and \eqref{ebh}. An interesting feature of these formulas (which equations (2.8) and (2.9) of \cite{LR} do not possess) is that the expressions on their right hand sides are defined also for even $b$. It is most remarkable --- given how different equations (2.8) and (2.9) of \cite{LR} (which correspond to even $b$) are from equations (2.10) and (2.11) of \cite{LR} (which correspond to odd $b$) --- that for odd values of $b$, the expressions on the right hand sides of \eqref{ebg} and \eqref{ebh} above still give the correct number of $120^\circ$-rotationally-invariant tilings of $S_{2n,a,b,k}$.
\end{rem}

To state our next result, we need to define the correlation of holes in a sea of dimers. The original definition, for two monomers on the square lattice, is due to Fisher and Stephenson \cite{FS}. It was extended to an arbitrary finite collection of holes on the triangular lattice in \cite{sc}. In particular, the correlation of the core and the three satellites is defined as
\begin{equation}
\omega(a,b,k):=\lim_{n\to\infty}\overline{\M}(S_{2n,a,b,k})=\lim_{n\to\infty}\frac{\M(S_{2n,a,b,k})}{\M(S_{2n,a,b,0})}.
\label{ebi}
\end{equation}

We recall that the Barnes $G$-function $G(z)$ is defined for complex $z$ to be
\begin{equation}
G(z+1)=(2\pi)^{z/2}\exp\left(-\frac{z+z^2(1+\gamma)}{2}\right)
\prod_{k=1}^\infty\left\{\left(1+\frac{z}{k}\right)^k \exp\left(\frac{z^2}{2k}-z\right)\right\},
\end{equation}
where $\gamma$ is Euler's constant.

In fact, since in Theorem \ref{tbc} below the argument of $G(z)$ is always either a non-negative integer or a non-negative half-integer, it will be enough for us to know the values of $G$ at such values.

The function $G(z)$ satisfies the recurrence
\begin{equation}
G(z+1)=\Gamma(z)G(z),
\label{ebj}
\end{equation}
and thus for non-negative integers $n$ it is given by
\begin{equation}
G(n)=\prod_{i=0}^{n-2}i!
\label{ebk}
\end{equation}
(we note that for $0\leq n\leq1$, when the product limits are out of order, we use the general convention \eqref{eba} to obtain $G(0)=0$ and $G(1)=1$).

On the other hand, by the recurrence \eqref{ebj}, we have
\begin{equation}
G\left(n+\frac12\right)=G\left(\frac12\right)\Gamma\left(\frac12\right)\Gamma\left(\frac32\right)\cdots\Gamma\left(n-\frac12\right).
\label{ebl}
\end{equation}

All the values we need are then specified by the known fact (see e.g. \cite{Finch}, Section 2.15, p. 136) that
\begin{equation}
G\left(\frac12\right)=\frac{e^{1/8}2^{1/24}}{A^{3/2}\pi^{1/4}},
\label{ebm}
\end{equation}
where $A$ is the Glaisher-Kinkelin constant\footnote{The Glaisher-Kinkelin constant (see \cite{Glaish}) is the value $A$ for which                                
\begin{equation}
\label{GK}
\lim_{n\to\infty}                                                                        
\dfrac                                                                                    
 {0!\,1!\,\cdots\,(n-1)!}                                                                 
 {n^{\frac{n^2}{2}-\frac{1}{12}}\,(2\pi)^{\frac{n}{2}}\,e^{-\frac{3n^2}{4}}}              
=                                                                                         
\dfrac                                                                                    
 {e^{\frac{1}{12}}}                                                                       
 {A}                                                                                      
\end{equation}.}. It is interesting that our results would lead one to guess this very value for $G(1/2)$, had it not already been known (see Remark~\ref{barnes} for a detailed explanation).

We are now ready to state the main asymptotic result of this paper, which is what our hexagonal regions with four holes were designed for.

\begin{theo}
\label{tbc}
Assuming that Conjecture 1 holds, for non-negative integers $a$, $b$ and $k$ with $a$ even we have
\begin{equation}
\omega(a,b,k)\sim
\left[
\frac{G\left(\frac{a}{2}+1\right)}{G\left(\frac{a}{2}+\frac{3b}{2}+1\right)}
\right]^2
\left\{
3^{b^2/4}G\left(\frac{b}{2}+1\right)^{\!\!2}k^{b(a+b)/2}
\right\}^3
,\ \ \ k\to\infty,
\label{ebn}
\end{equation}
where $G$ is the Barnes $G$-function.
\end{theo}

\begin{rem}
This proves the electrostatic conjecture mentioned above (in other words, up to the specification of the multiplicative constant, Conjecture \ref{tca}) for the system of holes consisting of the core and the three satellites, achieving this way the original motivating goal of this work (see also the equivalent form \eqref{ech} of \eqref{ebn}).
We discuss in detail in Section 3 the strong evidence \eqref{ebn} provides for Conjecture \ref{tca}
in this special case.
\end{rem}

\begin{rem} 
\label{barnes}
We note that since $a$ is even, when $b$ is also even the values of the Barnes $G$-function in \eqref{ebn} are simply given by \eqref{ebk}. It is most remarkable that \eqref{ebn} holds also for odd $b$, when the right hand side involves the fourth power of the complicated constant  \eqref{ebm}.

In fact, we could have guessed the value of $G(1/2)$ (had it not been known already) from the natural assumption that \eqref{ebn} holds also for odd $b$. Indeed, set $a=0,b=1$ in \eqref{ebn}, and compare the leading coefficient in $k$ on the left hand side (which we obtain explicitly from the asymptotic analysis of our formulas) with the coefficient of the power of $k$ on the right hand side of the thus specialized \eqref{ebn}. Using \eqref{ebj}, this gives a linear equation for $G(1/2)^4$, which leads us precisely to the value in~\eqref{ebm}!
\end{rem} 







The special case of Conjecture \ref{tbaa} when $a=0$, when combined with Theorem 1 of \cite{3b}, affords a product formula for the number of lozenge tilings of a hexagon with a triad of bowties (an arrangement such as pictured in Figure \ref{fba}) removed from its center. We obtain the following result.

\begin{theo}
\label{triadform}
Let $T_{n,k,B,a,b,c}$ be the region obtained from the hexagon whose side-lengths alternate between $n+a+b+c$ and $n+3B-a-b-c$ $($with the top side of length $n+a+b+c$$)$ by removing from its center three bowties in a triad formation as indicated in Figure $\ref{fba}$, where the outer lobe sizes are $a$, $b$, $c$, the inner lobe sizes $B-a$, $B-b$, $B-c$ $($counterclockwise from top$)$, the distance between two bowtie nodes is~$3k+3B-a-b-c$, and the distance between the outer lobes and the facing hexagon sides is $n-k$.  

Then, writing $\langle n\rangle = G(n+1)$ for short,  we have
\begin{align}
&
\frac{\M(T_{n,k,B,a,b,c})}{\M(S_{n,0,B,k})}=
\frac{\langle 3k+B\rangle^3}{\langle 3k\rangle \langle B\rangle^3}
\frac{\langle n+k+B\rangle^3 \langle n-k+2B\rangle^3}{\langle n-2k+B\rangle^3 \langle n+2k+2B\rangle^3}
\nonumber
\\
&
\times
\frac{\langle 3k+3B-a-b-c\rangle^4\langle a\rangle\langle b\rangle\langle c\rangle}
{\langle 3k+3B-a-b\rangle \langle 3k+3B-a-c\rangle \langle 3k+3B-b-c\rangle}
\nonumber
\\
&
\times
\frac{\langle B-a\rangle\langle B-b\rangle\langle B-c\rangle}
{\langle 3k+2B-a-b\rangle \langle 3k+2B-a-c\rangle \langle 3k+2B-b-c\rangle}
\nonumber
\\  
&
\times
\frac{\langle n-2k+a\rangle \langle n+2k+3B-a\rangle}
     {\langle n+k+3B-b-c\rangle \langle n-k+b+c\rangle}
\frac{\langle n-2k+b\rangle \langle n+2k+3B-b\rangle}
     {\langle n+k+3B-a-c\rangle \langle n-k+a+c\rangle}
\nonumber
\\
&
\times
\frac{\langle n-2k+c\rangle \langle n+2k+3B-c\rangle}
     {\langle n+k+3B-a-b\rangle \langle n-k+a+b\rangle}.
\label{triadform}
\end{align}

\end{theo}

\begin{figure}[h]
\centerline{\hfill \includegraphics[width=0.5\textwidth]{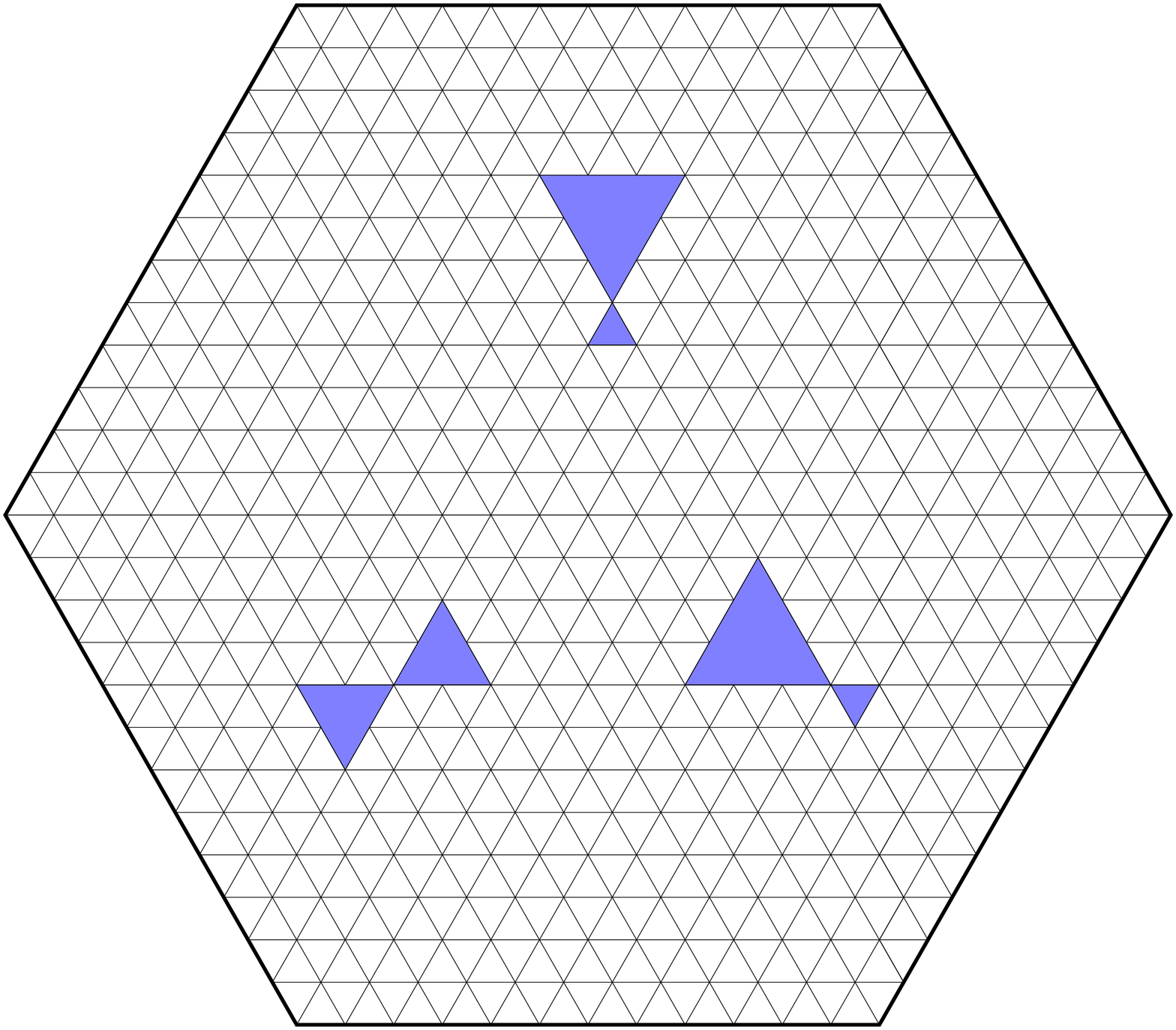} \hfill }
\caption{\label{fba} A triad of bowties at the center of a hexagon.
}
\end{figure}

\begin{rem}
Note that the above result allows in particular to squeeze in completely the outer lobe of any of the three bowties independently, obtaining a triangular satellite of opposite orientation compared to $S_{n,a,b,k}$. This includes the case $a=0$ of the regions $S'_{n,a,b,k}$ shown on the right in Figure~\ref{faa}!
\end{rem}

Asymptotic analysis of the formula for $\M(T_{n,k,B,a,b,c})$ that follows from \eqref{triadform} and Conjecture \ref{tbaa} lets us deduce the following result, the details of whose proof will appear in a subsequent paper.

\begin{theo}
\label{triadcorr}
Consider three bowties $X_1$, $X_2$ and $X_3$ in a triad formation, as shown in Figure $\ref{fba}$. Their outer lobes have sizes $a$, $b$ and $c$, and their inner lobes have sizes $a'$, $b'$ and $c'$, respectively. The distance between the nodes of two bowties is $3k$. Assume  $a+b+c=a'+b'+c'$, and define the correlation $\widetilde\omega(X_1,X_2,X_3)$ using equation \eqref{ecb}.

If Conjecture $\ref{tbaa}$ holds and $a+a'=b+b'=c+c'=B$, writing $\langle n\rangle = G(n+1)$ as before, we have
\begin{align}
&
  \widetilde\omega(X_1,X_2,X_3)\sim  
\frac{3^{B^2/8}}{(2\pi)^{B/2}}\frac{\langle\frac{B}{2}\rangle^2\langle a\rangle \langle a'\rangle}{\langle B\rangle}
\frac{3^{B^2/8}}{(2\pi)^{B/2}}\frac{\langle\frac{B}{2}\rangle^2 \langle b\rangle \langle b'\rangle}{\langle B\rangle}
\frac{3^{B^2/8}}{(2\pi)^{B/2}}\frac{\langle\frac{B}{2} \rangle^2 \langle c\rangle \langle c'\rangle}{\langle B\rangle}
\nonumber
\\[10pt]
&\ \ \ \ \ \ \ \ \ \ \ \ \ \ \ \ \ \ \ \ \ \ \ \ \ \ \ \ \ \ 
\times
(3k)^{-\frac12[(a-a')(b-b')+(a-a')(c-c')+(b-b')(c-c')]}, \ \ \ k\to\infty.
\label{ebo}
\end{align}
\end{theo}

\begin{rem}
As $3k$ is the distance between each pair of bowties, and their charges are $a-a'$, $b-b'$ and $c-c'$, this proves the electrostatic conjecture for a system of three bowties arranged in a triad when $a+a'=b+b'=c+c'$ and $a+b+c=a'+b'+c'$. In fact, the electrostatic conjecture follows even without assuming $a+b+c=a'+b'+c'$. The reason we assumed this condition is because it allows us to compute the multiplicative constant in \eqref{ebo} explicitly. We will use it in Section~\ref{consequences} (see Remark~\ref{turnsout}).
\end{rem}

\section{Consequences for the correlation of holes}
\label{consequences}

In this section we show how Theorem \ref{tbc} can be used to ``bootstrap'' an earlier conjecture of the first author \cite[Conjecture 1]{ov} on the asymptotics of the correlation $\widetilde{\omega}$ of any finite collection $O_1,\dotsc,O_n$ of triangular holes, by specifying explicitly the involved multiplicative constant (see Conjecture \ref{tca} in this section).

To achieve this, we need to discuss some more subtle points involving two other definitions of the correlation of holes, which the first author introduced in \cite{ov}. For convenience we reproduce their definitions below.

\begin{figure}[h]
\centerline{
\hfill
{\includegraphics[width=0.75\textwidth]{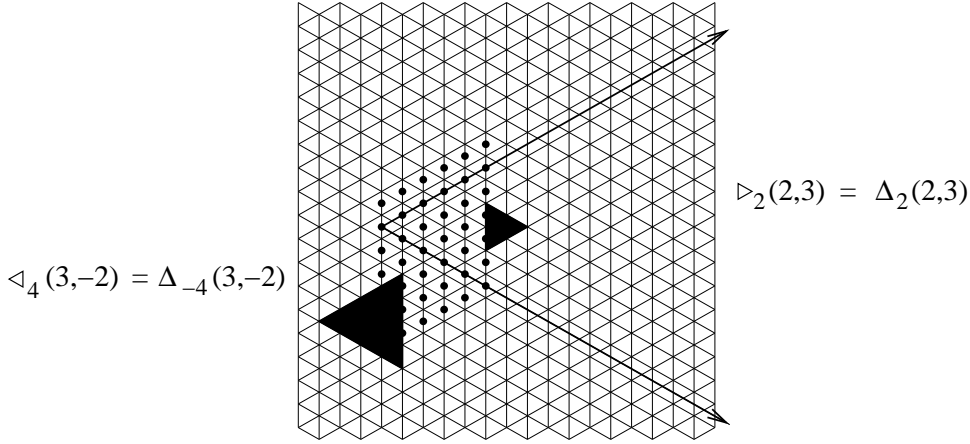}}
\hfill
}
\caption{\label{fca} Marked points in our $60^\circ$ coordinate system; the right $2$-triangular hole $\triangleright_2(2,3)=\triangle_2(2,3)$ and the left $4$-triangular hole $\triangleleft_4(3,-2)=\triangle_{-4}(3,-2)$: see text.}
\end{figure}

Denote the triangular lattice by $\mathcal T$, and draw it (only in this section) so that one family of lattice lines is vertical. Think of the hexagonal lattice $\mathcal H$ as the dual of  $\mathcal T$. Then the vertices of $\mathcal H$ are the unit triangles of $\mathcal T$, and a dimer on $\mathcal H$ is a lozenge. Monomers on $\mathcal H$ are unit triangles of $\mathcal T$; we call them right-monomers and left-monomers according to the direction they point to. Allow  holes in $\mathcal H$ to be arbitrary finite (not necessarily connected) unions of monomers.

Call the midpoints of vertical lattice segments in $\mathcal T$ {\it marked points}, and coordinatize them by 
pairs of integers in a $60^\circ$ coordinate system (see Figure \ref{fca}), by picking one of them to be
the origin, and taking the $x$- and $y$-axes in the polar directions $-\pi/3$ and $\pi/3$, respectively. 
Then each right-monomer is specified by a pair of integer coordinates, and so is each left-monomer.

Define the {\it right $k$-triangular hole} $\triangleright_k(x,y)$ to be the right-pointing triangular hole with a side of length $k$ (the unit being the side-length of a unit triangle) whose topmost marked point (those on its boundary included) has coordinates $(x,y)$; the {\it left $k$-triangular hole} $\triangleleft_k(x,y)$ is defined to be the analogous left-pointing triangular hole. In some instances we will find it convenient to have a unifying notation for these two types of holes. To this end, for $k\in\Z$ we define the {\it $k$-triangular hole} $\triangle_k(x,y)$ by
\begin{equation}
\triangle_k(x,y):= 
\begin{cases}
\triangleright_k(x,y),&{\,\,\text{\rm if $k>0$,}}\\ 
                   \triangleleft_k(x,y),&{\text{\rm if $k<0$}}
\label{eca}
\end{cases}
\end{equation}
(see Figure \ref{fca} for two illustrations).

We define the correlation $\widetilde{\omega}$ of any finite collection $O_1,\dotsc,O_n$ of holes as follows. 

For any positive integer $N$, let $T_N$ be the torus obtained from the rhombus  $\{(x,y): |x|,|y|\leq N-1/2\}$ on $\mathcal T$ by identifying opposite sides. Let the {\it charge} $\ch(O)$ of the hole $O$ be the difference between the number of right- and  left-monomers in $O$. 
By performing a reflection across a vertical lattice line, it suffices to define the correlation when $\sum_{i=1}^n \ch(O_i)\geq0$. 
Define $\widetilde{\omega}$ inductively by:

\parindent4pt
\medskip
$(i)$ If $\sum_{i=1}^n \ch(O_i)=0$, set
\begin{equation}
\widetilde{\omega}(O_1,\dotsc,O_n):=
\lim_{N\to\infty}\frac{\M\left(T_N\setminus O_1\cup\cdots\cup O_n\right)}{M\left(T_N\right)};
\label{ecb}
\end{equation}
$(ii)$ If $\sum_{i=1}^n \ch(O_i)=s>0$, set
\begin{equation}
\widetilde{\omega}(O_1,\dotsc,O_n):=
\lim_{R\to\infty} \frac{R^{s/2} \,\widetilde{\omega}\left(O_1,\dotsc,O_n,\triangleleft_1(R,0)\right)}
{\sqrt{C}},
\label{ecc}
\end{equation}
where the constant $C$ is determined by $\widetilde{\omega}(\triangleright_1(0,0),\triangleleft_1(R,0))\sim C\,R^{-1/2}$, $R\to\infty$.

\parindent12pt
Given a hole $O$ and integers $x$ and $y$, denote by $O(x,y)$ the translation of $O$ under which its topmost (and leftmost, if there are ties) marked point is brought to the point $(x,y)$. In \cite{ov} we presented the following generalization of the Fisher-Stephenson conjecture \cite{FS} on the rotational invariance of the  monomer-monomer correlation (the original Fisher-Stephenson conjecture was recently proved by Dub\'edat \cite{Dub}).


\begin{con} \cite{ov}
\label{tca}
For any hole types $O_1,\dotsc,O_n$ and any distinct pairs of integers $(x_1,y_1)$, $\dotsc$, $(x_n,y_n)$ we have as $R\to\infty$ that
\begin{equation}
\widetilde{\omega}(O_1(Rx_1,Ry_1),\dotsc,O_n(Rx_n,Ry_n))\sim
\prod_{i=1}^n\widetilde{\omega}(O_i)
\prod_{1\leq i<j\leq n}\di((Rx_i,Ry_i),(Rx_j,Ry_j))^{\frac12\ch(O_i)\ch(O_j)},
\label{ecd}
\end{equation}
where $\di$ is the Euclidean distance expressed in units equal to a unit triangle side.

\end{con}

The second correlation we need is denoted by $\overline{\omega}$. It is a variant of $\widetilde{\omega}$, but defined only for those collections of holes whose total charge is even. The correlation $\overline{\omega}$ is defined inductively using $(i)$ above, and the modification of $(ii)$ in which $\triangleleft_1(R,0)$ is replaced by $\triangleleft_2(R,0)$ (note that this causes the constant $C$ to be replaced by the leading coefficient $C'$ in the asymptotics of $\overline{\omega}(\triangleright_2(0,0),\triangleleft_2(R,0))$, $R\to\infty$; it turns out that $\overline{\omega}(\triangleright_2(0,0),\triangleleft_2(R,0))\sim \frac{3}{4\pi^2} R^{-2}$, $R\to\infty$, and therefore $C'=\frac{3}{4\pi^2}$).

The special case $q=1$ of \cite[Proposition 2.2]{ec}, stated in terms of the correlation $\overline{\omega}$ (in \cite{ec}
it is phrased in terms of a variant of $\overline{\omega}$, denoted there by $\hat{\omega}$) implies that for non-negative integers~$s$ we have
\begin{equation}
\overline{\omega}(\triangleright_{2s})=\frac{3^{s^2/2}}{(2\pi)^s}\,[0!\,1!\,\cdots\,(s-1)!]^2.
\label{ece}
\end{equation}
%


Based on physical intuition,
it is expected that $\widetilde{\omega}$ agrees with $\overline{\omega}$, and therefore \eqref{ece} is expected to hold with $\overline{\omega}$ replaced by $\widetilde{\omega}$. If we would also know --- at least conjecturally --- the values of the $\widetilde{\omega}(\triangleright_{2s+1})$'s, then we could write down explicitly the multiplicative constant on the right hand side of \eqref{ecd} in the (quite general) special case when $O_i$ is an arbitrary triangular hole (of even or odd side-length, pointing either to the right or to the left), for $i=1,\dotsc,n$.

Based on the experience with Theorem \ref{tbb} (see Remark 3), we could make the daring guess that~\eqref{ece} holds with $\overline{\omega}$ replaced by $\widetilde{\omega}$ {\it also for odd side triangular holes:} As $0!\,1!\,\cdots\,(s-1)!=G(s+1)$, this leads to guessing that
\begin{equation}
\widetilde{\omega}(\triangleright_{k})=\frac{3^{k^2/8}}{(2\pi)^{k/2}}\left[G\left(\frac{k}{2}+1\right)\right]^2, \ \ \ \text{\rm all $k\geq0$.}
\label{ecf}
\end{equation}
%


As it turns out, this daring guess is strongly supported by Theorem \ref{tbc}, as we explain in this section. We therefore formulate the following strengthening of Conjecture \ref{tca} in the case when the holes are arbitrary triangles.

\begin{con}
\label{tcb} 
For arbitrary integers $k_1,\dotsc,k_n$, and any distinct pairs of integers $(x_1,y_1)$, $\dotsc$, $(x_n,y_n)$, we have
\begin{align}
&
\widetilde{\omega}\left(\triangle_{k_1}(Rx_1,Ry_1),\dotsc,\triangle_{k_n}(Rx_n,Ry_n)\right)
\nonumber
\\
&\ \ \ \ \ \ \ \ \ \ \ 
\sim
\prod_{i=1}^n \frac{3^{k_i^2/8}}{(2\pi)^{|k_i|/2}}\left[G\left(\frac{|k_i|}{2}+1\right)\right]^2
\nonumber
\prod_{1\leq i<j\leq n}\di((Rx_i,Ry_i),(Rx_j,Ry_j))^{\frac12 k_i k_j}, \ \ \ R\to\infty.
\label{ecg}
\end{align}

\end{con}


We now discuss the supporting evidence for equation \eqref{ecf}. We start by rewriting the statement of Theorem \ref{tbc} in terms of the Euclidean distance between the holes, expressed in units equal to the side-length of a unit triangle. In these units, the distance between the core and each satellite is $k\sqrt{3}$, and the distance between each pair of satellites is $3k$. Denoting the core by $S_0$ and the satellites by $S_i=S_i(k)$, $k=1,2,3$, one readily checks that the statement of Theorem \ref{tbc} can be rewritten as
\begin{equation}
\omega(S_0,S_1,S_2,S_3)\sim
\frac{3^{a^2/8}G\left(\frac{a}{2}+1\right)^{\!2}\left[3^{b^2/8}G\left(\frac{b}{2}+1\right)^{\!2}\right]^{3}}
{3^{(a+3b)^2/8}G\left(\frac{a+3b}{2}+1\right)^{\!2}}
\prod_{0\leq i <j\leq 3} \di(S_i,S_j)^{\frac12 \ch(S_i)\ch(S_j)},
\label{ech}
\end{equation}
in the limit as the satellites recede away from the core at the same rate, where $\di$ is the Euclidean distance expressed in units equal to the side-length of a unit triangle.

In order to make our arguments, we will need two assumptions. The first one consists of two special cases of the assumption that the correlations $\omega$ and $\widetilde{\omega}$ are equal up to a multiplicative factor depending only on the shapes and sizes of the holes, and not on their relative positions.

The first special case we need is that when the collection of holes consists of the core and the three satellites, when it is equivalent to the statement in the first part below. The second is described in the second part.

\medskip
\parindent0pt
{\bf Assumption I.} (a) The ratio $\dfrac{\omega(S_0(k),S_1(k),S_2(k),S_3(k))}{\tilde\omega(S_0(k),S_1(k),S_2(k),S_3(k))}$ does not depend on $k$.

\medskip
(b) Let $m$ be a non-negative integer, and let $\mathcal O$ be the collection consisting of one $\triangleright_1$ and $m$ $\triangleright_2$ collinear holes lined up along a horizontal axis, so that the leftmost of them is the $\triangleright_1$. Then $\dfrac{\omega(\mathcal O)}{\tilde\omega(\mathcal O)}$ does not depend on the relative distances between the holes in the collection  $\mathcal O$.



\parindent15pt
\medskip


\medskip
This is a reasonable assumption, as $\omega$ is defined by placing the holes at the very center of the enclosing hexagons (in the fine mesh limit as the lattice spacing approaches zero, the enclosing hexagon approaches a regular hexagon, with the core and the satellites shrinking to its center), and there the entropy is maximal (cf. \cite{CLP}). The denominator in Equation \eqref{ebi} (resp., for part (b), the denominator in equation (2.2) of \cite{sc}) is a natural choice, but other choices would clearly work, so in the very definition of $\omega$ there is a residing and somewhat arbitrary multiplicative constant.

When taking ratios of correlations, the multiplicative constants cancel out. Therefore, under Assumption I(a) we get
\begin{equation}
\frac{\omega(S_0,S_1(k),S_2(k),S_3(k))}{\omega(S_0,S_1(0),S_2(0),S_3(0))}
=
\frac{\widetilde{\omega}(S_0,S_1(k),S_2(k),S_3(k))}{\widetilde{\omega}(S_0,S_1(0),S_2(0),S_3(0))}.
\label{eci}
\end{equation}
By definition \eqref{ebi}, the denominator on the left hand side above is equal to 1. Thus, equation~\eqref{eci} combined with \eqref{ech}  gives
\begin{equation}
\frac{\widetilde{\omega}(S_0,S_1(k),S_2(k),S_3(k))}{\widetilde{\omega}(S_0,S_1(0),S_2(0),S_3(0))}
\sim
\frac{3^{a^2/8}G\left(\frac{a}{2}+1\right)^{\!2}\left[3^{b^2/8}G\left(\frac{b}{2}+1\right)^{\!2}\right]^{3}}
{3^{(a+3b)^2/8}G\left(\frac{a+3b}{2}+1\right)^{\!2}}
\prod_{0\leq i <j\leq 3} \di(S_i,S_j)^{\frac12 \ch(S_i)\ch(S_j)}.
\label{ecj}
\end{equation}
However, due to forced lozenges at the points of contact of $S_1(0)$, $S_2(0)$ and $S_3(0)$ with the core~$S_0$, the denominator on the left hand side above is equal to $\widetilde{\omega}(\triangleright_{a+3b})$. Therefore, if Conjecture \ref{tca} holds, \eqref{ecj} implies that
\begin{align}
\frac{\widetilde{\omega}(\triangleright_a)\widetilde{\omega}(\triangleright_b)^3}{\widetilde{\omega}(\triangleright_{a+3b})}
&
=
\frac{3^{a^2/8}G\left(\frac{a}{2}+1\right)^{\!2}\left[3^{b^2/8}G\left(\frac{b}{2}+1\right)^{\!2}\right]^{3}}
{3^{(a+3b)^2/8}G\left(\frac{a+3b}{2}+1\right)^{\!2}},
\label{eck}
\\
&
=
\frac{\dfrac{3^{a^2/8}}{(2\pi)^{a/2}}\,G\left(\frac{a}{2}+1\right)^{\!2}\left[\dfrac{3^{b^2/8}}{(2\pi)^{b/2}}\,G\left(\frac{b}{2}+1\right)^{\!2}\right]^{3}}
{\dfrac{3^{(a+3b)^2/8}}{(2\pi)^{(a+3b)/2}}\,G\left(\frac{a+3b}{2}+1\right)^{\!2}},\ \ \ \text{for all}\ 0\leq a,b\in\Z,\ a\ \text{even.}
\nonumber
\end{align}
While strictly speaking not implying \eqref{ecf}, the above equation does strikingly support it.

\medskip
In fact, it turns out that equation \eqref{ecf} is implied by Conjecture \ref{tca} and Assumption I, provided we make one additional assumption, which we describe below.

\begin{rem}
\label{turnsout}
It turns out that Theorem \ref{triadcorr}, Conjecture \ref{tca} and equation \eqref{ecf} imply the following generalization of \eqref{ecf}: For a bowtie $X_{a,a'}$ with lobe sizes $a$ and $a'$, its correlation is given by
\begin{equation}
\widetilde\omega(X_{a,a'})=
\frac{3^{(a+a')^2/8}}{(2\pi)^{(a+a')/2}}\frac{G\left(\frac{a+a'}{2}+1\right)^2G(a+1)G(a'+1)}{G(a+a'+1)}.
\label{bowtiecorr}
\end{equation}

In turn, the above equation, when combined with equation (1.4) of \cite{ff}, yields more generally the correlation of the shamrock $S(a,b,c,m)$ (the structure consisting of an up-pointing triangular core of side $m$ and three down-pointing triangular lobes of sides $a$, $b$ and $c$ touching it at the vertices). We obtain
\begin{equation}
\widetilde\omega(S(a,b,c,m))=
\frac{3^{(a+b+c+m)^2/8}}{(2\pi)^{(a+b+c+m)/2}}
\frac{G\left(\frac{a+b+c+m}{2}+1\right)^2G(m+1)^3G(a+1)G(b+1)G(c+1)}{G(a+m+1)G(b+m+1)G(c+m+1)}.
\label{shamrockcorr}
\end{equation}

Similarly, combining equation \eqref{bowtiecorr} above with equation (1.5) of \cite{fv}, we can find the correlation of the fern $F(a_1,\dotsc,a_k)$ (a string of contiguous triangular lobes of sizes $a_1,\dotsc,a_k$ lined up along a lattice line, alternately oriented up and down). With $a=a_1+\cdots+a_k$, $o:=a_1+a_3+\cdots$ and $e=a_2+a_4+\cdots$, we obtain
\begin{equation}
\widetilde\omega(F(a_1,\dotsc,a_k))=
\frac{3^{a^2/8}}{(2\pi)^{a/2}}
\frac{G\left(\frac{a}{2}+1\right)^2G(o+1)G(e+1)}{G(a+1)}s(a_1,\dotsc,a_k)s(a_2,\dotsc,a_k),
\label{ferncorr}
\end{equation}
where
\begin{align}
&
s(b_1,b_2,\dotsc,b_{2l})=s(b_1,b_2,\dotsc,b_{2l-1})=
\frac{\prod_{\text{\rm $1\leq i\leq j\leq 2l-1$, $j-i+1$ odd}}G(b_i+b_{i+1}+\cdots+b_j+1)}{\prod_{\text{\rm $1\leq i\leq 2l-1$, $j-i+1$ even}}G(b_i+b_{i+1}+\cdots+b_j+1)}
\nonumber
\\[5pt]
&\ \ \ \ \ \ \ \ \ \ \ \ \ \ \ \ \ \ \ \ \ \ \ \ \ \ \ \ \ \ \ \ \ \ \ \ \ \ \ \ \ \ \ \ \ \ \ \ 
\times\frac{1}{G(b_1+b_3+\cdots+b_{2l-1}+1)}.
\label{sdef}
\end{align}

Equations \eqref{shamrockcorr} and \eqref{ferncorr} then naturally extend Conjecture \ref{tcb} to arbitrary collections of shamrocks and ferns. The details will be presented in a subsequent paper.
\end{rem}

We point out that part $(i)$ of the definition \eqref{ecb}--\eqref{ecc} of the correlation $\widetilde{\omega}$ is most natural, but in part $(ii)$ a very specific choice was made about how to handle collections of holes of strictly positive total charge: Namely, to repeatedly send to infinity negative charges of unit magnitude\footnote{ Furthermore, in the definition of  $\widetilde{\omega}$ these auxiliary negative unit charges are always sent to infinity along the polar direction $-\pi/3$. This was chosen for technical reasons, to aid the computations. Due to the expected rotational invariance, the obtained values should be independent of the direction. } until the total charge is reduced to zero, so that part $(i)$ can be used. 

Once a decision is made upon how exactly to balance the total charge (e.g. for a collection of holes of total charge $2k>0$, one way to do the balancing --- the way done in the definition of $\widetilde{\omega}$ --- is to repeatedly send a negative monomer $\triangleleft_1$ to infinity $2k$ times; another way --- corresponding to the definition of $\overline{\omega}$ --- is to 
repeatedly send a $\triangleleft_2$ hole to infinity $k$ times), we claim that there is a unique choice for the value of $C$ at the denominator on the right hand side of~\eqref{ecc} that gives a chance for Conjecture~\ref{tca} to hold.

We justify this claim for the two cases of a $\triangleleft_1$ or a $\triangleleft_2$ being sent to infinity (these are the only instances we need in our arguments below; the general case is handled the same way). For the case of a $\triangleleft_1$, the claim follows by considering in $(ii)$ the special case when $n=1$ and $O_1=\triangleright_1$. Indeed, then \eqref{ecc} becomes
\begin{equation}
\widetilde{\omega}(\triangleright_1)=\widetilde{\omega}(\triangleright_1(0,0)):=
\lim_{R\to\infty} \frac{R^{1/2} \,\widetilde{\omega}\left(\triangleright_1(0,0),\triangleleft_1(R,0)\right)}
{\sqrt{C}}.
\label{ecl}
\end{equation}
If we want to end up with a correlation for which Conjecture \ref{tca} holds, then we must have
\begin{equation}
\widetilde{\omega}\left(\triangleright_1(0,0),\triangleleft_1(R,0)\right)\sim
\widetilde{\omega}\left(\triangleright_1(0,0)\right)\widetilde{\omega}\left(\triangleleft_1(R,0)\right) R^{-1/2},\ \ \ R\to\infty.
\label{ecma}
\end{equation}
Clearly, \eqref{ecl} and \eqref{ecma} give (using also $\widetilde{\omega}\left(\triangleleft_1\right)=\widetilde{\omega}\left(\triangleright_1\right)$) that $\widetilde{\omega}(\triangleright_1)=\sqrt{C}$. Combined with~\eqref{ecma}, this gives
\begin{equation}
\widetilde{\omega}\left(\triangleright_1(0,0),\triangleleft_1(R,0)\right)\sim
C R^{-1/2},\ \ \ R\to\infty,
\label{ecm}
\end{equation}
which determines $C$ uniquely, as claimed, to be the value we used in the definition of $\widetilde{\omega}$. The case of $\overline{\omega}\left(\triangleright_2\right)$ is justified the same way, leading to the unique choice of $C'$ used in the definition of $\overline{\omega}$.

Our second assumption is a special case of what we could call self-consistency: That all the different possible ways to balance a given collection of holes in part $(ii)$ of the definition lead to the same value of the correlation, provided the denominator on the right hand side of \eqref{ecc} is always chosen to have the unique value determined by the statement of Conjecture  \ref{tca}. In fact, we only need this for our two correlations $\widetilde{\omega}$ and $\overline{\omega}$, and only for a single triangular hole of side two.

\parindent0pt
\medskip
{\bf Assumption II.} $\overline{\omega}(\triangleright_{2})=\widetilde{\omega}(\triangleright_{2})$.

\parindent15pt
\medskip

There is one more result on the asymptotics of the correlation $\omega$ that we need, which follows from the product formula of \cite[Theorem 1.1]{ppone} by the same reasoning that derived Theorem \ref{tbc} from the formula in Conjecture \ref{tbaa}. Making the same arguments that led to \eqref{eck} (i.e. assuming that Conjecture \ref{tca} and Assumption I(b) hold), we obtain
\begin{equation}
\frac{{\widetilde\omega}(\triangleright_1){\widetilde\omega}(\triangleright_2)^{m}}{\widetilde{\omega}(\triangleright_{2m+1})}
=
\frac{3^{1/8}G\left(\frac{3}{2}\right)^{\!2}\left[3^{1/2}G\left(3\right)^{\!2}\right]^{m}}
{3^{(2m+1)^2/8}G\left(\frac{2m+1}{2}+1\right)^{\!2}},\ \ \ \text{for all}\ m\geq0.
\label{ecn}
\end{equation}
%
%

\medskip
{\it Deducing the value of $\widetilde{\omega}(\triangleright_{1})$.} Consider equation \eqref{eck} (which recall follows from Theorem~\ref{tbc}, provided Conjecture \ref{tca} and Assumption I(a) hold), and set $a=0$ and $b=1$. Using $\widetilde{\omega}\left(\triangleright_0\right)=1$ (which follows from the definition of $\widetilde{\omega}$), the recurrence \eqref{ebj} and the fact that $\Gamma(3/2)=\sqrt{\pi}/2$, we obtain 
\begin{equation}
\frac{\left[\widetilde{\omega}\left(\triangleright_1\right)\right]^3}{\widetilde{\omega}\left(\triangleright_3\right)}
=
\frac{4}{3^{3/4}\pi}G\left(\frac{3}{2}\right)^{\!4}.
\label{eco}
\end{equation}
%
%
%
%
%
%

\begin{figure}[h]
\centerline{
\hfill
{\includegraphics[width=0.4\textwidth]{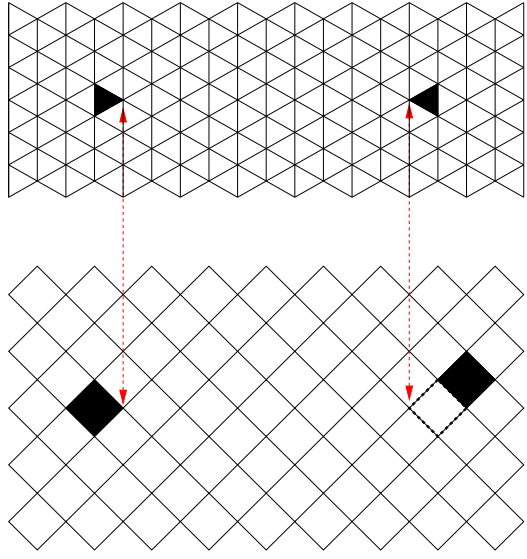}}
\hfill
}
\caption{\label{fcb} Squaring the hexagonal lattice: If the removed unit triangles and the removed unit squares are lined up as shown, having $d$ long lozenge diagonals, respectively $d$ unit square diagonals in between, then their correlations decay to zero asymptotically like $c/\sqrt{d}$ with the same value of $c$ (namely $c=e^{1/2}2^{-5/6}A^{-6}$, $A$ being the Glaisher-Kinkelin constant), where $d$ is the Euclidean distance between the removed monomers, measured on the triangular lattice in units equal to a long lozenge diagonal, and on the square lattice in units equal to a unit square diagonal. Phrased in terms of monomer correlations in a sea of dimers, as the dual of the triangular lattice is the hexagonal lattice (while the square lattice is self-dual), this shows how to calibrate the size of the hexagonal lattice against the square lattice so that the monomer-monomer correlations decay identically.}
\end{figure}

\parindent0pt
On the other hand, setting $m=1$ in \eqref{ecn}, we get
\begin{equation}
\frac{{\tilde\omega}\left(\triangleright_1\right){\tilde\omega}\left(\triangleright_2\right)}
{\widetilde{\omega}\left(\triangleright_3\right)}
=
\frac{4}{3^{1/2}\pi}.
\label{ecp}
\end{equation}
By \eqref{ece} and Assumption II, ${\widetilde\omega}\left(\triangleright_2\right)=\sqrt{3}/(2\pi)$. 
Thus, combining equations \eqref{eco} and \eqref{ecp}, we get 
\begin{equation}
\widetilde{\omega}\left(\triangleright_1\right)=\frac{3^{1/8}}{\sqrt{2\pi}}G\left(\frac{3}{2}\right)^{\!2}. 
\label{ecq}
\end{equation}

\parindent15pt
\medskip
{\it Deducing the values  $\widetilde{\omega}(\triangleright_{2m+1})$.} Having determined the value of $\triangleright_1$, the value of $\widetilde{\omega}(\triangleright_{2m+1})$ for any positive integer $m$ follows directly from \eqref{ecn}, using again that (by \eqref{ece} and Assumption II) $\widetilde{\omega}\left(\triangleright_2\right)=\sqrt{3}/(2\pi)$. This leads to \eqref{ecf}, and thus to the explicit multiplicative constant in Conjecture~\ref{tcb}.

\medskip
We end this section with a pretty astounding way of relating the hexagonal and square lattices from the point of view of the rate of decay to zero of the the monomer-monomer correlation. This is afforded by comparing the value of $\widetilde{\omega}(\triangleright_{1})$ derived above to the analogous constant for the square lattice, which was determined by Hartwig \cite{Hart} in 1966.

Hartwig showed in \cite{Hart} that 
\begin{equation}
\omega(\square_{0,0},\square_{d,d+1})\sim\frac{e^{1/2}}{2^{\tfrac56} A^6}\, d^{-1/2}                  
,\ \ \ d\to\infty,
\label{ecs}
\end{equation}

\parindent0pt
where $\square_{p,q}$ denotes the unit square on the square lattice whose bottom left corner has coordinates $(p,q)$, and the correlation $\omega$ on the square lattice is defined in analogy to \eqref{ebi}, using large squares centered at the origin to enclose the monomers.

\parindent15pt
On the other hand, equation \eqref{ecm}, together with $\widetilde{\omega}(\triangleright_{1})=\sqrt{C}$ and the value for $\widetilde{\omega}(\triangleright_{1})$ derived in \eqref{ecq}, gives
\begin{equation}
\widetilde{\omega}\left(\triangleright_1(0,0),\triangleleft_1(d,0)\right)\sim
\frac{3^{1/4}}{2\pi}\,G\!\left(\frac{3}{2}\right)^{\!4}
d^{-1/2},\ \ \ d\to\infty.
\label{ect}
\end{equation}
By the recurrence \eqref{ebj} and the value \eqref{ebm} of $G(1/2)$, we have
\begin{equation}
G\left(\frac32\right)=\frac{2^{1/24}e^{1/8}\pi^{1/4}}{A^{3/2}},
\label{ebu}
\end{equation}
and \eqref{ect} becomes 
\begin{equation}
\widetilde{\omega}\left(\triangleright_1(0,0),\triangleleft_1(d,0)\right)\sim
\frac{3^{1/4}e^{1/2}}{2^{5/6}A^6}\,
d^{-1/2},\ \ \ d\to\infty.
\label{ecv}
\end{equation}

By Conjecture \ref{tca} we should have
\begin{equation}
\widetilde{\omega}\left(\triangleright_1(0,0),\triangleleft_1(d,d)\right)\sim
\frac{3^{1/4}e^{1/2}}{2^{5/6}A^6}\,
(d\sqrt{3})^{-1/2}
=\frac{e^{1/2}}{2^{5/6}A^6}\,d^{-1/2}
,\ \ \ d\to\infty,
\label{ecw}
\end{equation}
because the Euclidean distance between $\triangleright_1(0,0)$ and $\triangleleft_1(d,d)$ is $d\sqrt{3}$, expressed in units equal to a unit triangle side.

The agreement of the multiplicative constants in \eqref{ecs} and \eqref{ecw} is most unexpected. 
Note that in \eqref{ecs} the distance between the removed unit squares is $d$ unit square diagonals, and the distance between the removed unit triangles in \eqref{ecw} is $d$ long lozenge diagonals. Therefore, the agreement of the right hand sides in \eqref{ecs} and \eqref{ecw} has the following interpretation: If the triangular lattice is scaled so that the lengths of a long lozenge diagonal matches the length of a unit square diagonal on the square lattice (see Figure \ref{fcb}), then the monomer-monomer correlations on these two lattices decay to zero at precisely the same rate. Since unit holes in  lozenge tilings are equivalent to monomers in dimer systems on the hexagonal lattice, we can view this agreement as specifying how to scale the hexagonal lattice against the square lattice in order to get precisely the same decay --- squaring the hexagonal lattice, as it were.

%
%
%

%
%
%
%
%
%


\section{Determinantal formula for $\m(S_{n,a,b,k})$} 
\label{detsec}

The purpose of this section is to derive a convenient determinantal formula for $\m(S_{n,a,b,k})$. 
This derivation is divided into the following steps according to the four subsections.
\begin{enumerate}
\item First we use the Lindstr\"om-Gessel-Viennot theorem \cite{Lindstr,GesselViennot} to derive a determinantal formula for the number of lozenge tilings of $S_{n,a,b,k}$ \emph{assuming that $b$ is even}. This is standard, however, we introduce a notation that will be useful in the following. Also, for what follows, we need a more general setting, where the sizes of the three satellites are independent integers $b_1,b_2,b_3$.
\item Next we show that the number of lozenge tilings in this more general setting is for each 
$i \in \{1,2,3\}$ a polynomial in $b_i$ when fixing the other $b_j$'s, $n$, $a$, and $k$. Here we employ arguments that have been used in \cite[Section~6]{CEKZ}.
\item Then we use this polynomiality to modify the determinantal formula from the first step so that it gives the correct values also if $b$ is odd.
\item Finally, we modify the determinant further such that it reveals the polynomiality in $a$ (so that $a$ does not appear in the size of the matrix and all matrix entries are polynomials in $a$). This is necessary to be able to apply the identification of factors method, see \cite[Sec. 2.4]{KrattDet}.
\end{enumerate}

\subsection{Trapezoids with triangular holes}

For positive integers $n,l$, we refer to the isosceles trapezoid whose longer base is of length $l$, whose legs are of length $n$ and with lower base angles $60^\circ$ as an $(n,l)$-trapezoid. The $(11,16)$-trapezoid is given in Figure~\ref{exampletrapez}. If we draw such a trapezoid on the triangular lattice in the usual way so that the longer base is horizontal and below the shorter base, and the vertices are lattice points, then the trapezoid has $n$ more up-pointing unit triangles than down-pointing unit triangles. Hence such a trapezoid does not have a lozenge tiling, but may have one if we remove $n$ up-pointing unit triangles from it. 

As indicated in Figure~\ref{exampletrapez}, such lozenge tilings correspond to families of non-intersecting lattice paths where the starting points are arranged along the left leg of the trapezoid, while the end points are situated at the centers of the $/$-sides of the removed triangles (which are the black triangles in our example). If we number the starting points from bottom to top with $1$ to $n$ and fix also a numbering of the removed triangles from $1$ to $n$, such a family of non-intersecting lattice paths induces in a natural way a permutation of $1,2,\ldots,n$. The sign of this permutation is said to be the \emph{sign }of the lozenge tiling. In our example, numbering the removed triangles from bottom to top and within a row from left to right gives the permutation $1 \, 2 \, 3 \, 4 \, 6 \, 7 \, 8 \, 9 \, 5 \, 10 \, 11$. 

We say that the set of removed triangles is \emph{even} if each such triangle that is not situated in the bottom row is contained in a maximal (connected) horizontal chain of removed triangles that is of even length. The set of removed triangles in Figure~\ref{exampletrapez} is even. Under this assumption, all lozenge tilings have the same sign. Finally note that removing a horizontal chain of, say, $n$ unit triangles is equivalent (due to forced lozenges) to removing an up-pointing triangle of size $n$.

\begin{figure}
\scalebox{0.4}{\includegraphics{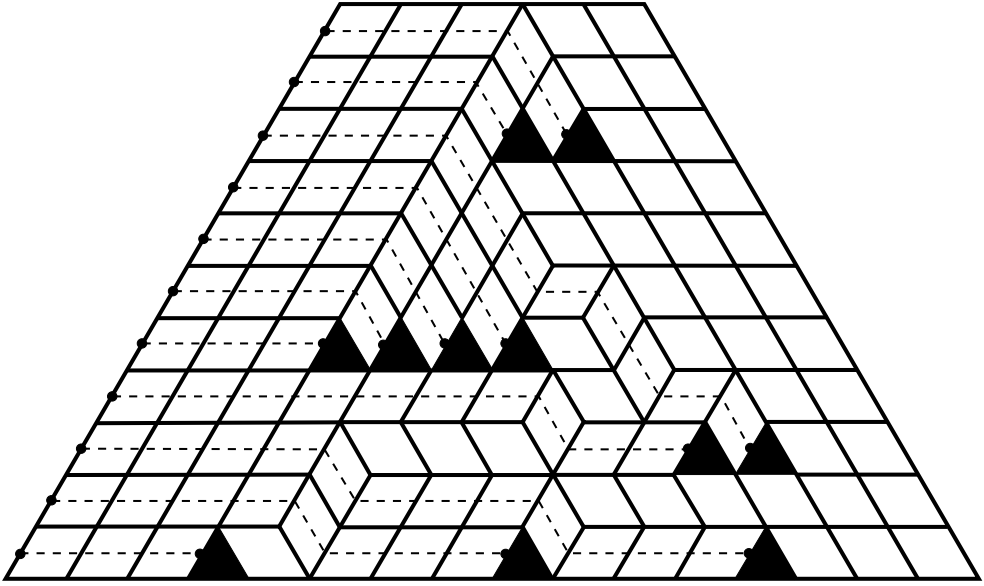}}
\caption{\label{exampletrapez} A lozenge tiling of an $(11,16)$-trapezoid along with the corresponding family of non-intersecting lattice paths.}
\end{figure}

The following lemma, which follows immediately from the Lindstr\"om-Gessel-Viennot theorem, allows us to compute the number of lozenge tilings of a trapezoid with an even set of up-pointing unit triangles removed. In order to formulate it, we need the forward difference operator $\fd_x$ which is defined as 
$$\fd_x p(x) = p(x+1)-p(x).$$ Moreover, we set 
$$
\binom{n}{k} = \begin{cases} \frac{n(n-1) \cdots (n-k+1)}{k!} & k \ge 0 \\ 0 & \text{otherwise} \end{cases}.
$$

\begin{lem} 
\label{gelfand} Consider an $(n,l)$-trapezoid with $n$ up-pointing unit triangles $R_1,R_2,\ldots,R_n$ removed. For each $i$, let $r_i$ be the row of $R_i$, counted from the bottom starting with $1$, and $c_i$ be the position of $R_i$ in its row, counted from the left starting with $1$. Then the signed enumeration of lozenge tilings\footnote{Each tiling being counted with a sign equal to the sign of the permutation induced by the paths of lozenges encoding the tiling; see Figure \ref{exampletrapez}.} of the $(n,l)$-trapezoid
where the triangles $R_1,R_2,\ldots,R_n$ have been removed is 
\begin{equation}
\label{operator}
\prod_{i=1}^{n} \fd_{c_i}^{r_i-1} \prod_{1 \le i < j \le n} \frac{c_j-c_i}{j-i} = \prod_{i=1}^{n} \fd_{c_i}^{r_i-1} \det_{1 \le i, j \le n} \left( \binom{c_i-d}{j-1} \right), 
\end{equation} 
for any $d$. 
If $R_1,\ldots,R_n$ is even, then the absolute value of this expression is the number of lozenge tilings.
\end{lem}

\begin{proof} We use the bijection between lozenge tilings and families of non-intersecting lattice paths as indicated in Figure~\ref{exampletrapez}.
The starting points of the lattice paths on the left leg of the trapezoid can be parametrized by 
$(1,1),(2,2),\ldots,(n,n)$, from bottom to top, while the endpoint on the $/$-side of $R_i$ is then  $(c_i+r_i,r_i)$, and we allow unit steps $(1,0)$ and $(0,-1)$ in our paths. In general, the number of such lattice paths from $(a,b)$ to $(c,d)$ is $\binom{c-a+b-d}{b-d}$, so the number of paths from $(j,j)$ to $(c_i+r_i,r_i)$ is 
$$
\binom{c_i}{j-r_i} = \fd_{c_i}^{r_i-1} \binom{c_i}{j-1}.
$$
By the Lindstr\"om-Gessel-Viennot theorem, the signed enumeration is equal to 
$$
\det_{1 \le i, j \le n} \left( \fd_{c_i}^{r_i-1} \binom{c_i}{j-1} \right) = 
\prod_{i=1}^{n} \fd_{c_i}^{r_i-1} \det_{1 \le i, j \le n} \left( \binom{c_i}{j-1} \right),
$$
where we have used the linearity in the rows of the determinant to show the equality of the expressions.
The result is now a consequence of the following: Suppose $p_j(c)$ is a sequence of monic polynomials for $j=1,\ldots,n$ with 
$\deg_{c} p_j(c) = j-1$. Then, by elementary column operations,
$$
\det_{1 \le i, j \le n} \left( p_j(c_i) \right)=\prod_{1 \le i < j \le n} (c_j-c_i), 
$$
and the assertion now follows by choosing $p_j(c) = (j-1)! \binom{c-d}{j-1}$.
\end{proof}

As it was used in the proof, we may apply the powers of the forward difference operators also ``inside'' the determinant in \eqref{operator} (by the linearity of the determinant in the rows). That way we obtain a determinant in which each row corresponds to a removed triangle. Horizontal (connected) chains of removed triangles then correspond to sets of consecutive rows (if the numbering of the removed triangles was chosen accordingly) in the matrix; these are referred to as \emph{blocks} in the following. The parameter $d$ will play a crucial role and it is the reason why we write the formula in \eqref{operator} in this particular form: We will see that, for any such block, we can choose $d$ appropriately in such a way that this block can be ``eliminated''. We will find it useful to eliminate certain parameters (typically the length of a chain of removed unit triangles) from the matrices underlying our determinants, as this will help us obtain expressions that are polynomials in these parameters. It is this somewhat simple observation that is applied in the following repeatedly to derive two useful formulas for our concrete problem.
  
Next we apply this lemma to our setting. However, in order to be able to extend  
the determinantal formula for even $b$ to odd $b$, we need to work with 
satellites of independent sizes. It is not more difficult to consider a multivariate generalization, where also the three sides --- originally of length $n$ --- are allowed to have independent lengths, and so are the distances between the core and the satellites.

For non-negative integers $n_1,n_2,n_3,b_1,b_2,b_3,k_1,k_2,k_3$ and non-negative even $a$, we denote the hexagon with side lengths $n_1+a+b_1+b_2+b_3,n_3,n_2+a+b_1+b_2+b_3,n_1,n_3+a+b_1+b_2+b_3,n_2$ (clockwise from the northwestern side) that has four triangular holes with side lengths $a,b_1,b_2,b_3$, respectively, as indicated in Figure~\ref{example0} by $S_{n_1,n_2,n_3,a,b_1,b_2,b_3,k_1,k_2,k_3}$: The hole of size $a$ ({the} \emph{core}) has distance $(n_1+b_1) \cdot \frac{\sqrt{3}}{2}$, $(n_2+b_2) \cdot \frac{\sqrt{3}}{2}$, $(n_3+b_3) \cdot \frac{\sqrt{3}}{2}$ from the three sides of length $n_1+a+b_1+b_2+b_3$, $n_2+a+b_1+b_2+b_3$, $n_3+a+b_1+b_2+b_3$, respectively. The three holes of size $b_1,b_2,b_3$ ({the} \emph{satellites}) 
point towards the center of the core and have distance $2 k_1 \cdot \frac{\sqrt{3}}{2} , 2 k_2 \cdot \frac{\sqrt{3}}{2}, 2 k_3 \cdot \frac{\sqrt{3}}{2}$ from the core, respectively, where the satellite of size $b_i$ is situated between the core and the long side of the hexagon that has distance $n_i+b_i$ from the core. 

Note that the geometry of the configuration implies 
\begin{equation}
\label{geometry}
n_1 \le a + n_2 + b_2 + n_3 + b_3.
\end{equation}
This can be seen as follows: Consider the line that includes the  ``$/$''-side of the core. The length of the portion of this line included in the wedge obtained by extending the sides of length $n_2+a+b_1+b_2+b_3$ and $n_3+a+b_1+b_2+b_3$ until they meet, which is $a+n_2+b_2+n_3+b_3$, needs to be at least as large as the length of the southeastern edge of the hexagon, which is $n_1$. By symmetry, we also have $n_2 \le a + n_1+b_1+n_3+b_3$ and $n_3 \le a + n_1 + b_1 + n_2 + b_2$. 

\begin{figure}
\scalebox{0.5}{
\psfrag{a}{\Huge$a$}
\psfrag{b1}{\Huge$b_1$}
\psfrag{b2}{\Huge$b_2$}
\psfrag{b3}{\Huge$b_3$}
\psfrag{n1+b1}{\Huge$n_1+b_1$}
\psfrag{n2+b2}{\Huge$n_2+b_2$}
\psfrag{n3+b3}{\Huge$n_3+b_3$}
\psfrag{2k1}{\Huge$k_1$}
\psfrag{2k2}{\Huge$k_2$}
\psfrag{2k3}{\Huge$k_3$}
\psfrag{n1}{\Huge$n_1$}
\psfrag{n2}{\Huge$n_2$}
\psfrag{n3}{\Huge$n_3$}
\psfrag{n1+a+b1+b2+b3}{\Huge$n_1+a+b_1+b_2+b_3$}
\psfrag{n2+a+b1+b2+b3}{\Huge$n_2+a+b_1+b_2+b_3$}
\psfrag{n3+a+b1+b2+b3}{\Huge$n_3+a+b_1+b_2+b_3$}
\includegraphics{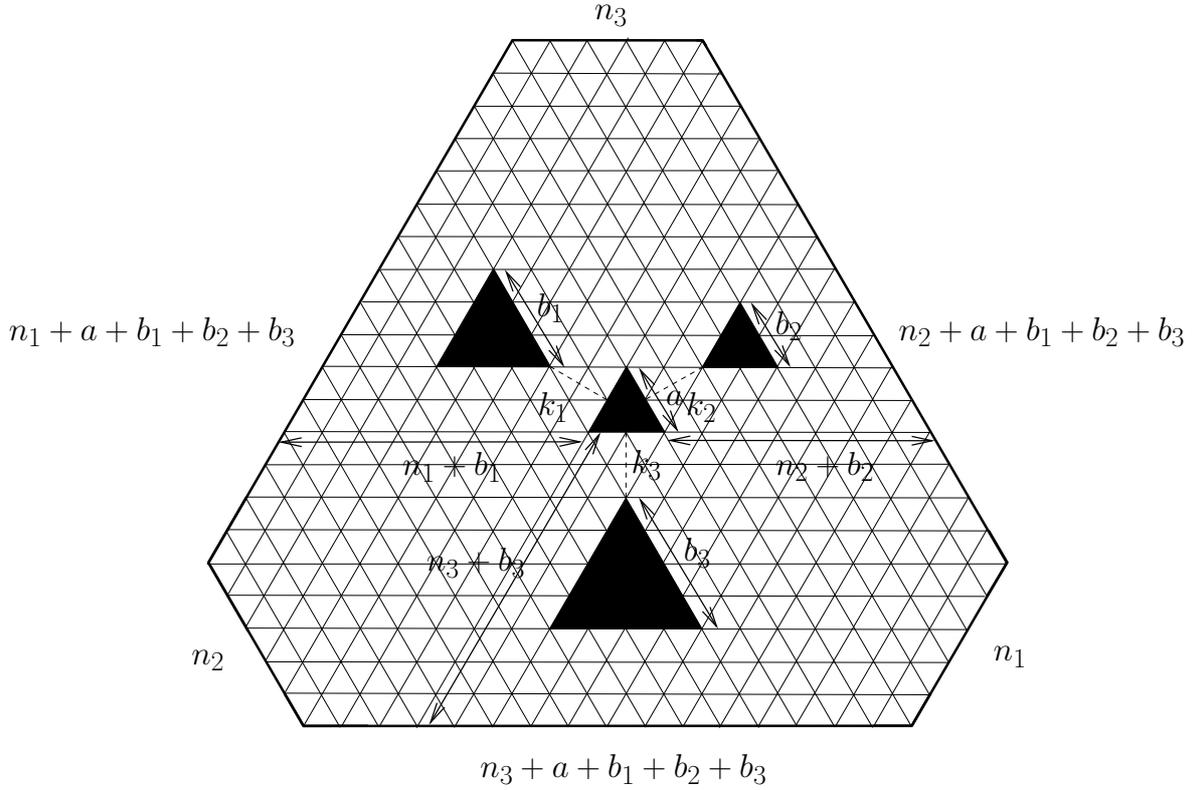}} 
\caption{\label{example0} Independent satellites}
\end{figure}

\begin{figure}
\scalebox{0.4}{\includegraphics{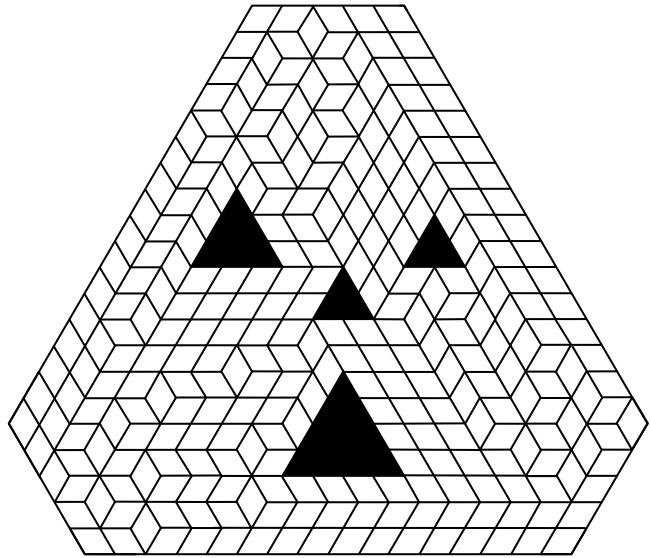}}
\caption{\label{example} An example of a lozenge tiling of $S_{5,5,5,2,3,2,4,1,1,1}$}
\end{figure}

If $b_1,b_2,b_3$ are even, Lemma~\ref{gelfand} can be applied to compute the number of lozenge tilings of this region. Indeed, set 
$n=n_1+n_2+a+b_1+b_2+b_3$. In order to start from an $(n,n+n_3)$-trapezoid, we add a triangle of size $n_2$ at the bottom left corner of the hexagon, while we add a triangle of size $n_1$ at the bottom right corner. We have six chains of triangles to be removed as follows.
\begin{enumerate}
\item At height $0$\footnote{The height of a removed triangle is one less than its row number.} of length $n_2$ in positions $1,\ldots,n_2$.
\item At height $0$ of length $n_1$ in positions $n_2+n_3+a+b_1+b_2+b_3+1,\ldots,n_1+n_2+n_3+a+b_1+b_2+b_3$.
\item At height $n_3-2 k_3$ of length $b_3$ in positions $n_1+\frac{a}{2}+b_1+k_3+1,\ldots,n_1+\frac{a}{2}+b_1+b_3+k_3$.
\item At height $n_3+b_3$ of length $a$ in positions $n_1+b_1+1,\ldots,n_1+a+b_1$.
\item At height $n_3+\frac{a}{2}+b_3+k_1$ of length $b_1$ in positions $n_1-2 k_1+1,\ldots,n_1 - 2 k_1+b_1$.
\item At height $n_3+\frac{a}{2}+b_3+k_2$ of length $b_2$ in positions $n_1+\frac{a}{2}+b_1+k_2+1,\ldots,n_1+\frac{a}{2}+b_1+b_2+k_2$.
\end{enumerate}

Using Lemma~\ref{gelfand}, it follows that the number of lozenge tilings of $S_{n_1,n_2,n_3,a,b_1,b_2,b_3,k_1,k_2,k_3}$ is 
\begin{equation}
\label{op}
\renewcommand\arraystretch{1.3}
\left|
\prod_{i=1}^{b_3} \fd_{c_{3,i}}^{n_3-2 k_3} \prod_{i=1}^{a} \fd_{c_{4,i}}^{n_3+b_3} \prod_{i=1}^{b_1} \fd_{c_{5,i}}^{n_3+\frac{a}{2}+b_3+k_1} \prod_{i=1}^{b_2} \fd_{c_{6,i}}^{n_3+\frac{a}{2}+b_3+k_2} 
\det \left( 
\begin{array}{cc}  
\binom{c_{1,i}-d}{j-1} &  \quad {1 \le i \le n_2} \\   
\binom{c_{2,i}-d}{j-1} & \quad {1 \le i \le n_1}  \\ 
\binom{c_{3,i}-d}{j-1} &  \quad {1 \le i \le b_3} \\ 
\binom{c_{4,i}-d}{j-1} & \quad {1 \le i \le a} \\ 
\binom{c_{5,i}-d}{j-1} & \quad {1 \le i \le b_1} \\ 
\binom{c_{6,i}-d}{j-1} & \quad {1 \le i \le b_2} 
\end{array} 
\right)_{1 \le j \le n}\right|,
\end{equation}
where $c_{1,i}=i$, $c_{2,i}=n_2+n_3+a+b_1+b_2+b_3+i$, $c_{3,i}=n_1+\frac{a}{2}+b_1+k_3+i$, $c_{4,i}=n_1+b_1+i$, $c_{5,i}=n_1 -2 k_1+i$ and $c_{6,i}=n_1+\frac{a}{2}+b_1+k_2+i$.  We obtain the following result.

\begin{prop}
For even $b_1,b_2,b_3$, we have 
\begin{equation}
\m(S_{n_1,n_2,n_3,a,b_1,b_2,b_3,k_1,k_2,k_3})= \left|
\label{bigmatrix}
\det \left( \begin{matrix*}[l]  \binom{i-d}{j-1} & {1 \le i \le n_2}  \vspace{2mm} \\ \binom{n_2+n_3+a+b_1+b_2+b_3+i-d}{j-1} & {1 \le i \le n_1}  \vspace{2mm} \\ \binom{n_1+\frac{a}{2}+b_1+k_3+i-d}{j-1-n_3+2 k_3} & {1 \le i \le b_3} \vspace{2mm} \\ \binom{n_1+b_1+i-d}{j-1-n_3-b_3} & {1 \le i \le a} \vspace{2mm} \\ \binom{n_1 -2 k_1+i-d}{j-1-n_3-\frac{a}{2}-b_3-k_1} & {1 \le i \le b_1} \vspace{2mm} \\ \binom{n_1+\frac{a}{2}+b_1+k_2+i-d}{j-1-n_3-\frac{a}{2}-b_3-k_2} & {1 \le i \le b_2} \end{matrix*} \right)_{1 \le j \le n} \right|, 
\end{equation}
where $d$ can be chosen arbitrarily.
\end{prop}

\subsection{Polynomiality in the sizes of the satellites}

The technique we are using to deal with odd-sized satellites is based on the following crucial observation.

\begin{lem} 
\label{indeplem}
For any $i \in \{1,2,3\}$, the quantity $\m(S_{n_1,n_2,n_3,a,b_1,b_2,b_3,k_1,k_2,k_3})$ is a polynomial in $b_i$ when fixing the $n_i$'s, the $k_i$'s, the core size $a$, and the two $b_j$'s with $j \not=i$.
\end{lem}

\begin{proof} We follow the ideas provided in \cite[Section~6]{CEKZ}, which were used there to show the polynomiality of $\m(S_{n,n,n,a,0,0,0,0,0,0})$ in $a$. By symmetry, it suffices to consider the case $i=2$.

Set $S=S_{n_1,n_2,n_3,a,b_1,b_2,b_3,k_1,k_2,k_3}$. Let $R$ be the smallest lattice hexagon that contains the southwestern side of $S$, the core, and the satellites of side-lengths $b_1$ and $b_3$. For the region in Figure \ref{example}, the resulting region $R$ is shown on the lower left in Figure \ref{example1} (delimited by the dashed line).

\begin{figure}[h]
\centerline{
\hfill
{\includegraphics[width=0.60\textwidth]{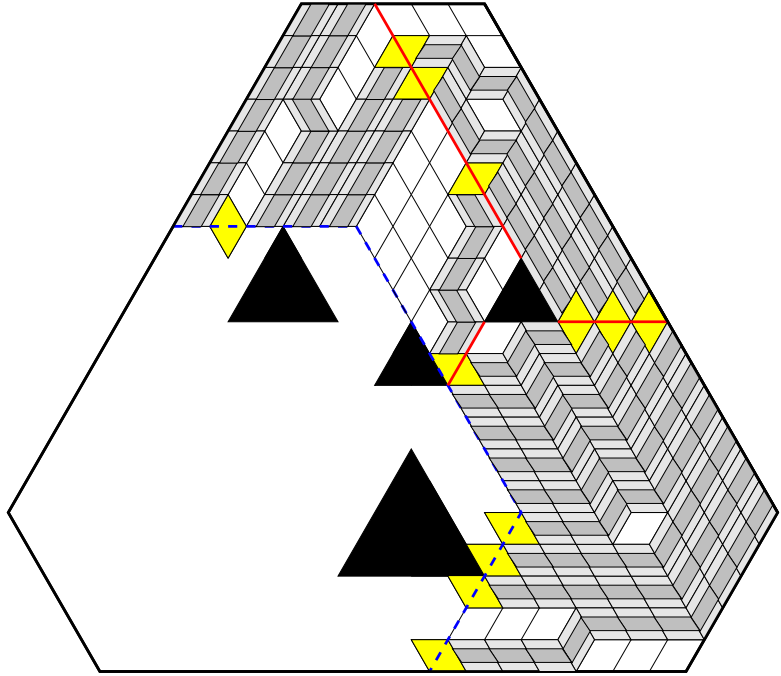}}
\hfill
}
\caption{\label{example1} Proving polynomiality in $b_2$: The region $R$ on the lower left delimited by the dashed line does not change as $b_2$ varies; the lozenges that straddle the dashed line in one of its fixed number (independent of $b_2$) of tilings are indicated. The number of ways each such tiling of $R$ can be extended to a tiling of $S_{n_1,n_2,n_3,a,b_1,b_2,b_3,k_1,k_2,k_3}$ is polynomial in $b_2$ for fixed $n_1,n_2,n_3,a,b_1,b_3,k_1,k_2,k_3$, by the argument in \cite{CEKZ}.}
\end{figure}

As $b_2$ varies over the non-negative integers (with all the other parameters having fixed values), the region $R$ does not change. In particular, the number of lozenge tilings of $R$ in which the dashed portion of the boundary is treated as free (i.e. lozenges are allowed to protrude halfway across them) is a fixed number, independent of $b_2$.

One instance of such a tiling (for focus, only the lozenges straddling the dashed line) is shown in Figure \ref{example1}. By the observation in the previous paragraph, it suffices to show that for any such fixed choice of lozenges straddling the dashed line, the complement $R'$ of $R$ in $S$ --- which {\it does} change as $b_2$ varies --- has a number of tilings that is a polynomial in $b_2$.

This follows by the very same arguments we used in Section 6 of \cite{CEKZ}. Indeed, extend rays from the satellite of side $b_2$ as indicated by the thick solid lines in Figure \ref{example1}. Depending on the actual values of the fixed parameters $n_1,n_2,n_3,a,b_1,b_3,k_1,k_2,k_3$, the ray going southwest may intersect either the NE or SE side of $R$. Similarly, the ray going east intersects either the NE or the SE side of $S$, and the ray going northwest either the N or the NW side of $S$. Figure \ref{example1} shows one of these possibilities. We prove polynomiality of $\m(R')$ in $b_2$ in this case; the others follow the same way.

As we did in \cite{CEKZ}, we partition the lozenge tilings of $R'$ that have a fixed set $\mathcal{L}$ of lozenges straddling the dashed line according to the sets of lozenges $\mathcal{L}_1$, $\mathcal{L}_2$ and $\mathcal{L}_3$ that straddle each of the three rays. Of key importance is the fact that the length each of these rays runs in $R'$ has a fixed value, independent of $b_2$. It therefore suffices to show that for each fixed choice of the position of the lozenges in $\mathcal{L}$, $\mathcal{L}_1$, $\mathcal{L}_2$ and $\mathcal{L}_3$, the number of lozenge tilings of $R'$ that contain these lozenges is a polynomial in $b_2$.

Clearly, this number is the product of the number of corresponding tilings of the three regions that the rays divide $R'$ into. For each of these three regions, encode their tilings as families of paths of lozenges (equivalently, lattice paths on $\Z^2$) as indicated in Figure \ref{example1}. Then by the Lindstr\"om-Gessel-Viennot theorem \cite{Lindstr,GesselViennot}, the number of tilings of each of these three regions is equal to a determinant whose order is independent of $b_2$, and all of whose entries are --- as can be easily checked --- either independent of $b_2$, or of the form $b_2+c\choose d$, with $c$ and $d$ independent of $b_2$. This implies that each of them is a polynomial in $b_2$, and the proof is complete. \end{proof}



\subsection{A determinantal formula for general $b_1,b_2,b_3$ assuming $k_2=k_3$}

The goal of this section is to derive the following determinantal formula for $\m(S_{n_1,n_2,n_3,a,b_1,b_2,b_3,k_1,k_2,k_2})$ that holds for general $b_1,b_2,b_3$. 
Note that we assume $k_2=k_3$ because the situation is simpler then, but the procedure can be adapted so that it works also if $k_2 \not= k_3$. In this formula, we use the convention 
$$
\sum_{i=a}^{b} p(i) = - \sum_{i=b+1}^{a-1} p(i) 
$$
if $b<a$. Note that this implies $\sum\limits_{i=a}^{a-1} p(i) =0$.


\begin{theo} 
\label{oddb}
For all non-negative integers $n_1,n_2,n_3,b_1,b_2,b_3,k_1,k_2$ and even $a$, we have 
\begin{multline*}
\renewcommand\arraystretch{1.5}
\m(S_{n_1,n_2,n_3,a,b_1,b_2,b_3,k_1,k_2,k_2}) \\   =
\left| \det \left( \begin{array}{ll} (-1)^{[j \ge n_3-2 k_2+1] b_3} \left( \binom{i-n_1-\frac{a}{2}-b_1-k_2-1}{j-1} \right. \\   \left. \quad + ((-1)^{b_1}-1)  \sum\limits_{q=n_3+\frac{a}{2}+b_3+k_1+1}^{j}  \binom{i-n_1+2 k_1-1}{q-1} \binom{-2 k_1 -\frac{a}{2}-b_1-k_2}{j-q} \right)  & {1 \le i \le n_2}   \\ \hline  (-1)^{[j \ge n_3+\frac{a}{2}+b_3+k_2+1] b_2} \binom{-n_1 + n_2+n_3+\frac{a}{2}+b_2+b_3-k_2 + i-1}{j-1} & {1 \le i \le n_1 }  \\  \hline \binom{i-1}{j-1-n_3+2 k_2} & {1 \le i \le b_3 }  \\ \hline  \binom{i-\frac{a}{2}-k_2-1)}{j-1-n_3-b_3} &  {1 \le i \le a }  \\  \hline \binom{-\frac{a}{2}-b_1 -2 k_1-k_2 +i-1}{j-1-n_3-\frac{a}{2}-b_3-k_1} & {1 \le i \le b_1}  \\  \hline \binom{i-1}{j-1-n_3-\frac{a}{2}-b_3-k_2} & {1 \le i \le b_2} \end{array} \right)_{1 \le j \le n} \right|, 
\end{multline*}
where we use the Iverson bracket which is defined as 
$$
[\text{statement}] = \begin{cases} 1 & \text{if the statement is true} \\ 0 & \text{otherwise} \end{cases}.
$$
\end{theo}

\begin{proof}
We set $d=d_1=n_1-2 k_1+1$ in the determinant in \eqref{bigmatrix}. Then all entries in the first row of the fifth block (which is the one corresponding to the satellite of size $b_1$) are zero except for the one in column $j=1+n_3+\frac{a}{2}+b_3+k_1$. We expand with respect to this row. The new top row of the fifth block has again only a non-zero entry in column $j=1+n_3+\frac{a}{2}+b_3+k_1$, and so we expand with respect to this row now. We keep doing this until the fifth block has vanished and obtain 
$$
(-1)^{(n_1+n_2+n_3+\frac{3a}{2}+k_1)b_1} 
\det \left( \begin{matrix*}[l]  \binom{i-d_1}{j-1} & {1 \le i \le n_2}  \vspace{2mm} \\ \binom{n_2+n_3+a+b_1+b_2+b_3+i-d_1}{j-1} & {1 \le i \le n_1}  \vspace{2mm} \\ \binom{n_1+\frac{a}{2}+b_1+k_3+i-d_1}{j-1-n_3+2 k_2} & {1 \le i \le b_3} \vspace{2mm} \\ \binom{n_1+b_1+i-d_1}{j-1-n_3-b_3} & {1 \le i \le a} \vspace{2mm}  \\ \binom{n_1+\frac{a}{2}+b_1+k_2+i-d_1}{j-1-n_3-\frac{a}{2}-b_3-k_2} & {1 \le i \le b_2} \end{matrix*} \right)_{1 \le j \le n_3+\frac{a}{2}+b_3+k_1 \atop n_3+\frac{a}{2}+b_1+b_3+k_1+1 \le j \le n}.
$$
This can also be written as follows (omitting now the ranges for the rows $i$).
\begin{multline*}
\renewcommand\arraystretch{1.3}
(-1)^{(n_1+n_2+n_3+\frac{3a}{2}+k_1)b_1} \\ 
\times 
\det \left( \begin{array}{l|l}  \binom{i-d_1}{j-1}_{1 \le j \le n_3+\frac{a}{2}+b_3+k_1}  & \binom{i-d_1}{j+n_3+\frac{a}{2}+b_1+b_3+k_1-1}_{1 \le j \le n_1+n_2-n_3+\frac{a}{2}+b_2-k_1}   \\ \hline \binom{n_2+n_3+a+b_1+b_2+b_3+i-d_1}{j-1}_{1 \le j \le n_3+\frac{a}{2}+b_3+k_1} & \binom{n_2+n_3+a+b_1+b_2+b_3+i-d_1}{j+n_3+\frac{a}{2}+b_1+b_3+k_1-1}_{1 \le j \le n_1+n_2-n_3+\frac{a}{2}+b_2-k_1}   \\ \hline \binom{n_1+\frac{a}{2}+b_1+k_3+i-d_1}{j-1-n_3+2 k_2}_{1 \le j \le n_3+\frac{a}{2}+b_3+k_1} & \binom{n_1+\frac{a}{2}+b_1+k_3+i-d_1}{j+n_3+\frac{a}{2}+b_1+b_3+k_1-1-n_3+2 k_2}_{1 \le j \le n_1+n_2-n_3+\frac{a}{2}+b_2-k_1}  \\ \hline \binom{n_1+b_1+i-d_1}{j-1-n_3-b_3}_{1 \le j \le n_3+\frac{a}{2}+b_3+k_1} & \binom{n_1+b_1+i-d_1}{j+n_3+\frac{a}{2}+b_1+b_3+k_1-1-n_3-b_3}_{1 \le j \le n_1+n_2-n_3+\frac{a}{2}+b_2-k_1}   \\ \hline \binom{n_1+\frac{a}{2}+b_1+k_2+i-d_1}{j-1-n_3-\frac{a}{2}-b_3-k_2}_{1 \le j \le n_3+\frac{a}{2}+b_3+k_1} & \binom{n_1+\frac{a}{2}+b_1+k_2+i-d_1}{j+n_3+\frac{a}{2}+b_1+b_3+k_1-1-n_3-\frac{a}{2}-b_3-k_2}_{1 \le j \le n_1+n_2-n_3+\frac{a}{2}+b_2-k_1} \end{array} \right)
\end{multline*}
(Concerning the range of the right block, note that $n_1+n_2-n_3+\frac{a}{2}+b_2-k_1 \ge 0$ because of the following: consider the wedge of the line containing the $(n_1+a+b_1+b_2+b_3)$-side of the hexagon and of the line containing the $(n_2+a+b_1+b_2+b_3)$-side of the hexagon. Then the length of the section of the horizontal line containing the top vertex of the satellite of size $b_1$ in this wedge is $n_1+n_2+\frac{a}{2}+b_2-k_1$ which is now obviously greater than or equal to $n_3$, the length of the top side of the hexagon.)

Next we will show that we can modify this formula by introducing $(-1)^{b_1}$ at various places such that the result is a polynomial function in $b_1$. Since this has of course no effect if $b_1$ is even, the modified formula will give the number of lozenge tilings for even $b_1$. However, using Lemma~\ref{indeplem} and the fact that a polynomial is uniquely determined by its evaluation on even integers, the modified formula gives the number of lozenge tilings for all non-negative integers $b_1$ (however still assuming that $b_2$ and $b_3$ are even).

The following observations are crucial:
\begin{itemize} 
\item The entries in the first $n_3+\frac{a}{2}+b_3+k_1$ columns are all polynomials in $b_1$, since $b_1$ appears at most in the upper parameter of the binomial coefficient. 

\item The entries that are right of column $n_3+\frac{a}{2}+b_3+k_1$ and below row $n_2$ are also polynomials in $b_1$: These entries are binomial coefficients of the form $\binom{b_1+s}{b_1+t}$ for some integers $s$ and $t$. We have $b_1+s \ge 0$ which follows basically because the satellite of size $b_1$ is the leftmost removed (big) triangle except for the triangle of size $n_2$, which however corresponds to the top block (see also \eqref{geometry}). Thus we can apply the symmetry of the binomial coefficient, i.e.,
\begin{equation}
\label{sym}
\binom{n}{k}= \binom{n}{n-k} \quad \text{if $n \ge 0$},
\end{equation}
to obtain binomial coefficients where $b_1$ only appears in the top parameter. 
\item As for the remaining entries in row $1$ to $n_2$, they are binomial coefficients of the form $\binom{s}{b_1+t}$ where $s<b_1+t$. In order to see this, observe that the extreme case with regard to this inequality is when $i=n_2$ and $j=1$. In this case we need to show that 
$$
n_2 \le \left(n_3 + \frac{a}{2} + b_3 + k_1\right)+b_1+\left(n_1-2 k_1\right).
$$
However, this is obvious: $n_3+\frac{a}{2}+b_3+k_1$ is the ``lattice'' distance of the satellite of size $b_1$ from the bottom of the hexagon, thus $\left(n_3 + \frac{a}{2} + b_3 + k_1\right)+b_1$ is the 
``lattice'' distance between the top of this satellite to the bottom of the hexagon and going $n_1-2 k_1$ unit steps from this top into $\nwarrow$-direction will bring us to a point on the side of length $n_1+a+b_1+b_2+b_3$, which is thus surely above the side of length $n_2$. 

We claim that this implies that $(-1)^{b_1} \binom{s}{b_1+t}$ is polynomial in $b_1$: 
We use the second elementary transformation for binomial coefficients, i.e.,
\begin{equation}
\label{elementary2}
\binom{n}{k} = (-1)^{k} \binom{k-n-1}{k} 
\end{equation}
to see that 
$$
\binom{s}{b_1+t} = (-1)^{b_1+t} \binom{b_1+t-s-1}{b_1+t} = (-1)^{b_1+t} \binom{b_1+t-s-1}{-s-1},
$$
where the last step follows from the symmetry \eqref{sym} which can be applied since $b_1+t-s-1 \ge 0$.
\end{itemize}
It follows that we obtain a formula that is a polynomial function in $b_1$ and coincides with the original formula for even $b_1$ if we do the following: 
\begin{itemize}
\item Multiply the expression with $(-1)^{(n_1+n_2+n_3+\frac{3a}{2}+k_1)b_1}$.
\item Multiply the entries in the first $n_2$ rows and right of column $n_3+\frac{a}{2}+b_3+k_1$ with $(-1)^{b_1}$.
\end{itemize}
If we ``reverse'' after this modification our calculation so that we again have a block that corresponds to the satellite of size $b_1$, we obtain 
\begin{multline*}
(-1)^{(n_1+n_2+n_3+\frac{3a}{2}+k_1)b_1} 
\prod_{i=1}^{b_3} \fd_{c_{3,i}}^{n_3-2 k_3} \prod_{i=1}^{a} \fd_{c_{4,i}}^{n_3+b_3} \prod_{i=1}^{b_1} \fd_{c_{5,i}}^{n_3+\frac{a}{2}+b_3+k_1} \prod_{i=1}^{b_2} \fd_{c_{6,i}}^{n_3+\frac{a}{2}+b_3+k_2} \\
\det \left( \begin{matrix*}[l]   \binom{c_{1,i}-d_1}{j-1} (-1)^{[j \ge n_3+\frac{a}{2}+b_1+b_3+k_1+1] b_1} & {1 \le i \le n_2}   \vspace{2mm} \\ \binom{c_{2,i}-d_1}{j-1} & {1 \le i \le n_1 }  \vspace{2mm} \\ \binom{c_{3,i}-d_1}{j-1} & {1 \le i \le b_3 } \vspace{2mm} \\ \binom{c_{4,i}-d_1}{j-1} &  {1 \le i \le a } \vspace{2mm} \\ \binom{c_{5,i}-d_1}{j-1} & {1 \le i \le b_1} \vspace{2mm} \\ \binom{c_{6,i}-d_1}{j-1} & {1 \le i \le b_2} \end{matrix*} \right)_{1 \le j \le n}
\end{multline*}
at $c_{1,i}=i,$ $c_{2,i}=n_2+n_3+a+b_1+b_2+b_3+i$,$c_{3,i}=n_1+\frac{a}{2}+b_1+k_2+i$, $c_{4,i}=n_1+b_1+i$, $c_{5,i}=n_1 -2 k_1+i$, $c_{6,i}=n_1+\frac{a}{2}+b_1+k_2+i$, provided that $d_1=n_1-2 k_1+1$.
Note that $(-1)^{[j \ge n_3+\frac{a}{2}+b_1+b_3+k_1+1] b_1}$ can actually be replaced by any 
$(-1)^{[j \ge n_3+\frac{a}{2}+l+b_3+k_1+1] b_1}$ with $0 \le l \le b_1$: when ``eliminating'' a block it becomes apparent that the values of certain entries do not play a role at all. When we reverse the procedure, we are free to choose the values conveniently. We choose $l=0$ in the following.

Now observe that, by the Chu-Vandermonde summation, 
$$
\binom{c-d_2}{j-1} = \sum_{q=1}^{j} \binom{c-d_1}{q-1} \binom{d_1-d_2}{j-q}, 
$$
and multiply the matrix underlying the determinant from the right with the upper triangular matrix $(\binom{d_1-d_2}{j-i})_{1 \le i,j \le n}$ with determinant $1$.  This gives the following matrix.
$$
\left( \begin{matrix*}[l] \binom{c_{1,i}-d_2}{j-1} + ((-1)^{b_1}-1)  \sum\limits_{q=n_3+\frac{a}{2}+b_3+k_1+1}^{j}  \binom{c_{1,i}-d_1}{q-1} \binom{d_1-d_2}{j-q}  & {1 \le i \le n_2}  \vspace{2mm} \\ \binom{c_{2,i}-d_2}{j-1} & {1 \le i \le n_1 }  \vspace{2mm} \\ \binom{c_{3,i}-d_2}{j-1} & {1 \le i \le b_3 } \vspace{2mm} \\ \binom{c_{4,i}-d_2}{j-1} & {1 \le i \le a } \vspace{2mm} \\ \binom{c_{5,i}-d_2}{j-1} & {1 \le i \le b_1} \vspace{2mm} \\ \binom{c_{6,i}-d_2}{j-1} & {1 \le i \le b_2} \end{matrix*} \right)_{1 \le j \le n}
$$
Specializing the $c_{l,i}$, we obtain the following expression.
\begin{multline*}
(-1)^{(n_1+n_2+n_3+\frac{3a}{2}+k_1)b_1}  \\
\times \det \left( \begin{matrix*}[l] \binom{i-d_2}{j-1} + ((-1)^{b_1}-1)  \sum\limits_{q=n_3+\frac{a}{2}+b_3+k_1+1}^{j}  \binom{i-d_1}{q-1} \binom{d_1-d_2}{j-q}  & {1 \le i \le n_2}  \vspace{2mm} \\ \binom{n_2+n_3+a+b_1+b_2+b_3+i-d_2}{j-1} & {1 \le i \le n_1 }  \vspace{2mm} \\ \binom{n_1+\frac{a}{2}+b_1+k_2+i-d_2}{j-1-n_3+2 k_2} & {1 \le i \le b_3 } \vspace{2mm} \\ \binom{n_1+b_1+i-d_2}{j-1-n_3-b_3} & {1 \le i \le a } \vspace{2mm} \\ \binom{n_1 -2 k_1+i-d_2}{j-1-n_3-\frac{a}{2}-b_3-k_1} & {1 \le i \le b_1} \vspace{2mm} \\ \binom{n_1+\frac{a}{2}+b_1+k_2+i-d_2}{j-1-n_3-\frac{a}{2}-b_3-k_2} & {1 \le i \le b_2} \end{matrix*} \right)_{1 \le j \le n}
\end{multline*}
We set $d_2=n_1+\frac{a}{2}+b_1+k_2+1$. (The assumption $k_2=k_3$ is now useful because it allows us to eliminate the blocks of the satellites of sizes $b_2$ and $b_3$ simultaneously.) With this, all entries in first row of the bottom block are zero except for the one in column $j=1+n_3+\frac{a}{2}+b_3+k_2$, and so we expand with respect to this row. We can keep doing this until the bottom block vanishes and obtain the following.
\begin{multline*}
(-1)^{(n_1+n_2+n_3+\frac{3a}{2}+k_1)b_1+(n_1+n_2+n_3+\frac{3a}{2} +b_1+k_2)b_2} \\
\times \det \left( \begin{matrix*}[l] \binom{i-d_2}{j-1} + ((-1)^{b_1}-1)  \sum\limits_{q=n_3+\frac{a}{2}+b_3+k_1+1}^{j}  \binom{i-d_1}{q-1} \binom{d_1-d_2}{j-q} & {1 \le i \le n_2}  \vspace{2mm} \\ \binom{n_2+n_3+a+b_1+b_2+b_3+i-d_2}{j-1} & {1 \le i \le n_1 }  \vspace{2mm} \\ \binom{n_1+\frac{a}{2}+b_1+k_2+i-d_2}{j-1-n_3+2 k_2} & {1 \le i \le b_3 } \vspace{2mm} \\ \binom{n_1+b_1+i-d_2}{j-1-n_3-b_3} & {1 \le i \le a } \vspace{2mm} \\ \binom{n_1 -2 k_1+i-d_2}{j-1-n_3-\frac{a}{2}-b_3-k_1} & {1 \le i \le b_1}  \end{matrix*} \right)_{1 \le j \le n_3+\frac{a}{2}+b_3+k_2 \atop n_3+\frac{a}{2}+b_2+b_3+k_2+1 \le j \le n}
\end{multline*}
The first $n_3+\frac{a}{2}+b_3+k_2$ columns of the matrix underlying the determinant are 
$$\left( \begin{matrix*}[l] \binom{i-d_2}{j-1} + ((-1)^{b_1}-1)  \sum\limits_{q=n_3+\frac{a}{2}+b_3+k_1+1}^{j}  \binom{i-d_1}{q-1} \binom{d_1-d_2}{j-q} & {1 \le i \le n_2}  \vspace{2mm} \\ \binom{n_2+n_3+a+b_1+b_2+b_3+i-d_2}{j-1} & {1 \le i \le n_1 }  \vspace{2mm} \\ \binom{n_1+\frac{a}{2}+b_1+k_2+i-d_2}{j-1-n_3+2 k_2} & {1 \le i \le b_3 } \vspace{2mm} \\ \binom{n_1+b_1+i-d_2}{j-1-n_3-b_3} & {1 \le i \le a } \vspace{2mm} \\ \binom{n_1 -2 k_1+i-d_2}{j-1-n_3-\frac{a}{2}-b_3-k_1} & {1 \le i \le b_1}  \end{matrix*} \right)_{1 \le j \le n_3+\frac{a}{2}+b_3+k_2}.
$$
Only the entries in the second block depend on $b_2$, and, since $b_2$ appears only in the upper parameter of the binomial coefficient, these entries are polynomials in $b_2$. The matrix consisting of the remaining columns can be written as follows
$$\left( \begin{matrix*}[l] \binom{i-d_2}{j+n_3+\frac{a}{2}+b_2+b_3+k_2-1} + ((-1)^{b_1}-1)  \sum\limits_{q=n_3+\frac{a}{2}+b_3+k_1+1}^{j+n_3+\frac{a}{2}+b_2+b_3+k_2}  \binom{i-d_1}{q-1} \binom{d_1-d_2}{j+n_3+\frac{a}{2}+b_2+b_3+k_2-q} & {1 \le i \le n_2}  \vspace{2mm} \\ \binom{n_2+n_3+a+b_1+b_2+b_3+i-d_2}{j+n_3+\frac{a}{2}+b_2+b_3+k_2-1} & {1 \le i \le n_1 }  \vspace{2mm} \\ \binom{n_1+\frac{a}{2}+b_1+k_2+i-d_2}{j+\frac{a}{2}+b_2+b_3+ 3 k_2-1} & {1 \le i \le b_3 } \vspace{2mm} \\ \binom{n_1+b_1+i-d_2}{j+\frac{a}{2}+b_2+k_2-1} & {1 \le i \le a } \vspace{2mm} \\ \binom{n_1 -2 k_1+i-d_2}{j+b_2+k_2-k_1-1} & {1 \le i \le b_1}  \end{matrix*} \right),
$$
where $1 \le j \le n_1+n_2-n_3+\frac{a}{2}+b_1-k_2$. We analyze the different blocks of the matrix:
\begin{itemize}
\item Top block: In $\binom{i-d_2}{j+n_3+\frac{a}{2}+b_2+b_3+k_2-1}$, the upper parameter is always less than the lower parameter, and so, in analogy to a situation for $b_1$, this binomial coefficient is a polynomial function in $b_2$ after multiplication with $(-1)^{b_2}$. Now, as $d_1-d_2 < 0$ (unless $b_1=0$ in which case the entry simplifies to the binomial coefficient that was already discussed), $\binom{d_1-d_2}{j+n_3+\frac{a}{2}+b_2+b_3+k_2-q}$ is a polynomial function in $b_2$ and $q$ when multiplied with $(-1)^{b_2+q}$. In case $i-d_1$ is non-negative, we can sum over all $q$ less than or equal to $i-d_1+1$ (because otherwise the binomial coefficient $\binom{i-d_1}{q-1}$ is zero), and, since $b_2$ has now disappeared from the upper bound in the summation, the entry is seen to be a polynomial function in $b_2$ after multiplication with $(-1)^{b_2}$. If, however $i-d_1$ is negative, then $\binom{i-d_1}{q-1}$ is a polynomial in $q$ after multiplication with $(-1)^{q}$, and so the summand $\binom{i-d_1}{q-1} \binom{d_1-d_2}{j+n_3+\frac{a}{2}+b_2+b_3+k_2-q}$ is a polynomial function in $q$. Using the fact that $\sum\limits_{i=a}^{b} p(i)$ is a polynomial function in $a$ and $b$ if $p(i)$ is a polynomial in $i$, it follows that also in this case, the entry is a polynomial function in $b_2$ after multiplication with $(-1)^{b_2}$.
\item Second block: $b_2$ appears in the upper parameter as well as in the lower parameter of the binomial coefficient. As the upper parameter is non-negative, the symmetry can be applied in order to remove $b_2$ from the lower parameter.
\item As for the remaining blocks, the entries are always of the form $\binom{s}{b_2+t}$ where $s < b_2+t$, which implies that these entries are polynomial functions after multiplication with $(-1)^{b_2}$.
\end{itemize}
Summarizing we see that, in order to transform the determinant formula into a polynomial function in $b_2$, we need to do the following.
\begin{itemize}
\item Multiply with $(-1)^{(n_1+n_2+n_3+\frac{3a}{2} +b_1+k_2)b_2}$.
\item Multiply the entries in the columns right of the column $n_3+\frac{a}{2}+b_3+k_2$ with $(-1)^{b_2}$, except for those in the second block.
\end{itemize}
Since there are $n_1+n_2-n_3+\frac{a}{2}+b_1-k_2$ columns right of the column $n_3+\frac{a}{2}+b_3+k_2$, this is equivalent to the following.
\begin{itemize}
\item Multiply only the entries in the second block right of the column $n_3+\frac{a}{2}+b_3+k_2$ with $(-1)^{b_2}$.
\end{itemize}
Going back in our calculation and reintroducing a block with $b_2$ rows, we obtain 
\begin{multline*}
(-1)^{(n_1+n_2+n_3+\frac{3a}{2}+k_1)b_1}
 \times \det \left( \begin{matrix*}[l]  \binom{i-d_2}{j-1} + ((-1)^{b_1}-1)  \sum\limits_{q=n_3+\frac{a}{2}+b_3+k_1+1}^{j}  \binom{i-d_1}{q-1} \binom{d_1-d_2}{j-q}   & {1 \le i \le n_2}  \vspace{2mm} \\  (-1)^{[j \ge n_3+\frac{a}{2}+b_3+k_2+1] b_2} \binom{n_2+n_3+a+b_1+b_3+i-d_2}{j-1} & {1 \le i \le n_1 }  \vspace{2mm} \\   \binom{n_1+\frac{a}{2}+b_1+k_2+i-d_2}{j-1-n_3+2 k_2} & {1 \le i \le b_3 } \vspace{2mm} \\   \binom{n_1+b_1+i-d_2}{j-1-n_3-b_3} & {1 \le i \le a } \vspace{2mm} \\   \binom{n_1 -2 k_1+i-d_2}{j-1-n_3-\frac{a}{2}-b_3-k_1} & {1 \le i \le b_1} \vspace{2mm} \\  \binom{n_1+\frac{a}{2}+b_1+k_2+i-d_2}{j-1-n_3-\frac{a}{2}-b_3-k_2} & {1 \le i \le b_2} \end{matrix*} \right)_{1 \le j \le n}.
\end{multline*}
As for $b_3$, a similar argument shows that we need to make the adjustment only in the top block. This concludes the proof of the theorem.
\end{proof}


\subsection{Polynomiality in $a$}

The purpose of this section is to modify the formula in Theorem~\ref{oddb} to reveal the polynomiality in $a$. More specifically, we prove the following theorem.

\begin{theo}
\label{evenodd}
Let
$$
A  = \left( \begin{matrix*}[l] \binom{i-n_1-b_1-1}{j-1} + ((-1)^{b_3} - 1) \sum\limits_{p=1+n_3-2 k_2}^j \binom{i-n_1-\frac{a}{2}-b_1-k_2-1)}{p-1}   \binom{\frac{a}{2}+k_2}{j-p}  & 1 \le i \le n_2
  \vspace{2mm} \\ \binom{-n_1+n_2+n_3+a+b_2+b_3+i-1}{j-1} 
  \vspace{2mm} &  1 \le i \le n_1 \\  \binom{\frac{a}{2}+k_2+i-1}{j-1-n_3+2 k_2}  & 1 \le i \le b_3 \vspace{2mm} \\   0  & 1 \le i \le b_1 \vspace{2mm} \\  0 & 1 \le i \le b_2 \end{matrix*} \right)_{1 \le j \le  n_3+b_3}
$$
and 
$$
\renewcommand\arraystretch{1.3}
B'= 
\left( \begin{array}{ll} (-1)^{j+n_3-1}  \binom{n_1+n_3+a+b_1+b_3-i+j-1}{n_1+b_1-i}    \\ 
\quad + \left((-1)^{b_1}-1 \right) (-1)^{j+n_3-1} \sum\limits_{q=n_3+\frac{a}{2}+b_3+k_1+1}^{j+n_3+a+b_3}    \binom{n_1-2 k_1-i+q-1}{n_1-2 k_1-i}  \binom{j+n_3+a+b_1+b_3+2 k_1-q-1}{b_1+2 k_1-1}  & 1 \le i \le n_2 \\ \hline \binom{-n_1+n_2+n_3+a+b_2+b_3+i-1}{
-n_1+n_2+b_2+i-j} \\ \quad +  \left((-1)^{b_2}-1 \right) 
\sum\limits_{p=1}^{-n_1+n_2+b_2-2 k_2+i} \binom{-n_1+n_2+n_3+\frac{a}{2}+b_2+b_3-k_2+i-1}{-n_1+n_2+b_2-2 k_2+i-p} \binom{\frac{a}{2}+k_2}{2 k_2-j+p}  & 1 \le i \le n_1  \\ \hline   0  & 1 \le i \le b_3  \\ \hline      (-1)^{j}\binom{\frac{a}{2}+b_1 + k_1 - i+j-1}{b_1 + 2 k_1 - i} & 1 \le i \le b_1 \\ \hline  \binom{\frac{a}{2}+k_2+i-1}{2 k_2+i-j} & 1 \le i \le b_2  \end{array} \right), 
$$
where the range of $j$ in $B'$ is $1 \le j \le n_1+n_2-n_3+b_1+b_2$. Then the number of lozenge tilings of $S_{n_1,n_2,n_3,a,b_1,b_2,b_3,k_1,k_2,k_2}$ is the absolute value of $\det \left( A \, | \, B' \right)$,
which is obviously a polynomial in $a$ since all matrix entries are polynomials in $a$.
\end{theo}

\begin{proof} 
We need to eliminate the fourth block in the formula in Theorem~\ref{oddb}. Note that this formula can also be written (up to sign) as follows: 
\begin{multline*}
(-1)^{(n_1+n_2+n_3+\frac{3a}{2}+k_1)b_1 +(n_1+n_2+n_3)b_3}
  \\
\times \prod_{i=1}^{b_3} \fd_{c_{3,i}}^{n_3-2 k_2} \prod_{i=1}^{a} \fd_{c_{4,i}}^{n_3+b_3} \prod_{i=1}^{b_1} \fd_{c_{5,i}}^{n_3+\frac{a}{2}+b_3+k_1} \prod_{i=1}^{b_2} \fd_{c_{6,i}}^{n_3+\frac{a}{2}+b_3+k_2} \\
\det \left( \begin{matrix*}[l] 
(-1)^{[j \ge 1 + n_3-2 k_2] b_3} 
\left( \binom{c_{1,i}-d_2}{j-1} + ((-1)^{b_1}-1) \sum\limits_{q=n_3+\frac{a}{2}+b_3+k_1+1}^{j} \binom{c_{1,i}-d_1}{q-1} \binom{d_1-d_2}{j-q} \right) &
{1 \le i \le n_2}  \vspace{2mm} \\
(-1)^{[j \ge n_3+\frac{a}{2}+b_3+k_2+1] b_2} \binom{c_{2,i}-d_2}{j-1} & {1 \le i \le n_1 }  \vspace{2mm} \\ 
 \binom{c_{3,i}-d_2}{j-1} & {1 \le i \le b_3 } \vspace{2mm} \\ 
 \binom{c_{4,i}-d_2}{j-1} & {1 \le i \le a } \vspace{2mm} \\ 
 \binom{c_{5,i}-d_2}{j-1} & {1 \le i \le b_1} \vspace{2mm} \\ 
 \binom{c_{6,i}-d_2}{j-1} & {1 \le i \le b_2} 
\end{matrix*} \right)_{1 \le j \le n}
\end{multline*}
evaluated at $c_{1,i}=i,$ $c_{2,i}=n_2+n_3+a+b_1+b_2+b_3+i$,$c_{3,i}=n_1+\frac{a}{2}+b_1+k_2+i$, $c_{4,i}=n_1+b_1+i$, $c_{5,i}=n_1 -2 k_1+i$, $c_{6,i}=n_1+\frac{a}{2}+b_1+k_2+i$, $d_1=n_1-2 k_1+1$ and $d_2=n_1+\frac{a}{2}+b_1+k_2+1$. 

We multiply the matrix underlying the determinant from the right with the upper triangular matrix $(\binom{d_2-d_3}{j-i})_{1 \le i,j \le n}$ with determinant $1$. In block $l$, $3 \le l \le 6$, the entry is then replaced by $\binom{c_{l,i}-d_3}{j-1}$. In the second block, we have 
\begin{multline*}
\sum_{p=1}^{n} (-1)^{[p \ge n_3+\frac{a}{2}+b_3+k_2+1] b_2}  \binom{c_{2,i}-d_2}{p-1} \binom{d_2-d_3}{j-p} \\ = \binom{c_{2,i}-d_3}{j-1} + ((-1)^{b_2}-1) 
\sum_{p=n_3+\frac{a}{2}+b_3+k_2+1}^{j} \binom{c_{2,i}-d_2}{p-1} \binom{d_2-d_3}{j-p}.
\end{multline*}
As for the top block, using $n_3 + \frac{a}{2} + b_3 + k_2 +1 \ge n_3 - 2 k_2 +1$ as well as the Chu-Vandermonde summation, 
we have 
\begin{multline*}
\sum_{p=1}^{n}  (-1)^{[p \ge 1+n_3-2 k_2+b_3] b_3}  \binom{c_{1,i}-d_2}{p-1}   \binom{d_2-d_3}{j-p}  \\  +  \left( (-1)^{b_1+b_3}+(-1)^{1+b_3} \right) \sum_{p \ge 1, q \ge n_3+\frac{a}{2}+b_3+k_1+1}  \binom{c_{1,i}-d_1}{q-1}  \binom{d_1-d_2}{p-q}  \binom{d_2-d_3}{j-p} \\
=  \binom{c_{1,i}-d_3}{j-1} + ((-1)^{b_3} - 1) \sum_{p=1+n_3-2 k_2}^j \binom{c_{1,i}-d_2}{p-1}   \binom{d_2-d_3}{j-p}  \\
+ \left((-1)^{b_1+b_3}-(-1)^{b_3} \right) \sum_{q=n_3+\frac{a}{2}+b_3+k_1+1}^{j}   \binom{c_{1,i}-d_1}{q-1}  \binom{d_1-d_3}{j-q}.
\end{multline*}
We obtain the following 
\begin{multline*}
\renewcommand\arraystretch{1.3}
(-1)^{(n_1+n_2+n_3+\frac{3a}{2}+k_1)b_1 +(n_1+n_2+n_3)b_3}  \\
\times 
\det \left( \begin{array}{ll}  \binom{c_{1,i}-d_3}{j-1} + ((-1)^{b_3} - 1) \sum\limits_{p=1+n_3-2 k_2}^j \binom{c_{1,i}-d_2}{p-1}   \binom{d_2-d_3}{j-p} \\
\quad + \left((-1)^{b_1+b_3}-(-1)^{b_3} \right) \sum\limits_{q=n_3+\frac{a}{2}+b_3+k_1+1}^{j}   \binom{c_{1,i}-d_1}{q-1}  \binom{d_1-d_3}{j-q}  & {1 \le i \le n_2}  \\ \hline   \binom{c_{2,i}-d_3}{j-1} + ((-1)^{b_2}-1) 
\sum\limits_{p=n_3+\frac{a}{2}+b_3+k_2+1}^{j} \binom{c_{2,i}-d_2}{p-1} \binom{d_2-d_3}{j-p}   & {1 \le i \le n_1} \vspace{2mm} \\  \hline \binom{c_{3,i}-d_3}{j-1-n_3+2 k_2} & {1 \le i \le b_3 } \vspace{2mm} \\ \hline  \binom{c_{4,i}-d_3}{j-1-n_3-b_3} & {1 \le i \le a }  \\ \hline   \binom{c_{5,i}-d_3}{j-1-n_3-\frac{a}{2}-b_3-k_1} & {1 \le i \le b_1} \vspace{2mm} \\ \hline  \binom{c_{6,i}-d_3}{j-1-n_3-\frac{a}{2}-b_3-k_2} & {1 \le i \le b_2} \end{array} \right)_{1 \le j \le n}.
\end{multline*}
Evaluating at $c_{1,i}=i,$ $c_{2,i}=n_2+n_3+a+b_1+b_2+b_3+i$,$c_{3,i}=n_1+\frac{a}{2}+b_1+k_2+i$, $c_{4,i}=n_1+b_1+i$, $c_{5,i}=n_1 -2 k_1+i$, $c_{6,i}=n_1+\frac{a}{2}+b_1+k_2+i$ gives
\begin{multline*}
\renewcommand\arraystretch{1.3}
(-1)^{(n_1+n_2+n_3+\frac{3a}{2}+k_1)b_1 +(n_1+n_2+n_3)b_3} \\
\times 
\det \left( \begin{array}{ll} \binom{i-d_3}{j-1} + ((-1)^{b_3} - 1) \sum\limits_{p=1+n_3-2 k_2}^j \binom{i-d_2}{p-1}   \binom{d_2-d_3}{j-p} \\  
\quad + \left((-1)^{b_1+b_3}-(-1)^{b_3} \right) \sum\limits_{q=n_3+\frac{a}{2}+b_3+k_1+1}^{j}   \binom{i-d_1}{q-1}  \binom{d_1-d_3}{j-q}   & {1 \le i \le n_2}   \\ \hline  \binom{n_2+n_3+a+b_1+b_2+b_3+i-d_3}{j-1} \\  \quad + ((-1)^{b_2}-1) 
\sum\limits_{p=n_3+\frac{a}{2}+b_3+k_2+1}^{j} \binom{n_2+n_3+a+b_1+b_2+b_3+i-d_2}{p-1} \binom{d_2-d_3}{j-p}
   & {1 \le i \le n_1}  \\ \hline  \binom{n_1+\frac{a}{2}+b_1+k_2+i-d_3}{j-1-n_3+2 k_2} & {1 \le i \le b_3 }  \\ \hline  \binom{n_1+b_1+i-d_3}{j-1-n_3-b_3} & {1 \le i \le a }  \\ \hline   \binom{n_1 -2 k_1+i-d_3}{j-1-n_3-\frac{a}{2}-b_3-k_1} & {1 \le i \le b_1}  \\ \hline   \binom{n_1+\frac{a}{2}+b_1+k_2+i-d_3}{j-1-n_3-\frac{a}{2}-b_3-k_2} & {1 \le i \le b_2} \end{array} \right)_{1 \le j \le n}.
\end{multline*}
Now we perform the replacement $d_1=n_1-2 k_1+1$ and $d_2=n_1+\frac{a}{2}+b_1+k_2+1$, and specify furthermore $d_3=n_1+b_1+1$.
\begin{multline*}
\renewcommand\arraystretch{1.3}
(-1)^{(n_1+n_2+n_3+\frac{3a}{2}+k_1)b_1 +(n_1+n_2+n_3)b_3}  \\
\times \det \left( \begin{array}{ll} \binom{i-n_1-b_1-1}{j-1} + ((-1)^{b_3} - 1) \sum\limits_{p=1+n_3-2 k_2}^j \binom{i-n_1-\frac{a}{2}-b_1-k_2-1)}{p-1}   \binom{\frac{a}{2}+k_2}{j-p}  \\
\quad + \left((-1)^{b_1+b_3}-(-1)^{b_3} \right) \sum\limits_{q=n_3+\frac{a}{2}+b_3+k_1+1}^{j}   \binom{i-n_1+2 k_1-1)}{q-1}  \binom{-b_1-2 k_1}{j-q} & 1 \le i \le n_2 \\ \hline   \binom{-n_1+n_2+n_3+a+b_2+b_3+i-1}{j-1}  \\  \quad + ((-1)^{b_2}-1) 
\sum\limits_{p=n_3+\frac{a}{2}+b_3+k_2+1}^{j} \binom{-n_1+n_2+n_3+\frac{a}{2}+b_2+b_3-k_2+i-1)}{p-1} \binom{\frac{a}{2}+k_2}{j-p}  & 1 \le i \le n_1  \\ \hline     \binom{\frac{a}{2}+k_2+i-1}{j-1-n_3+2 k_2} & 
1 \le i \le b_3  \\ \hline \binom{i-1}{j-1-n_3-b_3}  & 1 \le i \le a  \\ \hline   \binom{-b_1 -2 k_1+i-1}{j-1-n_3-\frac{a}{2}-b_3-k_1}  & 1 \le i \le b_1 \\ \hline  \binom{\frac{a}{2}+k_2+i-1}{j-1-n_3-\frac{a}{2}-b_3-k_2} & 1 \le i \le b_2 \end{array} \right)_{1 \le j \le n}
\end{multline*}
We can now eliminate the fourth block, and obtain 
\begin{multline*}
\renewcommand\arraystretch{1.3}
(-1)^{(n_1+n_2+n_3+\frac{3a}{2}+k_1)b_1 +(n_1+n_2+n_3)b_3}  \\
\times \det \left( \begin{array}{ll} \binom{i-n_1-b_1-1}{j-1} + ((-1)^{b_3} - 1) \sum\limits_{p=1+n_3-2 k_2}^j \binom{i-n_1-\frac{a}{2}-b_1-k_2-1)}{p-1}   \binom{\frac{a}{2}+k_2}{j-p}  \\
\quad + \left((-1)^{b_1+b_3}-(-1)^{b_3} \right) \sum\limits_{q=n_3+\frac{a}{2}+b_3+k_1+1}^{j}   \binom{i-n_1+2 k_1-1)}{q-1}  \binom{-b_1-2 k_1}{j-q}  & 1 \le i \le n_2  \\ \hline  \binom{-n_1+n_2+n_3+a+b_2+b_3+i-1}{j-1} \\ \quad + ((-1)^{b_2}-1) 
\sum\limits_{p=n_3+\frac{a}{2}+b_3+k_2+1}^{j} \binom{-n_1+n_2+n_3+\frac{a}{2}+b_2+b_3-k_2+i-1)}{p-1} \binom{\frac{a}{2}+k_2}{j-p}  & 1 \le i \le n_1  \\ \hline  \binom{\frac{a}{2}+k_2+i-1}{j-1-n_3+2 k_2}  & 1 \le i \le b_3  \\ \hline   \binom{-b_1 -2 k_1+i-1}{j-1-n_3-\frac{a}{2}-b_3-k_1} & 1 \le i \le b_1  \\ \hline  \binom{\frac{a}{2}+k_2+i-1}{j-1-n_3-\frac{a}{2}-b_3-k_2}  & 1 \le i \le b_2 \end{array} \right), 
\end{multline*}
where the range for $j$ is $1 \le j \le n_3+b_3$ and $n_3+a+b_3+1 \le j \le n$. Note that the entries vanish for $1 \le j \le n_3+b_3$ in blocks $4$ and $5$, as well as for 
$n_3+a+b_3+1 \le j \le n$ in block $3$. Also, for $1 \le j \le n_3+b_3$, the last sums for the entries in block $1$ and $2$ vanish, since the upper parameter in the summation is less than the lower parameter. Now note that the $n_3+b_3$ leftmost columns constitute the matrix $A$ in the statement of the theorem. The entries of $A$ are obviously polynomials in $a$ because $a$  appears only in the upper parameter of the binomial coefficients. We define 
 $$ 
 \renewcommand\arraystretch{1.3}
 B=   \left( \begin{array}{ll}  \binom{i-n_1-b_1-1}{j+n_3+a+b_3-1} + ((-1)^{b_3} - 1) \sum\limits_{p=1+n_3-2 k_2}^{j+n_3+a+b_3} \binom{i-n_1-\frac{a}{2}-b_1-k_2-1)}{p-1}   \binom{\frac{a}{2}+k_2}{j+n_3+a+b_3-p}  \\
\quad + \left((-1)^{b_1+b_3}-(-1)^{b_3} \right) \sum\limits_{q=n_3+\frac{a}{2}+b_3+k_1+1}^{j+n_3+a+b_3}   \binom{i-n_1+2 k_1-1)}{q-1}  \binom{-b_1-2 k_1}{j+n_3+a+b_3-q}    & 1 \le i \le n_2  \\ \hline \binom{-n_1+n_2+n_3+a+b_2+b_3+i-1}{j+n_3+a+b_3-1} \\  \quad + ((-1)^{b_2}-1) 
\sum\limits_{p=n_3+\frac{a}{2}+b_3+k_2+1}^{j+n_3+a+b_3} \binom{-n_1+n_2+n_3+\frac{a}{2}+b_2+b_3-k_2+i-1)}{p-1} \binom{\frac{a}{2}+k_2}{j+n_3+a+b_3-p}  & 1 \le i \le n_1  \\ \hline  0  & 1 \le i \le b_3  \\ \hline   \binom{-b_1 -2 k_1+i-1}{\frac{a}{2}-k_1+j-1} & 1 \le i \le b_1  \\ \hline  \binom{\frac{a}{2}+k_2+i-1}{\frac{a}{2}-k_2+j-1} & 1 \le i \le b_2  \end{array} \right),
 $$
 where $1 \le j \le n_1+n_2-n_3+b_1+b_2$. We know that 
$$
\m(H_{n_1,n_2,n_3,a,b_1,b_2,b_3,k_1,k_2,k_2}) =  |\det ( A \, \vline \, B)|.
$$
The entry in the first block can be simplified as follows: We can extend the first sum to all positive $p$ as $\binom{\frac{a}{2}+k_2}{j+n_3+a+b_3-p}$ vanishes for $1 \le p \le n_3-2 k_2 +b_3$.  Hence, by the Chu-Vandermonde summation, the first sum evaluates to $ \binom{i-n_1-b_1-1}{j+n_3+a+b_3-1}$, which can then be combined with the first term. We obtain 
 $$ 
 \renewcommand\arraystretch{1.3}
 B=  \left( \begin{array}{ll}  (-1)^{b_3} \binom{i-n_1-b_1-1}{j+n_3+a+b_3-1}  \\ \quad + \left((-1)^{b_1+b_3}-(-1)^{b_3} \right) \sum\limits_{q=n_3+\frac{a}{2}+b_3+k_1+1}^{j+n_3+a+b_3}   \binom{i-n_1+2 k_1-1)}{q-1}  \binom{-b_1-2 k_1}{j+n_3+a+b_3-q}   & 1 \le i \le n_2  \\ \hline \binom{-n_1+n_2+n_3+a+b_2+b_3+i-1}{j+n_3+a+b_3-1} \\ \quad + ((-1)^{b_2}-1) 
\sum\limits_{p=n_3+\frac{a}{2}+b_3+k_2+1}^{j+n_3+a+b_3} \binom{-n_1+n_2+n_3+\frac{a}{2}+b_2+b_3-k_2+i-1)}{p-1} \binom{\frac{a}{2}+k_2}{j+n_3+a+b_3-p}  & 1 \le i \le n_1  \\ \hline  0  & 1 \le i \le b_3  \\ \hline    \binom{-b_1 -2 k_1+i-1}{\frac{a}{2}-k_1+j-1} & 1 \le i \le b_1  \\ \hline  \binom{\frac{a}{2}+k_2+i-1}{\frac{a}{2}-k_2+j-1} & 1 \le i \le b_2  \end{array} \right).
 $$
However, using \eqref{sym} as well as \eqref{elementary2}, $B$ can also be written as follows
$$
\renewcommand\arraystretch{1.3}
B=   \left( \begin{array}{ll} (-1)^{j+a+n_3-1}  \binom{n_1+n_3+a+b_1+b_3-i+j-1}{n_1+b_1-i}    \\ 
\quad + \left((-1)^{b_1}-1 \right) (-1)^{j+n_3+a-1} \sum\limits_{q=n_3+\frac{a}{2}+b_3+k_1+1}^{j+n_3+a+b_3}    \binom{n_1-2 k_1-i+q-1}{n_1-2 k_1-i}  \binom{j+n_3+a+b_1+b_3+2 k_1-q-1}{b_1+2 k_1-1}   & 1 \le i \le n_2 \\ \hline  \binom{-n_1+n_2+n_3+a+b_2+b_3+i-1}{
-n_1+n_2+b_2+i-j} \\ \quad +  \left((-1)^{b_2}-1 \right) 
\sum\limits_{p=1}^{-n_1+n_2+b_2-2 k_2+i} \binom{-n_1+n_2+n_3+\frac{a}{2}+b_2+b_3-k_2+i-1}{-n_1+n_2+b_2-2 k_2+i-p} \binom{\frac{a}{2}+k_2}{2 k_2-j+p}  & 1 \le i \le n_1 \\ \hline   0  & 1 \le i \le b_3  \\ \hline     (-1)^{\frac{a}{2}-k_1+j-1}\binom{\frac{a}{2}+b_1 + k_1 - i+j-1}{b_1 + 2 k_1 - i} & 1 \le i \le b_1    \\
\hline  \binom{\frac{a}{2}+k_2+i-1}{2 k_2+i-j} & 1 \le i \le b_2  \end{array} \right).
$$
Now it can be seen that the entries are---up to some signs---polynomials in $a$. Since $a$ is even and we are only interested in the determinant up to sign, we can replace $B$ by 
$B'$ from the statement of the theorem.
\end{proof}

\section{The case $b_1=b_2=b_3=0$ for general $n_1, n_2,n_3$}
\label{exsec}

In this section we demonstrate how to compute the number of lozenge tilings of  $S_{n_1,n_2,n_3,a,0,0,0,0,0,0}$ using the determinant from Theorem~\ref{evenodd}. This establishes a new result --- it gives the number of lozenge tilings of an arbitrary hexagon with a triangular hole of the suitable size\footnote{ So that the resulting region has the same number of up- and down-pointing unit triangles.} removed from a different position than in \cite{CEKZ}. To be precise, in the latter a triangular hole of side $a$ was removed from the center of the hexagon $H$ of side-lengths $n_1$, $n_2+a$, $n_3$, $n_1+a$, $n_2$, $n_3+a$ (counterclockwise from the southeastern edge), while in our result below the distances from the sides of the triangular hole to the NW, NE and S sides of the hexagon are $n_1$, $n_2$ and $n_3$, respectively. One readily sees that this places the triangular hole inside the hexagon if and only if $n_1\leq n_2+n_3$, $n_2\leq n_1+n_3$, and $n_3\leq n_1+n_2$. The two positions agree only if $n_1=n_2=n_3$. For the formulation of the statement, recall that 
$$
\m(S_{n_1,n_2,n_3,a,0,0,0,0,0,0}) = \prod_{i_1=1}^{n_1} \prod_{i_2=1}^{n_2} 
\prod_{i_3=1}^{n_3} \frac{i_1+i_2+i_3-1}{i_1+i_2+i_3-2} =:B(n_1,n_2,n_3)
$$
by MacMahon's box formula \cite{MacM}.

\begin{theo}
\label{newtheo} 
Let $n_1, n_2, n_3$ be non-negative integers with $n_1 \le n_2 \le n_3$
and $n_3 \le n_1 + n_2$ and set 
\begin{multline}
\label{manyfactors}
  Q_{n_1,n_2,n_3}(a)= 
\prod_{i=\lceil (n_1+n_2-1)/2 \rceil}^{\lfloor (n_1+n_3-1)/2 \rfloor} (a+2i+1)^{2i+1-n_3} 
\prod_{i=\lfloor (n_1+n_3-1)/2 \rfloor+1}^{\lfloor (n_2+n_3-1)/2 \rfloor} (a+2i+1)^{n_1} \\
\times \prod_{i=\lfloor n_2+n_3-1)/2 \rfloor+1}^{\lfloor (n_1+n_2+n_3-1)/2 \rfloor} (a+2i+1)^{n_1+n_2+n_3-2 i-1} 
 \prod_{i=\lceil (n_1+n_2+n_3-2)/4 \rceil}^{\lfloor (n_1+n_2-2)/2 \rfloor} (a+2i+1)^{4i+2-n_1-n_2-n_3}  \\ \times \prod_{i= \left \lceil (n_1+n_2)/2 \right \rceil}^{\min \left( \left \lfloor (n_2+n_3-n_1-1)/2 \right \rfloor, \left \lfloor (n_1+n_3)/2 \right \rfloor \right)} (a+2i)^{2i-n_3} \prod_{i= \max \left( \left \lceil (n_1+n_2)/2 \right \rceil, \left \lceil (n_2+n_3-n_1)/2 \right \rceil \right)}^{\left \lfloor (n_1+n_3)/2 \right \rfloor} (a+2i)^{n_2-n_1} \\ \times 
\prod_{i= \max \left( \left \lceil (n_2+n_3-n_1)/2 \right \rceil, \left \lceil (n_1+n_3+1)/2 \right \rceil \right)}^{\left \lfloor (n_2+n_3)/2 \right \rfloor} (a+2i)^{n_2+n_3-2i} \prod_{i=\left \lceil (n_1+n_3+1)/2 \right \rceil}^{ \left \lfloor (n_2+n_3-n_1-1)/2 \right \rfloor} (a+2i)^{n_1} \\ 
\times \prod_{i= \left \lceil (n_1+n_2+n_3)/4  \right \rceil}^{\min\left( \left \lfloor (n_1+n_2-1)/2 \right \rfloor, \left \lfloor (n_2+n_3-n_1-1)/2 \right \rfloor  \right)} (a+2i)^{4i-n_1-n_2-n_3}
\prod_{i= \max \left( \left \lceil (n_2+n_3-n_1)/2 \right \rceil, n_1 \right)}^{\left \lfloor (n_1+n_2-1)/2 \right \rfloor} (a+2i)^{2i-2 n_1} \\
\times \prod_{i=n_2}^{\left \lfloor (n_2+n_3-1)/2 \right \rfloor} (a+2i)^{2i-2 n_2} 
\prod_{i=\left \lceil (n_2+n_3)/2 \right \rfloor}^{n_3} (a+2i)^{2 n_3 - 2 i} \\
\times \prod_{i=1}^{\lfloor (n_1+n_2-n_3)/2 \rfloor} (a+2i)^{2i}  
\prod_{i=\lfloor (n_1+n_2-n_3)/2 \rfloor+1}^{\lfloor (n_1+n_3-n_2)/2 \rfloor}
(a+2i)^{n_1+n_2-n_3}  \prod_{i=\lfloor (n_1+n_3-n_2)/2 \rfloor+1}^{\min(\lfloor (n_2+n_3-n_1)/2 \rfloor, n_1)} (a+2i)^{2 n_1 - 2 i} \\ 
\times \prod_{i=\lfloor (n_2+n_3-n_1)/2 \rfloor +1}^{\lfloor (n_1+n_2+n_3)/4 \rfloor} (a+2i)^{n_1+n_2+n_3-4 i},
\end{multline}
where unlike in \eqref{eba} products are $1$ if the range limits are out of order. Then 
$$
\m(S_{n_1,n_2,n_3,a,0,0,0,0,0,0}) = B_{n_1,n_2,n_3} \frac{Q_{n_1,n_2,n_3}(a)}{Q_{n_1,n_2,n_3}(0)}.
$$
\end{theo}

Our proof approach is to apply Krattenthaler's \emph{``identification of factors''} method \cite[Sec. 2.4]{KrattDet}, which is in this case not complicated as the linear combinations that prove the zeros\footnote{ When the left hand side is regarded as a polynomial in $a$.} turn out to be quite simple. The situation is somewhat similar for  $\m(S_{n,n,n,a,b,b,b,k,k,k})$: the linear combinations are as simple as those used in this section. There the only additional complication lies in the more elaborate (block) structure of the matrix. One important purpose of this section is to demonstrate on a simpler example the procedure that will be used in a forthcoming paper to compute $\m(S_{n,n,n,a,b,b,b,k,k,k})$ in general.

In the special case $b_1=b_2=b_3=0$, Theorem~\ref{evenodd} provides the following matrix, whose determinant we need to compute:
$$
\renewcommand\arraystretch{1.3}
M_{n_1,n_2,n_3} = \left( \begin{array}{c|c} \binom{i-n_1-1}{j-1}_{1 \le i \le n_2, 1 \le j \le n_3} & (-1)^{j+n_3-1} \binom{n_1+n_3+a-i+j-1}{n_1-i}_{1 \le i \le n_2, 1 \le j \le n_1+n_2-n_3}  \\[10pt]
\hline\\[-10pt]
\binom{-n_1+n_2+n_3+a+i-1}{j-1}_{1 \le i \le n_1, 1 \le j \le n_3}  & 
\binom{-n_1+n_2+n_3+a+i-1}{-n_1+n_2+i-j}_{1 \le i \le n_1, 1 \le j \le n_1+n_2-n_3} \end{array} \right).
$$  
We set 
$$
P_{n_1,n_2,n_3}(a) : = \det \left( M_{n_1,n_2,n_3} \right), 
$$
which is obviously a polynomial in $a$. In the next lemma, we compute an upper bound for the degree of this polynomial. As we will see later, this will turn out to be in fact the actual degree.

\begin{lem} 
\label{degree}
Let $n_1,n_2,n_3$ be non-negative integers. Then
$$
\deg_a P_{n_1,n_2,n_3}(a) \le \Bigl\lfloor \frac{2 n_1 n_2 + 2 n_1 n_3 + 2 n_2 n_3 - n_1^2 - n_2^2 - n_3^2}{4} \Bigr\rfloor.
$$
\end{lem}

{\it Proof.} We start by modifying the matrix applying a set of elementary row and column operations. First we transform the bottom block consisting of the bottom $n_1$ rows: We subtract the $(n_1+n_2-1)$-st row from the $(n_1+n_2)$-th row, then the $(n_1+n_2-2)$-nd row from the $(n_1+n_2-1)$-st row etc. until we subtract the $(n_2+1)$-st row from the $(n_2+2)$-nd row. We repeat this, but terminate with the subtraction of $(n_2+2)$-nd row from the $(n_2+3)$-rd row. We repeat this loop $n_1-1$ times where in every step we perform one subtraction less than in the previous step. This way we arrive at the matrix
$$
\renewcommand\arraystretch{1.3}
\left( \begin{array}{c|c} \binom{i-n_1-1}{j-1}_{1 \le i \le n_2, 1 \le j \le n_3} & (-1)^{j+n_3-1} \binom{n_1+n_3+a-i+j-1}{n_1-i}_{1 \le i \le n_2, 1 \le j \le n_1+n_2-n_3}  \\[10pt] \hline
  \\[-10pt]
\binom{-n_1+n_2+n_3+a}{j-i}_{1 \le i \le n_1, 1 \le j \le n_3}  & 
\binom{-n_1+n_2+n_3+a}{-n_1+n_2+i-j}_{1 \le i \le n_1, 1 \le j \le n_1+n_2-n_3} \end{array} \right).
$$  
Second we modify the right block consisting of the $n_1+n_2-n_3$ rightmost columns. We add the $(n_3+2)$-nd column to the $(n_3+1)$-st column, the $(n_3+3)$-rd column to the $(n_3+2)$-nd column etc. until we add the $(n_1+n_2)$-nd column to the $(n_1+n_2-1)$-st column. We repeat this, but terminate with the addition of the $(n_1+n_2-1)$-st column to the $(n_1+n_2-2)$-nd column. We repeat this loop $n_1+n_2-n_3-1$ times where in every step we perform one addition less than in the previous step. The result is the following matrix:
$$
\renewcommand\arraystretch{1.3}
\left( \begin{array}{c|c} \binom{i-n_1-1}{j-1}_{1 \le i \le n_2, 1 \le j \le n_3} & (-1)^{n_1+n_2-1} \binom{n_1+n_3+a-i+j-1}{-n_2+n_3 -i+j}_{1 \le i \le n_2, 1 \le j \le n_1+n_2-n_3}  \\[10pt] \hline
\\[-10pt]
  \binom{-n_1+n_2+n_3+a}{j-i}_{1 \le i \le n_1, 1 \le j \le n_3}  & 
\binom{2 n_2+a-j}{-n_1+n_2+i-j}_{1 \le i \le n_1, 1 \le j \le n_1+n_2-n_3} \end{array} \right).
$$  
Now we find the maximal degree in $a$ of the summands in the Leibniz formula of the determinant: Note that in the top left block of the matrix the degree of the entries is $0$, in the top right block it is $-n_2+n_3-i+j$, in the bottom left block it is $j-i$, while in the bottom right it is $-n_1+n_2+i-j$. 

Let $\sigma$ be a permutation that maximizes the degree in $a$ of the corresponding summand in the Leibniz formula of the determinant and suppose that this summand has $k$ entries from the bottom left block. As the degree is maximal in the top right corner of this block, while in the bottom right block the degree is maximal in the bottom left corner and in the top right block the degree is maximal in the top right corner, the $k$ entries coming from the bottom left block are situated in the top right square of size $k$ of this block. Furthermore, there are $n_1-k$ entries from the bottom right block, and they are situated in the bottom left square of size $n_1-k$ in this block. Similarly, the $k+n_2-n_3$ entries from the top right block are situated in the top right square of size $k+n_2-n_3$. The degrees coming from these squares of size $k$, $n_1-k$ and $k+n_2-n_3$, respectively, are the summands of the following expression.
$$
\sum_{i=1}^{k} (n_3+1-2 i) + \sum_{i=1}^{n_1-k} (n_2+1-2 i) + \sum_{i=1}^{k+n_2-n_3}  (n_1+1-2 i)
$$
This expression is equal to 
$$
- 3 k^2 + 3 k n_1 - n_1^2 - 3 k n_2 + 2 n_1 n_2 - n_2^2 + 3 k n_3 - n_1 n_3 + 2 n_2 n_3 - n_3^2.
$$
The maximum of this expression is at $k = \frac{n_1-n_2+n_3}{2}$. Note that we need to require $k \le n_1, n_3$, $n_1-k \le n_1,n_1+n_2-n_3$ and $k+n_2-n_3 \le n_2, n_1+n_2-n_3$, which in summary gives $n_3-n_2 \le k \le \min(n_1,n_3)$. As 
$$
n_3-n_2 \le \frac{n_1-n_2+n_3}{2} \le \min(n_1,n_3),
$$
which basically follows from $n_{x} \le n_{y}+n_{z}$ if $\{x,y,z\} = \{1,2,3\}$ (these are necessary conditions for the removed triangle to be inside of the hexagon, see Figure~\ref{example0}), 
the maximum is attained at $\lfloor \frac{n_1-n_2+n_3}{2} \rfloor$ and 
at $\lceil \frac{n_1-n_2+n_3}{2} \rceil$. The maximum is then 
\begin{multline*}
\Bigl\lfloor \frac{2 n_1 n_2 + 2 n_1 n_3 + 2 n_2 n_3 - n_1^2 - n_2^2 - n_3^2}{4} \Bigr\rfloor \\ 
= \begin{cases}  \frac{2 n_1 n_2 + 2 n_1 n_3 + 2 n_2 n_3 - n_1^2 - n_2^2 - n_3^2}{4} & \text{if $n_1+n_2+n_3 \equiv 0 \mod 2$} \\ 
\frac{2 n_1 n_2 + 2 n_1 n_3 + 2 n_2 n_3 - n_1^2 - n_2^2 - n_3^2-3}{4} & 
\text{if $n_1+n_2+n_3 \equiv 1 \mod 2$. \ \ \ \ \ $\square$} \end{cases}
\end{multline*}

\medskip
The identification of factors method uses the following principle: 
In order to prove that $P_{n_1,n_2,n_3}(a)$ has a zero at $a=i$ of multiplicity $m$, we need to find $m$ independent linear combinations of the rows (or columns) of the $a=i$ specialization of $M_{n_1,n_2,n_3}$. 

The odd zeros (i.e. the linear factors in \label{manyfactors} that become zero for odd values of $a$) are dealt with in the following lemma. If $A=(a_{i,j})$ is an 
$m \times n$ matrix, then the $d$-th (forward) difference with respect to the rows 
is defined to be the following $(m-d) \times n$ matrix:
$$
\left( \Delta^d_i a_{i,j} \right)_{1 \le i \le m-d,1 \le j \le n}.
$$
Clearly, the rows of this matrix are linear combinations of rows of the original matrix.
The definition is analogous for columns or operators different from $\Delta$.
\begin{lem}
\label{first} Let $d \ge 0$.
Assuming $a=i_1-i_2-n_2-n_3$, $i_1+i_2 \equiv n_2+n_3+1 \,  (\mod 2)$ and $i_2+n_3+d \ge i_1+n_1$, the $i_1$-th row of the $d$-th difference with respect to the rows in the top block is equal to the $i_2$-th row of the $d$-th 
difference in the bottom block, provided that $1 \le i_1 \le n_2-d$ and $1 \le i_2 \le n_1-d$.
\end{lem}

\begin{proof} 
This is obvious for the left block. 
The general entry of the $d$-th difference of the right top block is
$$
\fd_{i}^{d} \binom{n_1+n_3+a-i+j-1}{n_1-i} (-1)^{j+n_3+1} = (-1)^{j+n_3+d+1} \binom{n_1+n_3+a-i+j-1-d}{n_1-i},
$$
while the general entry of the $d$-th difference of the right bottom block is 
$$
\fd_{i}^{d} \binom{-n_1+n_2+n_3+a+i-1}{-n_1+n_2+i-j} = \binom{-n_1+n_2+n_3+a+i-1}{-n_1+n_2+i-j+d}.
$$
After setting $a=i_1-i_2-n_2-n_3$, the entry in row $i_1$ of the $d$-th difference in the top-right block is equal to 
$$
(-1)^{j+n_3+d+1} \binom{n_1-i_2-n_2+j-1-d}{n_1-i_1} = (-1)^{j+n_3+d+1+n_1+i_1} \binom{-i_1+i_2+n_2-j+d}{n_1-i_1}
$$
assuming $i_1 \le n_2-d$ and using \eqref{elementary2},
while the entry in row $i_2$ of the $d$-th difference in the bottom-right block is equal to 
$$
\binom{-n_1+i_1-1}{-n_1+n_2+i_2-j+d} = (-1)^{-n_1+n_2+i_2-j+d} \binom{-i_1+i_2+n_2-j+d}{-n_1+n_2+i_2-j+d}.
$$
The assertion now follows from the symmetry of the binomial coefficient \eqref{sym}.
\end{proof}
For $d=0$, this lemma provides the following linear factors.
\begin{multline*}
\prod_{1 \le i_1 \le n_2, 1 \le i_2 \le n_1 \atop i_1+i_2 \equiv n_2+n_3+1 (\mod 2), i_2 +n_3  \ge i_1 +n_2} 
(a-i_1+i_2+n_2+n_3) \\
= 
\prod_{i=\lceil (n_1+n_2-1)/2 \rceil}^{\lfloor (\min(n_1,n_2)+n_3-1)/2 \rfloor} (a+2i+1)^{2i+1-n_3} 
\prod_{i=\max(\lceil (n_1+n_2-1)/2 \rceil,\lfloor (\min(n_1,n_2)+n_3-1)/2 \rfloor+1)}^{\lfloor (\max(n_1,n_2)+n_3-1)/2 \rfloor} (a+2i+1)^{\min(n_1,n_2)} \\
\times
\prod_{i=\max(\lceil (n_1+n_2-1)/2 \rceil, \lfloor (\max(n_1,n_2)+n_3-1)/2 \rfloor+1)}^{\lfloor (n_1+n_2+n_3-1)/2 \rfloor} (a+2i+1)^{n_1+n_2+n_3-2 i-1}
\end{multline*}
By symmetry, we can assume in the following 
$$n_1 \le n_2 \le n_3.$$ 
The above expression then simplifies to
$$
\prod_{i=\lceil (n_1+n_2-1)/2 \rceil}^{\lfloor (n_1+n_3-1)/2 \rfloor} (a+2i+1)^{2i+1-n_3} 
\prod_{i=\lfloor (n_1+n_3-1)/2 \rfloor+1}^{\lfloor (n_2+n_3-1)/2 \rfloor} (a+2i+1)^{n_1} 
\prod_{i=\lfloor n_2+n_3-1)/2 \rfloor+1}^{\lfloor (n_1+n_2+n_3-1)/2 \rfloor} (a+2i+1)^{n_1+n_2+n_3-2 i-1}.
$$
The degree of this product is 
\begin{multline*} 
\left(-\left\lceil \frac{n_1+n_2-1}{2}
   \right\rceil +\left\lfloor \frac{n_1+n_3-1}{2}
  \right\rfloor +1\right) \left(\left\lceil
   \frac{n_1+n_2-1}{2}
  \right\rceil +\left\lfloor \frac{n_1+n_3-1}{2}
   \right\rfloor -n_3+1\right) \\ +n_1
   \left(\left\lfloor \frac{n_2+n_3-1}{2}
   \right\rfloor -\left\lfloor \frac{n_1+n_3-1}{2}
  \right\rfloor \right) \\ +\left(\left\lfloor
   \frac{n_2+n_3-1}{2}
   \right\rfloor -\left\lfloor \frac{n_1+n_2+n_3-1}{2}
   \right\rfloor \right)
   \left(\left\lfloor \frac{n_1+n_2+n_3-1}{2}
   \right\rfloor +\left\lfloor
   \frac{n_2+n_3-1}{2}
   \right\rfloor
   -n_1-n_2-n_3+2
   \right).
\end{multline*}
Now let $d > 0$. Then we can assume $i_2=i_1+n_1-n_3-d$ (because otherwise the zero is already covered by a smaller $d$) and the relevant zero is $d-n_1-n_2$.
The parity condition is $d \equiv n_1+n_2+1 (\mod 2)$. As $1 \le i_1 \le n_2-d$ and 
$1 \le i_2 \le n_1-d$, we have the following linear factors
\begin{multline*}
 \prod_{d \ge 1 \atop d \equiv n_1+n_2+1 (2)} (a-d+n_1+n_2)^{\max(\min(n_2-d,n_3)-\max(1,1-n_1+n_3+d)+1,0)} \\ = \prod_{i=\lceil (n_1+n_2+n_3-2)/4 \rceil}^{\lfloor (n_1+n_2-2)/2 \rfloor} (a+2i+1)^{4i+2-n_1-n_2-n_3}.
\end{multline*}
The degree of this product is
\begin{multline*}
\left(-\left\lceil \frac{n_1+n_2+n_3-2}{4}
   \right\rceil +\left\lfloor
   \frac{n_1+n_2-2}{2}
   \right\rfloor +1\right) \\
   \times \left(2 \left\lceil
   \frac{n_1+n_2+n_3-2}{4}
   \right\rceil +2 \left\lfloor
   \frac{n_1+n_2-2}{2}
  \right\rfloor
   -n_1-n_2-n_3+2
   \right).
\end{multline*}

In the following lemma, we identify a set of even zeros.
\begin{lem} 
\label{second}
Let $d \ge 0$.  Assuming
$a=i_1-i_2-n_2-n_3, i_1+i_2 \equiv n_2 + n_3 (\mod 2)$, $ n_1+1 \le i_1$ and $-i_1+i_2-n_1+n_3+d \ge 0$, the $i_1$-th row of the $d$-th difference in the top block is equal to the $i_2$-th row of the $d$-th difference in the bottom block, 
provided that $1 \le i_1 \le n_2-d$ and $1 \le i_2 \le n_1-d$.
\end{lem} 

\begin{proof} The proof follows the proof of Lemma~\ref{first}\,up to some point. Note that the assumptions $ n_1+1 \le i_1,-i_1+i_2-n_1+n_3+d \ge 0$ are chosen so that the relevant differences of the entries in the right block are zero.
\end{proof}

For $d=0$, we obtain the following factors.
\begin{multline*}
\prod_{i= \left \lceil (n_1+n_2)/2 \right \rceil}^{\min \left( \left \lfloor (n_2+n_3-n_1-1)/2 \right \rfloor, \left \lfloor (n_1+n_3)/2 \right \rfloor \right)} (a+2i)^{2i-n_3} \prod_{i= \max \left( \left \lceil (n_1+n_2)/2 \right \rceil, \left \lceil (n_2+n_3-n_1)/2 \right \rceil \right)}^{\left \lfloor (n_1+n_3)/2 \right \rfloor} (a+2i)^{n_2-n_1} \\ \times 
\prod_{i= \max \left( \left \lceil (n_2+n_3-n_1)/2 \right \rceil, \left \lceil (n_1+n_3+1)/2 \right \rceil \right)}^{\left \lfloor (n_2+n_3)/2 \right \rfloor} (a+2i)^{n_2+n_3-2i} \prod_{i=\left \lceil (n_1+n_3+1)/2 \right \rceil}^{ \left \lfloor (n_2+n_3-n_1-1)/2 \right \rfloor} (a+2i)^{n_1}
\end{multline*}
The degree of this product is
\begin{multline*}
\left(\max \left(\left\lceil \frac{n_1+n_3+1}{2}
  \right\rceil
   ,\left\lceil \frac{-n_1+n_2+n_3}{2}
   \right\rceil \right)- \left\lfloor \frac{n_2+n_3+2}{2}
   \right\rfloor \right) \\ \times \left(\max
   \left(\left\lceil \frac{n_1+n_3+1}{2}
  \right\rceil
   ,\left\lceil \frac{-n_1+n_2+n_3}{2}
   \right\rceil \right)- \left\lceil
   \frac{n_2+n_3}{2}\right\rceil \right) \\ +\left(\min
   \left(\left\lfloor
   \frac{n_1+n_3}{2}\right\rfloor
   ,\left\lfloor \frac{-n_1+n_2+n_3-1}{2}
   \right\rfloor \right)-\min \left(\left\lceil
   \frac{n_1+n_2-2}{2}
  \right\rceil
   ,\left\lfloor \frac{-n_1+n_2+n_3-1}{2}
  \right\rfloor \right)\right) \\ \times  \left(\min
   \left(\left\lceil
   \frac{n_1+n_2}{2}\right\rceil
   ,\left\lfloor \frac{-n_1+n_2+n_3+1}{2}
   \right\rfloor \right)-\max \left(\left\lceil
   \frac{n_3-n_1}{2}\right\rceil
   ,\left\lceil \frac{n_1-n_2+n_3+1}{2}
   \right\rceil \right)\right) \\ +(n_2-n_1)
   \left(\max \left(\left\lceil \frac{-n_1+n_2+n_3}{2}
   \right\rceil ,\left\lfloor \frac{n_1+n_3+2}{2}
   \right\rfloor
   \right)-\max \left(\left\lceil
   \frac{n_1+n_2}{2}\right\rceil
   ,\left\lceil \frac{-n_1+n_2+n_3}{2}
   \right\rceil \right)\right) \\ +n_1
   \left(\left\lfloor \frac{-n_1+n_2+n_3+1}{2}
   \right\rfloor -\min \left(\left\lceil \frac{n_1+n_3+1}{2}
   \right\rceil
   ,\left\lfloor \frac{-n_1+n_2+n_3+1}{2}
   \right\rfloor \right)\right).
\end{multline*}
For $d > 0$, we obtain 
$$
\prod_{i= \left \lceil (n_1+n_2+n_3)/4  \right \rceil}^{\min\left( \left \lfloor (n_1+n_2-1)/2 \right \rfloor, \left \lfloor (n_2+n_3-n_1-1)/2 \right \rfloor  \right)} (a+2i)^{4i-n_1-n_2-n_3}
\prod_{i= \max \left( \left \lceil (n_2+n_3-n_1)/2 \right \rceil, n_1 \right)}^{\left \lfloor (n_1+n_2-1)/2 \right \rfloor} (a+2i)^{2i-2 n_1}. 
$$
The degree is 
\begin{multline*}
\left(\max \left(\left\lceil \frac{-n_1+n_2+n_3-2}{2}
   \right\rceil ,\left\lfloor \frac{n_1+n_2-1}{2}
  \right\rfloor
   \right)+\max \left(-n_1,\left\lceil
   \frac{-5n_1+n_2+n_3}{2} \right\rceil
   \right)\right) \\ \times \left(\max
   \left(\left\lceil \frac{-n_1+n_2+n_3}{2}
   \right\rceil ,\left\lfloor \frac{n_1+n_2+1}{2}
  \right\rfloor
   \right)-\max \left(n_1,\left\lceil
   \frac{-n_1+n_2+n_3}{2}
  \right\rceil \right)\right) \\ +\left(\min
   \left(\left\lfloor \frac{n_1+n_2+1}{2}
  \right\rfloor
   ,\left\lfloor \frac{-n_1+n_2+n_3+1}{2}
  \right\rfloor \right)-\min \left(\left\lceil
   \frac{n_1+n_2+n_3}{4}
   \right\rceil 
   ,\left\lfloor \frac{-n_1+n_2+n_3+1}{2}
  \right\rfloor \right)\right) \\ \times  \left(2 \min
   \left(\left\lceil \frac{n_1+n_2+n_3}{4}
   \right\rceil 
   ,\left\lfloor \frac{-n_1+n_2+n_3+1}{2}
   \right\rfloor \right)+2 \min \left(\left\lfloor
   \frac{n_1+n_2-1}{2}
  \right\rfloor
   ,\left\lfloor \frac{-n_1+n_2+n_3-1}{2}
   \right\rfloor
   \right) \right. \\ \left.  -n_1-n_2-n_3\right).
\end{multline*}

A related set of even zeros can be obtained as follows. From the proof of 
Lemma~\ref{first}, we know that the $i_1$-th row in the top left block is equal to the $i_2$-th row in the bottom left block, assuming  that $a=i_1-i_2-n_2-n_3$ and $i_1+i_2 \equiv n_2+n_3 \, (\mod 2)$. For each such pair, we can produce a zero row in the left, say, bottom block. Those pairs $(i_1,i_2)$ which were already dealt with in Lemma~\ref{second} (having also zero entries in the right block) are excluded in the following. The dimension of the kernel of the submatrix consisting of the remaining rows is a lower bound for the ``additional'' multiplicity of the zero. If $m$ is the number of these rows, then $m-(n_1+n_2-n_3)$ is a lower bound for this dimension and thus for the additional multiplicity because $n_1+n_2-n_3$ is the number of columns in the right block. It is straightforward to check that this results in the following factors.
$$
\prod_{i=n_2}^{\left \lfloor (n_2+n_3-1)/2 \right \rfloor} (a+2i)^{2i-2 n_2} 
\prod_{i=\left \lceil (n_2+n_3)/2 \right \rfloor}^{n_3} (a+2i)^{2 n_3 - 2 i}
$$
The degree of this product is
$$
\left\lceil \frac{n_2-n_3-2}{2} \right\rceil 
   \left\lceil
   \frac{n_2-n_3}{2}\right\rceil
   +\left\lfloor \frac{-n_2+n_3-1}{2}
  \right\rfloor 
   \left\lfloor \frac{-n_2+n_3+1}{2}
   \right\rfloor.
$$
As for the remaining even zeros, we need the following lemma. It is useful to define the following operator.
$$\sigma_x p(x) = p(x+1)+p(x)$$
We refer to it as the anti-difference.
\begin{lem} 
\label{third}
Let $d \ge 0$. Assuming $a=j_1-j_2-n_3$, $j_1+j_2+n_3 \equiv 0 \, (\mod 2)$, $j_1 \ge n_2-n_1+1$ and $n_2-n_1+j_1+d \ge j_2$, the $j_1$-th column of the $d$-th anti-difference with respect to the columns in the left block is equal to the $j_2$-th column of the $d$-th anti-difference in the right block, provided that $1 \le j_1 \le n_3-d$ and $1 \le j_2 \le n_1+n_2-n_3-d$.
\end{lem}

\begin{proof} The $d$-th anti-difference in the bottom left block is 
$$
\sigma_j^d \binom{-n_1+n_2+n_3+a+i-1}{j-1} = \binom{-n_1+n_2+n_3+a+i-1+d}{j-1+d},
$$
while it is 
$$
\binom{-n_1+n_2+n_3+a+i-1+d}{-n_1+n_2+i-j}
$$
in the right bottom block. If we plug in $a=j_1-j_2-n_3$, then we need to employ the symmetry to show that the expressions are equal, which is possible as long as  $n_2-n_1+j_1+d \ge j_2$.
On the other hand, the $d$-th anti-difference in the top left block is 
$$
\sigma_j^d \binom{i-n_1-1}{j-1} = \binom{i-n_1-1+d}{j-1+d},
$$
while the $d$-th anti-difference in the top right block is 
$$
\sigma_j^d (-1)^{j+n_3-1} \binom{n_1+n_3+a-i+j-1}{n_1-i} = 
(-1)^{j+n_3+d-1} \binom{n_1+n_3+a-i+j-1}{n_1-i-d}.
$$
After plugging in $a=j_1-j_2-n_3$ and applying \eqref{elementary2} to the entry in the top left block, we see that we need to have 
$$
(-1)^{j_1+1+d} \binom{n_1+j_1-i-1}{j_1-1+d} = (-1)^{j_2+n_3+d-1} 
\binom{n_1+j_1-i-1}{n_1-i-d}.
$$
By the symmetry, this is true if $j_1 \ge n_2-n_1+1$ and $j_1+j_2+n_3 \equiv 0 \, (\mod 2)$.
\end{proof}
From this lemma, we can deduce the following linear factors if $d=0$.
\begin{multline*}
\prod_{j_1=\max(1,n_2-n_1+1)}^{n_3} \prod_{j_2=1 \atop j_1+j_2+n_3 \equiv 0 \, (2)}^{\min(n_2-n_1+j_1,n_1+n_2-n_3)} (a-j_1+j_2+n_3) \\
= \prod_{i=1}^{\lfloor (n_1+n_2-n_3)/2 \rfloor} (a+2i)^{2i}  
\prod_{i=\lfloor (n_1+n_2-n_3)/2 \rfloor+1}^{\lfloor (n_1+n_3-n_2)/2 \rfloor}
(a+2i)^{n_1+n_2-n_3}  \prod_{i=\lfloor (n_1+n_3-n_2)/2 \rfloor+1}^{\min(\lfloor (n_2+n_3-n_1)/2 \rfloor, n_1)} (a+2i)^{2 n_1 - 2 i} 
\end{multline*}
The degree of this product is
\begin{multline*}
\left(\min \left(0,\left\lfloor \frac{-3
   n_1+n_2+n_3}{2}
   \right\rfloor \right)+\left\lfloor \frac{-n_1-n_2+n_3+2}{2}
  \right\rfloor \right)  \\ \times \left(\left\lfloor
   \frac{n_1-n_2+n_3}{2}
   \right\rfloor -\min \left(n_1,\left\lfloor
   \frac{-n_1+n_2+n_3}{2}
   \right\rfloor \right)\right) \\ +\left\lfloor
   \frac{n_1+n_2-n_3}{2}
   \right\rfloor  \left\lfloor \frac{n_1+n_2-n_3+2}{2}
  \right\rfloor  \\ +(n_1+n_2-n_3)
   \left(\left\lfloor \frac{n_1-n_2+n_3}{2}
   \right\rfloor -\left\lfloor \frac{n_1+n_2-n_3}{2}
   \right\rfloor \right).
\end{multline*}
For $d > 0$, we obtain 
$$
\prod_{i=\lfloor (n_2+n_3-n_1)/2 \rfloor +1}^{\lfloor (n_1+n_2+n_3)/4 \rfloor} (a+2i)^{n_1+n_2+n_3-4 i}.
$$
The degree of this product is 
\begin{multline*}
\left(\left\lfloor \frac{-n_1+n_2+n_3}{2}
   \right\rfloor -\max \left(\left\lfloor \frac{-n_1+n_2+n_3}{2}
   \right\rfloor ,\left\lfloor \frac{n_1+n_2+n_3}{4}
   \right\rfloor \right)\right) \\ \times \left(2 \max
   \left(\left\lfloor \frac{-n_1+n_2+n_3}{2}
  \right\rfloor ,\left\lfloor \frac{n_1+n_2+n_3}{4}
   \right\rfloor \right)+2 \left\lfloor \frac{-n_1+n_2+n_3+2}{2}
   \right\rfloor
   -n_1-n_2-n_3 \right).
\end{multline*}
The even zeros coming from Lemma~\ref{third} are distinct from those that were identified before: It can be checked that the former factors are of the form $(a+2i)$ with $i \le \min \left(\left\lfloor \frac{n_1+n_2+n_3}{4} \right\rfloor,n_1 \right)$, while the latter are of the form $(a+2i)$ with $i > \min \left(\left\lfloor \frac{n_1+n_2+n_3}{4} \right\rfloor,n_1 \right)$. Now it remains to show that the degrees of the various factors add up to the upper bound for the total degree as computed in Lemma~\ref{degree} and to provide one additional evaluation. The former is a tedious but straightforward computation distinguishing several cases taking into the remainder of the  $n_i$'s modulo $4$ and certain linear inequalities (which can be assisted by a computer algebra system). The additional evaluation is provided by $a=0$ because then the result is equivalent to MacMahon's formula for the number of plane partitions in an $n_1 \times n_2 \times n_3$-box. It is straightforward to check that this leads to the leading coefficient in $a$ displayed in formula (5.1). This completes the proof of Theorem~\ref{newtheo}.

\medskip

\begin{rem} 
\label{remgeneralize}
As mentioned above, Theorem~\ref{newtheo} establishes a counterpart of the main result of \cite{CEKZ}, in which the removed triangle is not in the center of the hexagon, but in a new position (the two positions agree only if $n_1=n_2=n_3$). Using Theorem 1 of \cite{3b}, one can deduce from Theorem~\ref{newtheo} above a more general result --- we can allow, instead of just the triangular hole, to have a more general hole shape, consisting of a triangle with three other triangles of the opposite orientation touching its vertices (what is called a shamrock in \cite{ff}). 

\begin{theo}
\label{newshamrock} Let $SC'_{n_1,n_2,n_3}(a,b,c,m)$ be the region obtained from the hexagon of sides $n_1+a+b+c$, $n_2+m$, $n_3+a+b+c$, $n_1+m$, $n_2+a+b+c$, $n_3+m$ $($clockwise from top$)$ by removing a shamrock of core size $m$ and lobe sizes $a$, $b$ and $c$ $($counterclockwise from top; see Figure $2.1$ of \cite{fv} for an illustration of a shamrock$)$, placed in such a way that the top, left and right lobes are at distances $n_1$, $n_2$ and $n_3$  from the top, southwestern and southeastern sides of the hexagon, respectively. Then we have     
\begin{align}    
&
\frac{\M(SC'_{n_1,n_2,n_3}(a,b,c,m))}{\M(S_{n_1,n_2,n_3,a+b+c+m,0,0,0,0,0,0})}
=
\frac{\h(m)^3\h(a)\h(b)\h(c)}{\h(m+a)\h(m+b)\h(m+c)}
\frac{\h(n_1+a)\h(n_2+n_3-n_1+b+c+m)}{\h(n_1+a+m)\h(n_2+n_3-n_1+b+c)}
\nonumber
\\[10pt]
&\ \ \ \ \ \ \ \ 
\times
\frac{\h(n_2+b)\h(n_1+n_3-n_2+a+c+m)}{\h(n_2+b+m)\h(n_1+n_3-n_2+a+c)}
\frac{\h(n_3+c)\h(n_1+n_2-n_3+a+b+m)}{\h(n_3+c+m)\h(n_1+n_2-n_3+a+b)},
\label{newshamrockf}
\end{align}
where $\h(n)=0!\,1!\cdots(n-1)!$, and the denominator is given by Theorem $\ref{newtheo}$. 
\end{theo}

For the special case of a bowtie this was also conjectured by WonHyok Kim. This will be presented in his master thesis prepared under the supervision of the second author, along with a proof that reduces everything to proving some hypergeometric identities.
\end{rem}

\parindent15pt

\section{A determinantal formula of dimension $a+b_1+b_2+b_3$ for even $b_i$}
\label{withoutnsec}

The purpose of this section is to employ an idea that has been used in Section~\ref{detsec} to derive a determinantal formula for $\m(S_{n_1,n_2,n_3,a,b_1,b_2,b_3,k_1,k_2,k_3})$ when $b_1,b_2,b_3$ are even, such that the underlying matrix is of size $a+b_1+b_2+b_3$. This allows us to reduce the proof of our formula for $\m(S_{n,n,n,a,b,b,b,k,k,k})$ for any concrete even values of $a$ and $b$ to verifying certain hypergeometric identities. This shows, from a different point of view, the advantage of our new method.

\begin{theo} If $a$, $b_1$, $b_2$ and $b_3$ are even, then
\begin{multline}
\label{reduced}
\!\!\!\!\!
  \m(S_{n_1,n_2,n_3,a,b_1,b_2,b_3,k_1,k_2,k_3})=     
\prod_{i=1}^{n_1} \binom{n_2+n_3+a+b_1+b_2+b_3+i-1}{n_2} 
\prod_{j=1}^{a+b_1+b_2+b_3+n_1} \binom{j+n_2-1}{j-1}^{-1}  \\[10pt]
\ \ \ \ \ \ \ \ \ \ \ \ \ \ \ \ \ \ \ \ 
\times
\left|\det \left( \begin{matrix*}[l]   \sum\limits_{l \ge 1} \binom{n_1+\frac{a}{2}+b_1+k_3+i-1}{l+n_2-n_3+2 k_3-1}  \binom{l+n_2-1}{l-1} \binom{-n_3-a-b_1-b_2-b_3}{j+n_1-l}&  \quad {1 \le i \le b_3} \\ \sum\limits_{l \ge 1} \binom{n_1+b_1+i-1}{l+n_2-n_3-b_3-1} \binom{l+n_2-1}{l-1} \binom{-n_3-a-b_1-b_2-b_3}{j+n_1-l} & \quad {1 \le i \le a} \\ \sum\limits_{l \ge 1} \binom{n_1 -2 k_1+i-1}{l+n_2-n_3-\frac{a}{2}-b_3-k_1-1} \binom{l+n_2-1}{l-1} \binom{-n_3-a-b_1-b_2-b_3}{j+n_1-l} & \quad {1 \le i \le b_1} \\ \sum\limits_{l \ge 1} \binom{n_1+\frac{a}{2}+b_1+k_2+i-1}{l+n_2-n_3-\frac{a}{2}-b_3-k_2-1} \binom{l+n_2-1}{l-1} \binom{-n_3-a-b_1-b_2-b_3}{j+n_1-l} & \quad {1 \le i \le b_2} \end{matrix*} \right)_{1 \le j \le a+b_1+b_2+b_3}\right|.
\end{multline}
\end{theo}

\begin{proof}
The starting point is again \eqref{op}, where we now eliminate the top two blocks. By setting $d=1$ and plugging in $c_{1,i}=i$, we see that 
$\m(S_{n_1,n_2,n_3,a,b_1,b_2,b_3,k_1,k_2,k_3})$ is equal to 
\begin{equation*}
\left|\,
\prod_{i=1}^{b_3} \fd_{c_{3,i}}^{n_3-2 k_3} \prod_{i=1}^{a} \fd_{c_{4,i}}^{n_3+b_3} \prod_{i=1}^{b_1} \fd_{c_{5,i}}^{n_3+\frac{a}{2}+b_3+k_1} \prod_{i=1}^{b_2} \fd_{c_{6,i}}^{n_3+\frac{a}{2}+b_3+k_2} 
\det \left( \begin{matrix*}[l]  \binom{i-1}{j-1} & \quad {1 \le i \le n_2} \\ \binom{c_{2,i}-1}{j-1} & \quad {1 \le i \le n_1}  \\ \binom{c_{3,i}-1}{j-1} &  \quad {1 \le i \le b_3} \\ \binom{c_{4,i}-1}{j-1} & \quad {1 \le i \le a} \\ \binom{c_{5,i}-1}{j-1} & \quad {1 \le i \le b_1} \\ \binom{c_{6,i}-1}{j-1} & \quad {1 \le i \le b_2} \end{matrix*} \right)_{1 \le j \le n}
\,\right|,
\tag{6.2}
\end{equation*}
where $c_{2,i}=n_2+n_3+a+b_1+b_2+b_3+i$, $c_{3,i}=n_1+\frac{a}{2}+b_1+k_3+i$, $c_{4,i}=n_1+b_1+i$, $c_{5,i}=n_1 -2 k_1+i$, $c_{6,i}=n_1+\frac{a}{2}+b_1+k_2+i$ and $n=n_1+n_2+a+b_1+b_2+b_3$.
The only non-zero entry in the first row is in the first column, and so we expand with respect to this row. After having performed this reduction, the new first row has the same property. We can keep expanding until we have deleted the top block. We then set $c_{2,i}=n_2+n_3+a+b_1+b_2+b_3+i$ and obtain that the expression whose absolute value is taken in (6.2) equals
$$
\prod_{i=1}^{b_3} \fd_{c_{3,i}}^{n_3-2 k_3} \prod_{i=1}^{a} \fd_{c_{4,i}}^{n_3+b_3} \prod_{i=1}^{b_1} \fd_{c_{5,i}}^{n_3+\frac{a}{2}+b_3+k_1} \prod_{i=1}^{b_2} \fd_{c_{6,i}}^{n_3+\frac{a}{2}+b_3+k_2} 
\det \left( \begin{matrix*}[l]  \binom{n_2+n_3+a+b_1+b_2+b_3+i-1}{j-1} & \quad {1 \le i \le n_1}  \\ \binom{c_{3,i}-1}{j-1} &  \quad {1 \le i \le b_3} \\ \binom{c_{4,i}-1}{j-1} & \quad {1 \le i \le a} \\ \binom{c_{5,i}-1}{j-1} & \quad {1 \le i \le b_1} \\ \binom{c_{6,i}-1}{j-1} & \quad {1 \le i \le b_2} \end{matrix*} \right)_{n_2+1 \le j \le n}.
$$
Shifting $j$, this becomes
$$
\prod_{i=1}^{b_3} \fd_{c_{3,i}}^{n_3-2 k_3} \prod_{i=1}^{a} \fd_{c_{4,i}}^{n_3+b_3} \prod_{i=1}^{b_1} \fd_{c_{5,i}}^{n_3+\frac{a}{2}+b_3+k_1} \prod_{i=1}^{b_2} \fd_{c_{6,i}}^{n_3+\frac{a}{2}+b_3+k_2} 
\det \left( \begin{matrix*}[l]  \binom{n_2+n_3+a+b_1+b_2+b_3+i-1}{j+n_2-1} & \quad {1 \le i \le n_1}  \\ \binom{c_{3,i}-1}{j+n_2-1} &  \quad {1 \le i \le b_3} \\ \binom{c_{4,i}-1}{j+n_2-1} & \quad {1 \le i \le a} \\ \binom{c_{5,i}-1}{j+n_2-1} & \quad {1 \le i \le b_1} \\ \binom{c_{6,i}-1}{j+n_2-1} & \quad {1 \le i \le b_2} \end{matrix*} \right)_{1 \le j \le n-n_2}.
$$
Taking out the factor $\binom{n_2+n_3+a+b_1+b_2+b_3+i-1}{n_2}$  from row $i$, $1 \le i \le n_1$, as well as the factor $\binom{j+n_2-1}{j-1}^{-1}$ from column $j$, $1 \le j \le n-n_2$, this is equal to 
\begin{multline}
\label{factorout}
\prod_{i=1}^{n_1} \binom{n_2+n_3+a+b_1+b_2+b_3+i-1}{n_2} 
\prod_{j=1}^{n-n_2} \binom{j+n_2-1}{j-1}^{-1} \\
\times \prod_{i=1}^{b_3} \fd_{c_{3,i}}^{n_3-2 k_3} \prod_{i=1}^{a} \fd_{c_{4,i}}^{n_3+b_3} \prod_{i=1}^{b_1} \fd_{c_{5,i}}^{n_3+\frac{a}{2}+b_3+k_1} \prod_{i=1}^{b_2} \fd_{c_{6,i}}^{n_3+\frac{a}{2}+b_3+k_2} \\
\det \left( \begin{matrix*}[l]  \binom{n_3+a+b_1+b_2+b_3+i-1}{j-1}  & \quad {1 \le i \le n_1}  \\ \binom{c_{3,i}-1}{j+n_2-1}  \binom{j+n_2-1}{j-1}&  \quad {1 \le i \le b_3} \\ \binom{c_{4,i}-1}{j+n_2-1} \binom{j+n_2-1}{j-1} & \quad {1 \le i \le a} \\ \binom{c_{5,i}-1}{j+n_2-1} \binom{j+n_2-1}{j-1}  & \quad {1 \le i \le b_1} \\ \binom{c_{6,i}-1}{j+n_2-1} \binom{j+n_2-1}{j-1} & \quad {1 \le i \le b_2} \end{matrix*} \right)_{1 \le j \le n-n_2}.
\end{multline}

Now we multiply the matrix underlying the determinant in \eqref{factorout} on the right by $\left(\binom{-n_3-a-b_1-b_2-b_3}{j-i}\right)_{1 \le i,j \le n-n_2}$, which is a matrix that has determinant $1$. We obtain that expression \eqref{factorout} is equal to
\begin{multline}
\label{matrixmultiplied}
\prod_{i=1}^{n_1} \binom{n_2+n_3+a+b_1+b_2+b_3+i-1}{n_2} 
\prod_{j=1}^{n-n_2} \binom{j+n_2-1}{j-1}^{-1} \\
\times \prod_{i=1}^{b_3} \fd_{c_{3,i}}^{n_3-2 k_3} \prod_{i=1}^{a} \fd_{c_{4,i}}^{n_3+b_3} \prod_{i=1}^{b_1} \fd_{c_{5,i}}^{n_3+\frac{a}{2}+b_3+k_1} \prod_{i=1}^{b_2} \fd_{c_{6,i}}^{n_3+\frac{a}{2}+b_3+k_2} \\
\det \left( \begin{matrix*}[l]  \binom{i-1}{j-1}  & \quad {1 \le i \le n_1}  \\ \sum\limits_{l \ge 1} \binom{c_{3,i}-1}{l+n_2-1}  \binom{l+n_2-1}{l-1} \binom{-n_3-a-b_1-b_2-b_3}{j-l}&  \quad {1 \le i \le b_3} \\ \sum\limits_{l \ge 1} \binom{c_{4,i}-1}{l+n_2-1} \binom{l+n_2-1}{l-1} \binom{-n_3-a-b_1-b_2-b_3}{j-l} & \quad {1 \le i \le a} \\ \sum\limits_{l \ge 1} \binom{c_{5,i}-1}{l+n_2-1} \binom{l+n_2-1}{l-1} \binom{-n_3-a-b_1-b_2-b_3}{j-l} & \quad {1 \le i \le b_1} \\ \sum\limits_{l \ge 1} \binom{c_{6,i}-1}{l+n_2-1} \binom{l+n_2-1}{l-1} \binom{-n_3-a-b_1-b_2-b_3}{j-l} & \quad {1 \le i \le b_2} \end{matrix*} \right)_{1 \le j \le n-n_2}.
\end{multline}
We now eliminate the top block as before, apply the difference operators and specialize the $c_{t,i}$'s. Then expression \eqref{matrixmultiplied} becomes
\begin{multline*}
\prod_{i=1}^{n_1} \binom{n_2+n_3+a+b_1+b_2+b_3+i-1}{n_2} 
\prod_{j=1}^{n-n_2} \binom{j+n_2-1}{j-1}^{-1}  \\
\times \det \left( \begin{matrix*}[l]   \sum\limits_{l \ge 1} \binom{n_1+\frac{a}{2}+b_1+k_3+i-1}{l+n_2-n_3+2 k_3-1}  \binom{l+n_2-1}{l-1} \binom{-n_3-a-b_1-b_2-b_3}{j-l}&  \quad {1 \le i \le b_3} \\ \sum\limits_{l \ge 1} \binom{n_1+b_1+i-1}{l+n_2-n_3-b_3-1} \binom{l+n_2-1}{l-1} \binom{-n_3-a-b_1-b_2-b_3}{j-l} & \quad {1 \le i \le a} \\ \sum\limits_{l \ge 1} \binom{n_1 -2 k_1+i-1}{l+n_2-n_3-\frac{a}{2}-b_3-k_1-1} \binom{l+n_2-1}{l-1} \binom{-n_3-a-b_1-b_2-b_3}{j-l} & \quad {1 \le i \le b_1} \\ \sum\limits_{l \ge 1} \binom{n_1+\frac{a}{2}+b_1+k_2+i-1}{l+n_2-n_3-\frac{a}{2}-b_3-k_2-1} \binom{l+n_2-1}{l-1} \binom{-n_3-a-b_1-b_2-b_3}{j-l} & \quad {1 \le i \le b_2} \end{matrix*} \right)_{n_1+1 \le j \le n-n_2}.
\end{multline*}
Shifting $j$ we arrive at formula \eqref{reduced} in the statement. \end{proof}

\section{The leading coefficient in $a$ for even $b$}
\label{leadingsec}

\begin{lem}
\label{teb}
For $a$ and $b$ even we have
\begin{equation}
\M(S_{n,a,b,k})=\det(\overline{I}+\overline{B})\det(\xi \overline{I}+\overline{B})\det(\xi^2\overline{I}+\overline{B}),
\label{eeb}
\end{equation}
where $\xi=-\frac{1}{2}+\frac{\sqrt{3}}{2}$ is a cubic root of unity, and $\overline{I}$ and $\overline{B}$ are $(n+2b)\times(n+2b)$ matrices given by
\begin{equation}
\renewcommand{\arraystretch}{1.6}
\overline{I}=
\left[
\begin{array}{c|c}
 {I}_{n+b} & O_{n+b,b}\\
\hline
 O_{b,n+b} & O_{b,b}
\end{array}
\right]
\label{eec}
\end{equation}
$($where $I_m$ stands for the order $m$ identity matrix, and $O_{m,p}$ for the $m\times p$ zero matrix$)$ and 
\begin{equation}
\renewcommand{\arraystretch}{2.2}
\overline{B}=
\left[
\begin{array}{c|c}
\left({\displaystyle {a+i+j-2 \choose j-1}}    \right)_{1\leq i,j\leq n+b}  & \left({\displaystyle {\frac{a}{2}+k+i-1 \choose 2k+j-1}}    \right)_{1\leq i\leq n+b \atop 1 \leq j\leq b}\\
\hline
\left({\displaystyle {n+a+b+j-1 \choose j-i}}    \right)_{1\leq i\leq b \atop 1 \leq j\leq n+b}  & \left({\displaystyle {n+\frac{a}{2}+b+k \choose 2k+j-i}}    \right)_{1\leq i,j\leq b}
\end{array}
\right].
\label{eed}
\end{equation}

\end{lem}

\begin{proof} 
The case $b=0$ was proved at the beginning of Section 10 in \cite{CEKZ}. The general case follows by precisely the same arguments.
\end{proof}

Since all entries of the matrix $\overline{B}$ are polynomials in $a$, it follows by Lemma \ref{teb} that, for any fixed $n$, $b$ and $k$, $\M(S_{n,a,b,k})$ is a polynomial in $a$. The purpose of this section is to determine the degree and the leading coefficient of this polynomial. This is accomplished in the following result.

\begin{prop}
\label{tea}
For $b$ even, the degree of $\M(S_{2n,2a,b,k})$ regarded as a polynomial in $a$ is \linebreak $3(n^2+2bk)$, and its leading coefficient is
\begin{equation}
\left\{
\frac{1}{2^{n^2-n+2k}}
\frac{1}{\left(\frac{b}{2}+n-k+\frac12\right)_k \left(\frac12\right)_{n-k}}
\left[\prod_{i=1}^{n-k-1}\frac{1}{\left(\frac12\right)_{i}}
\prod_{i=1}^{k}\frac{1}{\left(\frac12\right)_i(2i)_{b-1}\left(i+\frac{b-1}{2}\right)_{n-k}}
\right]^2
\right\}^3.
\label{eea}
\end{equation}
%
\end{prop}

\begin{proof} The first factor on the right hand side of \eqref{eeb} can be written\footnote{ As is customary, we write $[n]$ for the set $\{1,2,\dotsc,n\}$, and $A_I^J$ for the submatrix of $A$ obtained by selecting the rows with indices in $I$ and the columns with indices in $J$.} as
\begin{equation}
\det(\overline{I}+\overline{B})=\sum_{I\subset[n+b]} \det\left(\overline{B}_{I\cup\{n+b+1,\dotsc,n+2b\}}^{I\cup\{n+b+1,\dotsc,n+2b\}}\right)
\label{eee}
\end{equation}
(in other words, $\det(\overline{I}+\overline{B})$ is equal to the sum of the determinants of all principal minors of $\overline{B}$ which contain the last $b$ rows and columns). To see this, regard each of the first $n+b$ columns of $\overline{I}+\overline{B}$ as being the sum of the corresponding column of  $\overline{I}$ with the corresponding column of  $\overline{B}$, and use the fact that the determinant is a linear function. 

By the same argument, we also have
\begin{equation}
\det(\xi\overline{I}+\overline{B})=\sum_{I\subset[n+b]} \xi^{|I|}\det\left(\overline{B}_{I\cup\{n+b+1,\dotsc,n+2b\}}^{I\cup\{n+b+1,\dotsc,n+2b\}}\right)
\label{eef}
\end{equation}
and
\begin{equation}
\det(\xi^2\overline{I}+\overline{B})=\sum_{I\subset[n+b]} \xi^{2|I|}\det\left(\overline{B}_{I\cup\{n+b+1,\dotsc,n+2b\}}^{I\cup\{n+b+1,\dotsc,n+2b\}}\right).
\label{eeg}
\end{equation}

Since all entries of $\overline{B}$ are polynomials in $a$, it follows that each summand in \eqref{eee} is also so. By Lemma \ref{tec}, for the summand corresponding to the index set $I=\{i_1,\dotsc,i_s\}\subset[n+b]$, the degree in $a$ satisfies
\begin{equation}
\deg_a\det\left(\overline{B}_{I\cup\{n+b+1,\dotsc,n+2b\}}^{I\cup\{n+b+1,\dotsc,n+2b\}}\right)\leq(i_1-1)+(i_2-2)+\cdots+(i_s-s)+2bk-bs.
\label{eeh}
\end{equation}

For fixed $s$, the bound on the right hand side of \eqref{eeh} attains its maximum only for $\{i_1,\dotsc,i_s\}=\{n+b-s+1,\dotsc,n+b\}$, when it is readily seen to equal $2bk+s(n-s)$. In turn, for even $n$ (note that the $n$-parameter in the statement of Proposition \ref{tea} is even), this is maximum only for $s=n/2$. Therefore, for all subsets $\{i_1,\dotsc,i_s\}\subset[n+b]$, we have
\begin{equation}
(i_1-1)+(i_2-2)+\cdots+(i_s-s)+2bk-bs\leq 2bk+\left(\frac{n}{2}\right)^2,
\label{eei}
\end{equation}
with equality only if $s=n/2$ and $\{i_1,\dotsc,i_{n/2}\}=\{n/2+b+1,\dotsc,n+b\}$.
Therefore, by equation \eqref{eee}, it follows that $\det(\overline{I}+\overline{B})$ has degree at most $2bk+\left(\frac{n}{2}\right)^2$.

However, by the arguments in the proof of Lemma \ref{teb}, for even $a$ and $b$ we have that 
\begin{equation}
\M_r(S_{n,a,b,k})=\det(\overline{I}+\overline{B}).
\label{eej}
\end{equation}
Thus, as a simple calculation shows, it follows from Theorem \ref{tbb} that the degree of $\det(\overline{I}+\overline{B})$ is actually {\it equal} to $2bk+\left(\frac{n}{2}\right)^2$. This can only happen if, for the unique index set $I_0$ for which equality is attained in \eqref{eei}, we have in fact that the degree in $a$ of $\det\left(B_{I_0\cup\{n+b+1,\dotsc,n+2b\}}^{I_0\cup\{n+b+1,\dotsc,n+2b\}}\right)$ is equal to $2bk+\left(\frac{n}{2}\right)^2$. In each of the sums on the right hand side in \eqref{eee}--\eqref{eeg}, the term corresponding to this index $I_0$ has degree $2bk+\left(\frac{n}{2}\right)^2$, while the degree in $a$ of all remaining terms is strictly smaller. This implies that  $\det(\xi^i \overline{I_0}+\overline{B})$ has degree equal to $2bk+\left(\frac{n}{2}\right)^2$, for $i=1,2,3$. The claim about the degree of the leading term in Proposition \ref{tea} follows then from the factorization of Lemma \ref{teb}. 

By Lemma \ref{teb}, for $a$ and $b$ even, the leading coefficient of $\M(S_{n,a,b,k})$ (when regarded as a polynomial in $a$) is equal to the product of the leading coefficients of the three factors. By formulas \eqref{eef}--\eqref{eeg}, for $i=1,2$, the leading coefficient in $\det(\xi^i\overline{I}+\overline{B})$ is equal to $\xi^{i|I_0|}$ times the leading coefficient in  $\det(\overline{I}+\overline{B})$. The latter can be read off directly from Theorem \ref{tbb}, and since $\xi\xi^2=1$, the claim about the leading coefficient in Proposition \ref{tea} follows.
\end{proof}

\begin{figure}[h]
\centerline{
\hfill
{\includegraphics[width=0.60\textwidth]{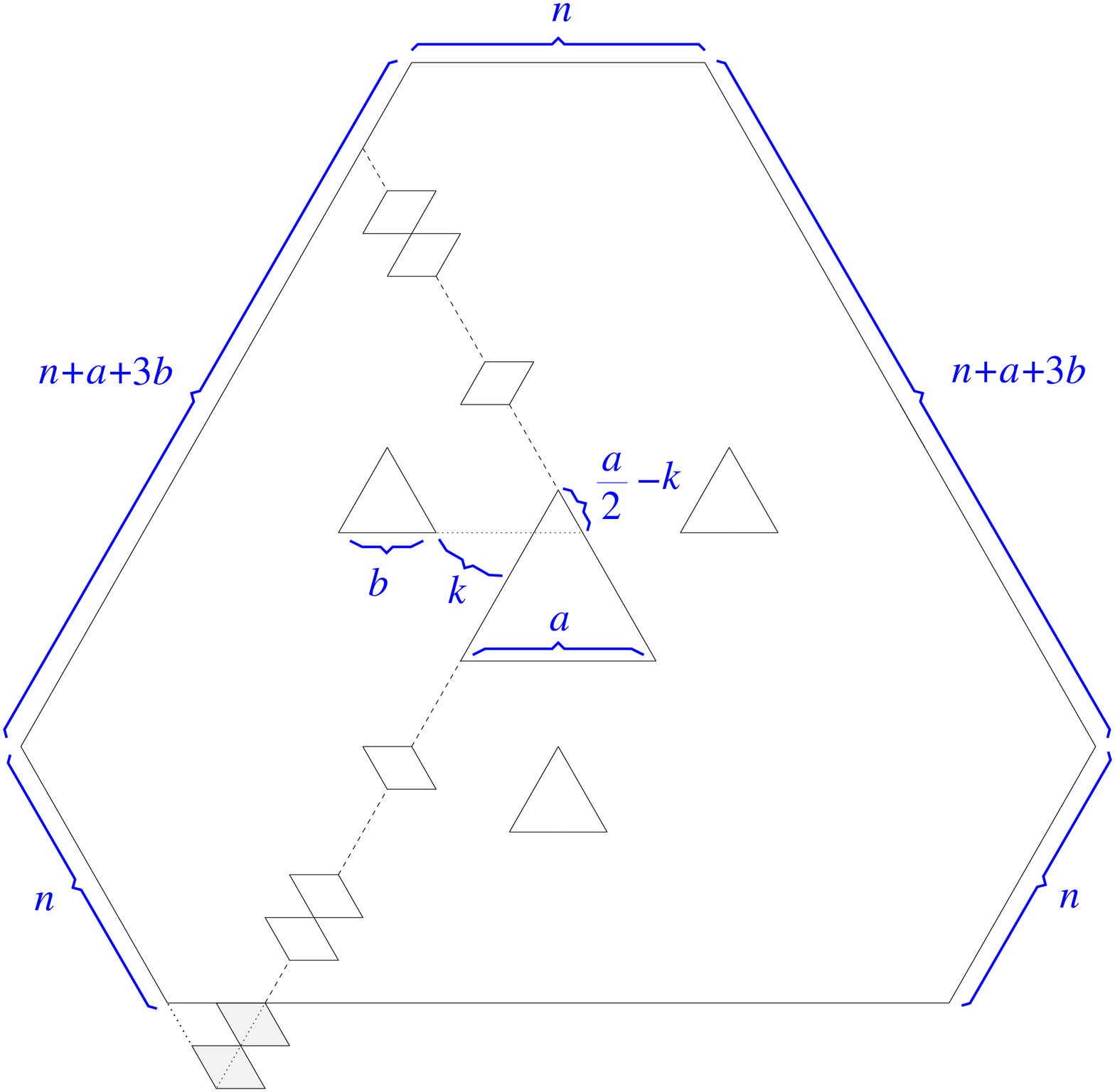}}
\hfill
}
\caption{\label{fea} Obtaining equations \eqref{eel}--\eqref{een} (illustrated on $S_{n,a,b,k}$ for $n=6$, $a=4$, $b=2$, $k=1$.}
\end{figure}

\begin{lem}
\label{tec}
Let $b$ be even. If $\overline{I}$ and $\overline{B}$ are given by \eqref{eec} and \eqref{eed}, for any index set $I=\{i_1,\dotsc,i_s\}\subset[n+b]$, the degree in $a$ of the polynomial $\det\left(\overline{B}_{I\cup\{n+b+1,\dotsc,n+2b\}}^{I\cup\{n+b+1,\dotsc,n+2b\}}\right)$ satisfies the inequality
\begin{equation}
\deg_a \det\left(\overline{B}_{I\cup\{n+b+1,\dotsc,n+2b\}}^{I\cup\{n+b+1,\dotsc,n+2b\}}\right)\leq(i_1-1)+(i_2-2)+\cdots+(i_s-s)+2bk-bs.
\label{eek}
\end{equation}

\end{lem}

\begin{proof}
Let $R_I$ be the region obtained from the fundamental region\footnote{ Under the action of rotation by 120 degrees.} of $S_{n,a,b,k}$ determined by the dashed rays in Figure \ref{fea} by removing the lozenges that straddle those rays and are at distances $i_1,\dotsc,i_s$ from the core (Figure \ref{fea} illustrates this for $I=\{3,6,7\}$). Encoding the lozenge tilings of $R_I$ by families of non-intersecting paths of lozenges that connect the horizontal unit segments on the boundary of  $R_I$ (including the $b$ such unit segments on the bottom of the left satellite), we obtain by applying the Gessel-Viennot theorem (and using that $b$ is even) that
\begin{equation}
\det\left(\overline{B}_{I\cup\{n+b+1,\dotsc,n+2b\}}^{I\cup\{n+b+1,\dotsc,n+2b\}}\right)
=
\M(R_I).
\label{eel}
\end{equation}

If we instead encode the lozenge tilings of $R_I$ by families of non-intersecting paths of lozenges that connect the {\it southwest/northeast facing} unit segments of the boundary of  $R_I$, and perform Laplace expansion in the determinant of the resulting Gessel-Viennot matrix $G$ along the rows corresponding to the $b$ starting segments on the northeastern side of the satellite, we claim that we obtain
\begin{equation}
\M(R_I)
=
\pm\sum_{\{i'_1<\cdots <i'_b\}\subset[n+b]\setminus I} (-1)^{i'_1+\cdots i'_b} \M(H_{I'})\M(R_{I,I'}),
\label{eem}
\end{equation}
where $I'=\{i'_1,\dotsc,i'_b\}$, and $H_{I'}$ and $R_{I,I'}$ are the regions described as follows. $H_{I'}$ is the region ``spanned'' by the $b$ unit segments on the northeastern side of the left satellite and the $b$ unit segments of the top dashed ray that are at distances $i'_1,\dotsc,i'_b$ from the core --- i.e., the region consisting of the union of all possible paths of lozenges that start at the former and end at the latter $b$ unit segments. $R_{I,I'}$ is the region obtained from $R_I$ by filling back in the satellite hole, and making $b$ more dents along the top dashed ray, at distances $i'_1,\dotsc,i'_b$ from the core (so $R_{I,I'}$ has $s$ dents along the bottom dashed ray, and $s+b$ dents along the top dashed ray).

Indeed, the described Laplace expansion yields first an equality like \eqref{eem} with the two tiling counts in the summand replaced by the determinants of two complementary submatrices of $G$. However, these submatrices are in their turn Gessel-Viennot matrices, and it is not hard to see that they correspond precisely to the above defined regions $H_{I'}$ and $R_{I,I'}$, when their tilings are encoded by families of non-intersecting paths of lozenges that connect the southwest/northeast facing unit segments of their boundary. Therefore, by the Gessel-Viennot theorem, each of the two determinants is equal to the corresponding tiling count, yielding \eqref{eem}.

\begin{figure}[h]
\centerline{
\hfill
{\includegraphics[width=0.60\textwidth]{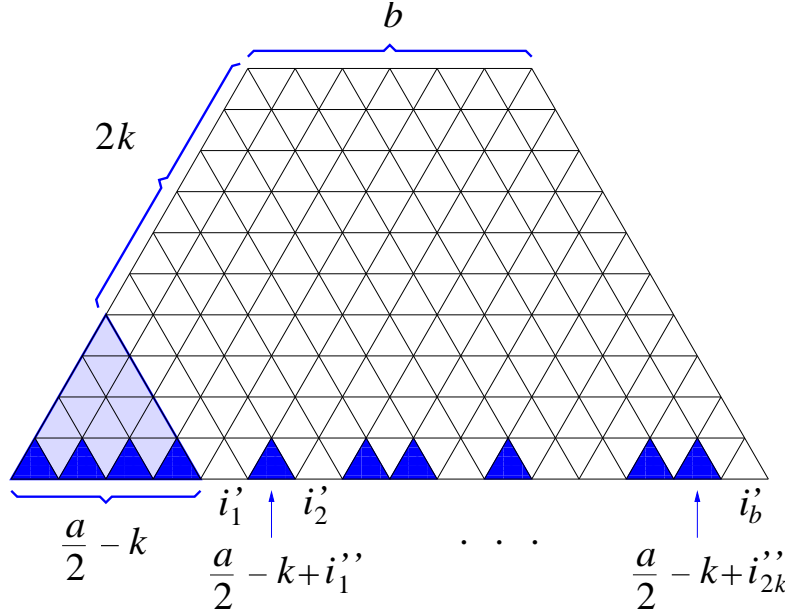}}
\hfill
}
\caption{\label{feb} The region $H_{I'}$  for $a=14$, $b=6$, $k=3$ and $I'=\{1,2,6,8,9,12\}$.}
\end{figure}

Now switch again the direction of the paths, and encode the tilings of $R_{I,I'}$ by paths of lozenges that connect the horizontal unit segments of its boundary. Things simplify if we replace $R_{I,I'}$ by $R'_{I,I'}$, the region obtained from $R_{I,I'}$ by adding a down-pointing dented triangle of side $b$ along its base, with $b$ dents along its southeastern side (see Figure \ref{fea}). Then clearly $\M(R_{I,I'})=\M(R'_{I,I'})$, and the mentioned encoding gives
\begin{equation}
\M(R_{I,I'})=\M(R'_{I,I'})
=
\det\left(B(n+b)_{I\cup\{n+b+1,\dotsc,n+2b\}}^{I\cup {I}'}\right),
\label{een}
\end{equation}
where $B(n)$ is the matrix
\begin{equation}
B(n):=
\left(
{a+i+j-2 \choose j-1}
\right)_{1\leq i,j\leq n}.
\label{eeo}
\end{equation}
%


A detailed picture of the region $H_{I'}$ --- provided $a/2\geq k$ --- is showed in Figure \ref{feb}. 
By Lemma~\ref{ted}, if we set\footnote{ We are using here the fact that $i'_b\leq 2k+b$. This is so because the unit segment at which a path of lozenges starting from the northeastern side of the left satellite crosses the top dashed ray is at distance at most $2k+b$ from the core.} $\{i''_1,\dotsc,i''_{2k}\}:=[2k+b]\setminus\{i'_1,\dotsc,i'_b\}$, we have 
\begin{equation}
\deg_a \M(H_{I'})=(i''_1-1)+(i''_2-2)+\cdots+(i''_{2k}-2k).
\label{eep}
\end{equation}

On the other hand, by Lemma \ref{tee} we have
\begin{equation}
\deg_a \det \left(B(n+b)_{I\cup\{n+b+1,\dotsc,n+2b\}}^{I\cup {I}'}\right)
\leq
(i_1-1)+\cdots+(i_s-s)+(i'_1-(s+1))+\cdots+(i'_b-(s+b)).
\label{eeq}
\end{equation}

By equations \eqref{eel}--\eqref{een}, \eqref{eep} and \eqref{eeq}, we obtain that for integers $a$ with $a/2\geq k$, the values of $\det\left(\overline{B}_{I\cup\{n+b+1,\dotsc,n+2b\}}^{I\cup\{n+b+1,\dotsc,n+2b\}}\right)$ depend polynomially on $a$, as a polynomial of degree less or equal than
\begin{align}
&
(i''_1-1)+(i''_2-2)+\cdots+(i''_{2k}-2k)+
(i_1-1)+\cdots+(i_s-s)
\nonumber
\\
&+(i'_1-(s+1))+\cdots+(i'_b-(s+b)).
\label{eer}
\end{align}
Since by definition $\{i'_1,\dotsc,i'_s,i''_1,\dotsc,i''_{2k}\}=\{1,\dotsc,2k+b\}$, the sum on the right hand side above is readily seen to be equal to $(i_1-1)+\cdots+(i_s-s)+2bk-bs$. However, as the above mentioned polynomial agrees with $\det\left(\overline{B}_{I\cup\{n+b+1,\dotsc,n+2b\}}^{I\cup\{n+b+1,\dotsc,n+2b\}}\right)$ (which is a polynomial in $a$ due to the fact that all its entries are so) on an infinite set of values (namely, all integers $a$ with $a/2\geq k$), it follows that they are identical, and thus \eqref{eek} holds.
\end{proof}

\begin{lem}
\label{ted}
Let $I'=\{i'_1,\dotsc,i'_b\}$, $1\leq i'_1<\cdots<i'_b\leq 2k+b$, and set $\{i''_1,\dotsc,i''_{2k}\}:=[2k+b]\setminus\{i'_1,\dotsc,i'_b\}$. Then the number of lozenge tilings of the region $H_{I'}$ shown in Figure $\ref{feb}$ is a polynomial in $a$ of degree
\begin{equation}
\deg_a \M(H_{I'})=(i''_1-1)+(i''_2-2)+\cdots+(i''_{2k}-2k).
\label{ees}
\end{equation}

\end{lem}

\begin{proof}
We use the classical fact \cite{CLP,GT} that the number of lozenge tilings of the trapezoid $T_m(x_1,\dotsc,x_n)$ of base length $m$, sides of length $n$, and with unit triangular dents on its base at positions $1\leq x_1<\cdots<x_n\leq m$, is given by
\begin{equation}
\M(T_m(x_1,\dotsc,x_n))=\frac{\Delta(x_1,\dotsc,x_n)}{\Delta(1,\dotsc,n)},
\label{eet}
\end{equation}
where
\begin{equation}
\Delta(x_1,\dotsc,x_n):=\prod_{1\leq i<j\leq n} (x_j-x_i).
\label{eeu}
\end{equation}

Our region $H_{I'}$ (see Figure \ref{feb}) is obtained from the region
\begin{equation}
T_{a/2+k+b}(1,\dotsc,a/2-k,a/2-k+i''_1,\dotsc,a/2-k+i''_{2k})
\nonumber
\end{equation}
by removing the lozenges forced by the $a/2-k$ initial contiguous dents (this effectively removes a triangle of side $a/2-k$ from the left corner of the trapezoid).

Therefore, by \eqref{eet} 
we have
\begin{equation}
\M(H_{I'})=\frac{\Delta(1,\dotsc,a/2-k,a/2-k+i''_1,\dotsc,a/2-k+i''_{2k})}
{\Delta(1,\dotsc,a/2-k,a/2-k+1,\dotsc,a/2-k+{2k})}.
\label{eev}
\end{equation}

Clearly, one can write\footnote{ Here $[n]$ denotes the sequence of integers $1,2,\dotsc,n$.}
\begin{equation}
\frac{\Delta([n],n+i_1,\dotsc,n+i_l)}{\Delta([n],n+1,\dotsc,n+l)}
=
\frac{\Delta([n],n+i_1)}{\Delta([n],n+1)}
\frac{\Delta([n],n+i_2)}{\Delta([n],n+2)}
\cdots
\frac{\Delta([n],n+i_l)}{\Delta([n],n+l)}
\frac{\Delta(n+i_1,\dotsc,n+i_l)}{\Delta(n+1,\dotsc,n+l)}.
\label{eew}
\end{equation}
One readily verifies that
\begin{equation}
\frac{\Delta([n],n+t)}{\Delta([n],n+i)}
=
\frac{(n+i)_{t-i}}{(i)_{t-i}}.
\label{eex}
\end{equation}
Apply \eqref{eew} to the right hand side of  \eqref{eev}, replacing $n$ by $a/2-k$, $l$ by $2k$ and $i_j$ by $i''_j$. By \eqref{eex}, for $j=1,\dotsc,2k$, the $j$th resulting fraction in the product is
\begin{equation}
\frac{(a/2-k+j)_{i''_j-j}}{(j)_{i''_j-j}},
\nonumber
\end{equation}
and has therefore degree $i''_j-j$ in $a$. Since the last fraction in the product --- which comes from the last fraction on the right hand side of \eqref{eew} --- is a constant (as a polynomial in $a$), we obtain for $\M(H_{I'})$ an explicit expression as a product of linear factors in $a$, having the degree specified on the right hand side of \eqref{ees}. 
This completes the proof. \end{proof}

\begin{lem}
\label{tee}
Let $B(n)$ be the matrix given by \eqref{eeo}. Then for any $I,J\subset[n]$, $|I|=|J|=s$, $J=\{j_1<j_2<\cdots<j_s\}$, the degree of $\det B(n)_I^J$ as a polynomial in $a$ satisfies
\begin{equation}
\deg_a \det B(n)_I^J \leq(j_1-1)+(j_2-2)+\cdots+(j_{s}-s).
\label{eez}
\end{equation}

\end{lem}

\begin{proof}
Consider the multivariate generalization $\tilde{B}(n)$ obtained from $B(n)$ by replacing $a$ by $a_i$ in all entries of row $i$, for $i=1,\dotsc,n$. Then we have
\begin{equation}
\renewcommand{\arraystretch}{2.3}
\tilde{B}(n)_I^J=
\begin{bmatrix}
{\displaystyle {a_{i_1}+i_1+j_1-2\choose j_1-1}} & {\displaystyle {a_{i_1}+i_1+j_2-2\choose j_2-1}} & . & . & . &
{\displaystyle {a_{i_1}+i_1+j_s-2\choose j_s-1}}\\
{\displaystyle {a_{i_2}+i_2+j_1-2\choose j_1-1}} & {\displaystyle {a_{i_2}+i_2+j_2-2\choose j_2-1}} & . & . & . &
{\displaystyle {a_{i_2}+i_2+j_s-2\choose j_s-1}}\\

.                &    .                        &   &   &   &   .              \\
.                &    .                        &   &   &   &   .              \\
.                &    .                        &   &   &   &   .              \\
{\displaystyle {a_{i_s}+i_s+j_1-2\choose j_1-1}} & {\displaystyle {a_{i_s}+i_s+j_2-2\choose j_2-1}} & . & . & . &
{\displaystyle {a_{i_s}+i_s+j_s-2\choose j_s-1}}\\
\end{bmatrix}.
\nonumber
\end{equation}

All entries in column $k$ are polynomials in $a_{i_1},a_{i_2},\dotsc,a_{i_s}$ of total degree $j_k-1$. It follows that in the expansion of $\det \tilde{B}(n)_I^J$ as a sum over permutations, each term, regarded as a polynomial in $a_{i_1},a_{i_2},\dotsc,a_{i_s}$, has total degree $(j_1-1)+\cdots+(j_s-1)$. Therefore, the total degree of $\det \tilde{B}(n)_I^J$ is at most $(j_1-1)+\cdots+(j_s-1)$.

However, note that when any two of $a_{i_1}+i_1,a_{i_2}+i_2,\dotsc,a_{i_s}+i_s$ are equal, there are two identical rows in $\tilde{B}(n)_I^J$, so its determinant is zero. This means that  
\begin{equation}
\det \tilde{B}(n)_I^J=P(a_{i_1},\dotsc,a_{i_s})\prod_{1\leq u<v\leq s} [(a_{i_v}+i_v)-(a_{i_u}+i_u)],
\label{eey}
\end{equation}
where $P$ is a polynomial of total degree at most $(j_1-1)+\cdots(j_s-1)-{s\choose2}=(j_1-1)+\cdots(j_s-s)$.
When specializing back all $a_i$'s to $a$, the degree in $a$ of the product on the right hand side of \eqref{eey} becomes zero. This completes the proof. \end{proof}

%
%
%


\section{Proof of  Theorem \ref{tbc}}

The correlation of the core and its three satellites could also be measured exclusively via rotationally symmetric tilings, by defining
\begin{equation}
\label{efaa}
\omega_r(a,b,k):=\lim_{n\to\infty}\frac{\M_r(S_{2n,a,b,k})}{\M_r(S_{2n,a,b,0})}.
\end{equation}

The asymptotics of this correlation is given in the following result.

\begin{prop} For non-negative $a$, $b$ and $k$ with $a$ even, we have
\label{tfa}
\begin{equation}
\label{efa}
\omega_r(a,b,k)\sim
3^{b^2/4}\frac
{G\left(\frac{b}{2}+1\right)^{\!2}}
{
\left\{
\dfrac{\Gamma(\frac{a}{6}+\frac{b}{2}+\frac13)}{\Gamma(\frac{a}{6}+\frac{b}{2}+\frac23)}
\dfrac{\Gamma(\frac{a}{6}+\frac23)}{\Gamma(\frac{a}{6}+\frac13)}
\dfrac{G(\frac{a}{2}+\frac{3b}{2}+1)}{G(\frac{a}{2}+1)}
\right\}^{2/3}
}
\,\,k^{b(a+b)/2},\ \ \ k\to\infty.
\end{equation}
%

%

\end{prop}

\begin{proof} We use the formulas for $\M_r(S_{n,a,b,k})$ provided in \cite{LR}. Since these are quite different for even and odd $b$, we distinguish between these two cases. Throughout this proof, $n$ is even (this can be assumed without loss of generality, as \eqref{efaa} only involves even values of the $n$-parameter of $S_{n,a,b,k}$).

{\it Case 1: $b$ even}. 
The $120^\circ$-rotationally invariant tilings of $S_{n,a,b,k}$ can be identified with the perfect matchings of the quotient graph $G$ of its planar dual under the action of the group generated by a $120^\circ$ rotation (see e.g. \cite{pptwo} for a detailed discussion of the case $b=0$, which readily adapts to the case of general $b$). This quotient graph $G$ is a bipartite planar graph that can be drawn in the plane so that it is symmetric about an axis. Thus, the factorization theorem of \cite{FT} can be applied to it. The resulting two ``halves'' are planar duals of regions whose lozenge tilings were enumerated by Lai and Rohatgi in \cite{LR}. The statement of the factorization theorem of \cite{FT} then yields
\begin{equation}
\label{efc}
\M_r(S_{n,a,b,k})=2^{n+b}
P_1\left(\frac{a}{2}+1,k,\frac{n}{2}-k-1,\frac{b}{2}\right)
P_2\left(\frac{a}{2}+1,k,\frac{n}{2}-k,\frac{b}{2}\right),
\end{equation}
where, cf. \cite{LR} (see equation (2.4) and (2.5) there), $P_1$ and $P_2$ are given by
\begin{align}
\label{efd}
&
P_1(x,y,z,a):=\frac{1}{2^{y+z}}\prod_{i=1}^{y+z}\frac{(2x+6a+2i)_i(x+3a+2i+\frac12)_{i-1}}{(i)_i(x+3a+i+\frac12)_{i-1}}
\nonumber
\\[5pt]
\times
&
\prod_{i=1}^a\frac
{(z+i)_{y+a-2i+1}(x+y+2z+2a+2i)_{2y+2a-4i+2}(x+3i-2)_{y-i+1}(x+3y+2i-1)_{i-1}}
{(i)_y(y+2z+2i-1)_{y+2a-4i+3}(2z+2i)_{y+2a-4i+1}(x+y+z+2a+i)_{y+a-2i+1}}
\end{align}
and
\begin{align}
\label{efe}
&\!\!\!\!\!\!
P_2(x,y,z,a):=\frac{(\frac{x}{2}+\frac{3y}{2})_a(x+2y+z+2a)_a}{2^{2a}(\frac{x}{2}+\frac{3y}{2}+z+a+\frac12)_a}
\frac{1}{2^{y+z}}
\prod_{i=1}^{y+z}\frac{(2x+6a+2i-2)_{i-1}(x+3a+2i-\frac12)_i}{(i)_i(x+3a+i-\frac12)_{i-1}}
\nonumber
\\[5pt]
\times
&\ \ \ \
\prod_{i=1}^a\frac
{(z+i)_{y+a-2i+1}(x+y+2z+2a+2i-1)_{2y+2a-4i+3}(x+3i-2)_{y-i}(x+3y+2i-1)_{i-1}}
{(i)_y(y+2z+2i-1)_{y+2a-4i+3}(2z+2i)_{y+2a-4i+1}(x+y+z+2a+i-1)_{y+a-2i+2}}.
\end{align}
Combining \eqref{efc} with its $k=0$ specialization we get
\begin{equation}
\label{eff}
\frac{\M_r(S_{n,a,b,k})}{\M_r(S_{n,a,b,0})}
=
\frac{P_1(\frac{a}{2}+1,k,\frac{n}{2}-k-1,\frac{b}{2})}{P_1(\frac{a}{2}+1,0,\frac{n}{2}-1,\frac{b}{2})}
\frac{P_2(\frac{a}{2}+1,k,\frac{n}{2}-k,\frac{b}{2})}{P_2(\frac{a}{2}+1,0,\frac{n}{2},\frac{b}{2})}.
\end{equation}
It is easy to see that, for fixed $a$, $b$ and $k$, as $n\to\infty$ we have
\begin{equation}
\label{efg}
\lim_{n\to\infty}\frac{P_1(\frac{a}{2}+1,k,\frac{n}{2}-k-1,\frac{b}{2})}{P_1(\frac{a}{2}+1,0,\frac{n}{2}-1,\frac{b}{2})}
=
\prod_{i=1}^{b/2}
\frac{1}{(i)_k}
\frac{(\frac{a}{2}+3i-1)_{k-i+1}}{(\frac{a}{2}+3i-1)_{-i+1}}
\frac{(\frac{a}{2}+3k+2i)_{i-1}}{(\frac{a}{2}+2i)_{i-1}}
\end{equation}
and
\begin{equation}
\label{efh}
\lim_{n\to\infty}\frac{P_2(\frac{a}{2}+1,k,\frac{n}{2}-k,\frac{b}{2})}{P_2(\frac{a}{2}+1,0,\frac{n}{2},\frac{b}{2})}
=
\frac{(\frac{a}{4}+\frac{3k}{2}+\frac12)_{b/2}}{(\frac{a}{4}+\frac12)_{b/2}}
\prod_{i=1}^{b/2}
\frac{1}{(i)_k}
\frac{(\frac{a}{2}+3i-1)_{k-i}}{(\frac{a}{2}+3i-1)_{-i}}
\frac{(\frac{a}{2}+3k+2i)_{i-1}}{(\frac{a}{2}+2i)_{i-1}}.
\end{equation}
Combining \eqref{efc}, \eqref{efg} and \eqref{efh} we obtain
\begin{align}
\label{efi}
&
\omega_r(a,b,k)=\lim_{n\to\infty}\frac{\M_r(S_{n,a,b,k})}{\M_r(S_{n,a,b,0})}
=
\nonumber
\\[5pt]
&\ \ \ \ \ \ \ \ \ \ \ \ 
\frac{(\frac{a}{4}+\frac{3k}{2}+\frac12)_{b/2}}{(\frac{a}{4}+\frac12)_{b/2}}
\prod_{i=1}^{b/2}\frac{1}{[(i)_k]^2}
\frac{(\frac{a}{2}+3i-1)_{k-i}(\frac{a}{2}+3i-1)_{k-i+1}}{(\frac{a}{2}+3i-1)_{-i}(\frac{a}{2}+3i-1)_{-i+1}}
\left[
\frac{(\frac{a}{2}+3k+2i)_{i-1}}{(\frac{a}{2}+2i)_{i-1}}
\right]^2.
\end{align}
To finish proving this case, we need to determine the asymptotics of the right hand side above as $k\to\infty$.

As $a$ and $b$ are fixed, we have
\begin{equation}
\label{efj}
\frac{(\frac{a}{4}+\frac{3k}{2}+\frac12)_{b/2}}{(\frac{a}{4}+\frac12)_{b/2}}
\sim
\frac{(\frac32)^{b/2}}{(\frac{a}{4}+\frac12)_{b/2}}\,k^{b/2},\ \ \ k\to\infty.
\end{equation}
Expressing the factor in the product in \eqref{efi} in terms of Gamma functions via formula \eqref{ebc}, and using that for any fixed $a$ and $b$ we have
\begin{equation}
\label{efk}
\frac{\Gamma(x+a)}{\Gamma(x+b)}\sim x^{a-b},\ \ \ x\to\infty
\end{equation}
(see e.g. equation (5.02) on page 119 of \cite{Olver}), we are led to
\begin{align}
\label{efl}
&
\frac{1}{[(i)_k]^2}
\frac{(\frac{a}{2}+3i-1)_{k-i}(\frac{a}{2}+3i-1)_{k-i+1}}{(\frac{a}{2}+3i-1)_{-i}(\frac{a}{2}+3i-1)_{-i+1}}
\left[
\frac{(\frac{a}{2}+3k+2i)_{i-1}}{(\frac{a}{2}+2i)_{i-1}}
\right]^2
\sim
\nonumber
\\[10pt]
&\ \ \ \ \ \ \ \ \ \ \ \ \ \ \ \ \ \ \ \ \ \ \ \ \ \ \ \ \ \ \ \ \ \ \ \ \ \ \ \ \ \ \ \ \ \ \ \ 
3^{2(i-1)}\left(\frac{a}{2}+2i-1\right)\frac{[\Gamma(i)]^2}{[\Gamma(\frac{a}{2}+3i-1)]^2}\,k^{a+4i-3},\ \ \ k\to\infty.
\end{align}
Using \eqref{efj} and \eqref{efl} in equation \eqref{efi}, we arrive at
\begin{equation}
\omega_r(a,b,k)\sim
3^{\frac{b^2}4}\prod_{i=1}^{b/2}\frac{\Gamma(i)^2}{\Gamma\left(\frac{a}{2}+3i-1\right)^2}
k^{ab/2+b^2/2}
,\ \ \ k\to\infty.
\label{efla}
\end{equation}
Clearly, $\prod_{i=1}^{b/2}\Gamma(i)=G(\frac{b}{2}+1)$. Furthermore, it is a straightforward exercise to show that
\begin{equation}
\label{eflb}
\prod_{i=1}^{b/2}{\Gamma\left(\frac{a}{2}+3i-1\right)}
=
\left\{
\dfrac{\Gamma(\frac{a}{6}+\frac{b}{2}+\frac13)}{\Gamma(\frac{a}{6}+\frac{b}{2}+\frac23)}
\dfrac{\Gamma(\frac{a}{6}+\frac23)}{\Gamma(\frac{a}{6}+\frac13)}
\dfrac{G(\frac{a}{2}+\frac{3b}{2}+1)}{G(\frac{a}{2}+1)}
\right\}^{1/3}.
\end{equation}
Using this, \eqref{efla} can, after some manipulation, be rewritten as \eqref{efa}.

{\it Case 2: b odd.} 
In the same fashion as we obtained \eqref{eff} for even $b$, we get for odd $b$ that
\begin{equation}
\label{efm}
\frac{\M_r(S_{n,a,b,k})}{\M_r(S_{n,a,b,0})}
=
\frac{F_1(\frac{a}{2}+1,k,\frac{n}{2}-k-1,\frac{b+1}{2})}{F_1(\frac{a}{2}+1,0,\frac{n}{2}-1,\frac{b+1}{2})}
\frac{F_2(\frac{a}{2}+1,k,\frac{n}{2}-k,\frac{b-1}{2})}{F_2(\frac{a}{2}+1,0,\frac{n}{2},\frac{b-1}{2})},
\end{equation}
where $F_1$ and $F_2$ are given by formulas (2.6) and (2.7) of \cite{LR} (these are more involved expressions than the ones for $P_1$ and $P_2$, and to keep the focus we do not list them here). Just as it was the case with equations \eqref{efg} and \eqref{efh}, it is straightforward to see that
\begin{align}
\label{efn}
&
\lim_{n\to\infty}\frac{F_1(\frac{a}{2}+1,k,\frac{n}{2}-k-1,\frac{b+1}{2})}{F_1(\frac{a}{2}+1,0,\frac{n}{2}-1,\frac{b+1}{2})}
=
\frac{1}{2^k(\frac{a}{4}+\frac{b}{2}+\frac{k}{2}+\frac12)_k}
\nonumber
\\[10pt]
&\ \ \ 
\times
\prod_{i=1}^{\lfloor (b+1)/6 \rfloor}\frac{(\frac{a}{2}+3k+6i-2)_{3(b+1)/2-9i+1}}{(\frac{a}{2}+6i-2)_{3(b+1)/2-9i+1}}
\prod_{i=1}^{\lfloor (b-1)/6 \rfloor}\frac{\frac{a}{2}+6i-1}{\frac{a}{2}+6i+3k-1}
\prod_{i=1}^{(b-1)/2}\frac{(\frac{a}{2}+3i-1)_{k-i+1}}{(\frac{a}{2}+3i-1)_{-i+1}}
\nonumber
\\[10pt]
&\ \ \ \ \ \ 
\times
\prod_{i=1}^k \frac{\Gamma(\frac{b}{2}+i+\frac32)}{\Gamma(i+\frac32)}
\frac{\Gamma(\frac{a}{2}+\frac{3b}{2}+3i-1)}{\Gamma(\frac{a}{2}+\frac{3b}{2}+3i-\frac52)}
\frac{1}{(i)_{(b+3)/2}(i+\frac32)_{(b-3)/2}}
\end{align}
and
\begin{align}
\label{efo}
&
\lim_{n\to\infty}\frac{F_2(\frac{a}{2}+1,k,\frac{n}{2}-k,\frac{b-1}{2})}{F_2(\frac{a}{2}+1,0,\frac{n}{2},\frac{b-1}{2})}
=
\frac
{\prod_{i=1}^{\lfloor (k+1)/3 \rfloor} (\frac{a}{2}+3i-1)_{3k-9i+4}}
{\prod_{i=1}^{\lfloor k/3 \rfloor} \frac{a}{2}+3k-6i+1}
\nonumber
\\[10pt]
&\ \ \ \ \ \ \ \ \ \ \ \ \ \ \ \ \ \ \ \ \ \ \ \ 
\times
\prod_{i=1}^k \frac{\Gamma(\frac{b}{2}+i+\frac32)}{\Gamma(i+\frac32)}
\frac{\Gamma(\frac{a}{2}+\frac{3b}{2}+3i-1)}{\Gamma(\frac{a}{2}+b+k+2i-1)}
\frac{1}{(i)_{(b+3)/2}(i+\frac32)_{(b-3)/2}}.
\end{align}
Looking back at \eqref{efm}, we see that we need to determine the $k\to\infty$ asymptotics of the expressions \eqref{efn} and  \eqref{efo}.

One new feature we have now is that these expressions contain products whose upper limit involves $k$. Because of this, in addition to \eqref{efk} we also need to use the asymptotics of the Barnes $G$-function given by \eqref{GK}.

The details of the calculations depend on the residue of $a$ modulo 3. If $a$ is a multiple of~3, since by assumptions $a$ is even, is is in fact a multiple of 6. Writing then $6a$ for the size of the core, and $2b+1$ for the size of the satellite (as the latter is assumed odd in the current case), we obtain after some straightforward if lengthy manipulations that the $k\to\infty$ asymptotics of the expression on the right in \eqref{efn} (in which $a$ is replaced by $6a$ and $b$ by $2b+1$) is
\begin{align}
\label{efp}
&\!\!\!\!
\frac{2^{41/36}\pi e^{1/12}}{3^{23/24}A\Gamma\left(\frac23\right)^2}
\nonumber
\\[10pt]
&
\times
\frac
{\prod_{i=0}^b\Gamma(i+\frac12)\prod_{i=1}^{a+b}\Gamma(3i-1)\prod_{i=1}^{\lfloor b/3 \rfloor}(3a+6i-1)}
{\prod_{i=0}^{a+b}\Gamma(3i+\frac12)\prod_{i=1}^{b}\Gamma(3a+2i)\prod_{i=1}^{\lfloor (b+1)/3 \rfloor}(3a+6i-2)_{3b-9i+4}}
\,3^{b^2/2}k^{3ab+b^2+3a/2+b+1/4},
\end{align}
while the $k\to\infty$ asymptotics of the expression on the right in \eqref{efo} (with $a$ is replaced by $6a$ and $b$ by $2b+1$) is
\begin{equation}
\label{efq}
\frac{2^{41/36}\pi e^{1/12}}{3^{11/24}A\Gamma\left(\frac23\right)^2}
\frac
{\prod_{i=1}^a\Gamma(3i-1)\prod_{i=0}^b\Gamma(i+\frac12)}
{\prod_{i=0}^{a+b}\Gamma(3i+\frac12)}
\,3^{b^2/2+b}k^{3ab+b^2+3a/2+b+1/4}.
\end{equation}
Then by \eqref{efm}--\eqref{efq} we obtain
\begin{align}
\omega_r(6a,2b+1,k)\sim
\sqrt{3}
&
\left[\frac{2^{41/36}\pi e^{1/12}}{3^{23/24}A\Gamma\left(\frac23\right)^2}\right]^2
\left[\frac{\prod_{i=0}^b\Gamma\left(i+\frac12\right)}{\prod_{i=0}^{a+b}\Gamma\left(3i+\frac12\right)}\right]^2
\frac
{\prod_{i=1}^{a}\Gamma(3i-1)\prod_{i=1}^{a+b}\Gamma(3i-1)}
{\prod_{i=1}^{b}\Gamma(3a+2i)}
\nonumber
\\[10pt]
\times
&
\frac
{\prod_{i=1}^{\lfloor b/3 \rfloor} 3a+6i-1}
{\prod_{i=1}^{\lfloor (b+1)/3 \rfloor} (3a+6i-2)_{3b-9i+4}}
3^{b(b+1)}k^{6ab+2b^2+3a+2b+1/2},\ \ \ k\to\infty.
\label{efr}
\end{align}
After some manipulation, using the recurrence \eqref{ebl} and the value of $G(1/2)$ given by \eqref{ebm}, one sees that \eqref{efr} can be written in terms of the Barnes $G$-function as
\begin{align}
&
\omega_r(6a,2b+1,k)\sim
3^{(2b+1)^2/4}
\frac
{G\left(\frac{2b+1}{2}+1\right)^2}
{
\left\{
\dfrac{\Gamma(a+\frac{2b+1}{2}+\frac13)}{\Gamma(a+\frac{2b+1}{2}+\frac23)}
\dfrac{\Gamma(a+\frac23)}{\Gamma(a+\frac13)}
\dfrac{G(3a+\frac{3(2b+1)}{2}+1)}{G(3a+1)}
\right\}^{2/3}
}
\nonumber
\\[10pt]
&\ \ \ \ \ \ \ \ \ \ \ \ \ \ \ \ \ \ \ \ \ \ \ \ \ \ \ \ \ \ \ \ \ \ \ \ \ \ \ \ \ \ \ \ \ \ \ \ \ \ \ \ \ \ \ \ 
\times
k^{(2b+1)(6a+2b+1)/2},\ \ \ k\to\infty.
\label{efra}
\end{align}
The remaining cases, when the size of the core is of the form $6a+2$ or $6a+4$ for some integer $a$, are handled similarly. Together they prove that for all even core sizes $a$ and odd satellite sizes $b$ we have
\begin{equation}
\label{efrb}
\omega_r(a,b,k)\sim
3^{b^2/4}
\frac
{G\left(\frac{b}{2}+1\right)^2}
{
\left\{
\dfrac{\Gamma(\frac{a}{6}+\frac{b}{2}+\frac13)}{\Gamma(\frac{a}{6}+\frac{b}{2}+\frac23)}
\dfrac{\Gamma(\frac{a}{6}+\frac23)}{\Gamma(\frac{a}{6}+\frac13)}
\dfrac{G(\frac{a}{2}+\frac{3b}{2}+1)}{G(\frac{a}{2}+1)}
\right\}^{2/3}
}
\,\,k^{b(a+b)/2},\ \ \ k\to\infty.
\end{equation}
This completes the proof.
\end{proof}

{\it Proof of Theorem \ref{tbc}.} 
Taking the limit as $n\to\infty$ in the statement of Conjecture \ref{tbaa}, it follows by \eqref{ebca} and \eqref{ebcb} that
\begin{equation}
\label{efs}
\frac{{\omega}(S_{n,a,b,k})}{{\omega}_r(S_{n,a,b,k})^3}=\left[\prod_{i=1}^k\frac{(a+6i-4)(a+3b+6i-2)}{(a+6i-2)(a+3b+6i-4)}\right]^2.
\end{equation}
One readily gets, using \eqref{efk}, that
\begin{equation}
\label{eft}
\prod_{i=1}^k\frac{(a+6i-4)(a+3b+6i-2)}{(a+6i-2)(a+3b+6i-4)}\to
\frac
{\Gamma(\frac{a}{6}+\frac23)\Gamma(\frac{a}{6}+\frac{b}{2}+\frac13)}
{\Gamma(\frac{a}{6}+\frac13)\Gamma(\frac{a}{6}+\frac{b}{2}+\frac23)},\ \ \ k\to\infty.
\end{equation}
Thus, the constant approached by the right hand side of \eqref{efs} as $k\to\infty$ precisely cancels the factors involving the Gamma function at the denominator in the cube of the right hand side of \eqref{efa}. Using \eqref{eft} and the expression \eqref{efa} for ${\omega}_r(S_{n,a,b,k})$, equation \eqref{efs} yields then formula~\eqref{ebn}.
\hfill $\square$





\section{Concluding remarks}

%










In this paper we presented an ``experiment'' designed to give the exact value of the correlation of a core and three satellite triangular holes. It relies on the first author's two decade old observation that if the satellites are enclosed symmetrically by a hexagon, the number of lozenge tilings of the resulting region is round, and on its almost decade-old generalization that brings in the presence of the core. We also presented asymptotic consequences of our exact formulas, which include the verification of the electrostatic conjecture \cite[Conjecture 1]{ov} for the system of gaps consisting of the core and satellites (it was the special case of this when the core is empty that was the original motivation for this work). In fact, combining our results with those in \cite{3b}, we obtain a verification of \cite[Conjecture 1]{ov} for arbitrary triples of bowtie gaps arranged in a triad, a satisfying generalization of our motivating case. Other consequences we present include a strengthening of \cite[Conjecture 1]{ov} (by specifying exactly the multiplicative constant), an unexpected exact way to calibrate the hexagonal lattice against the square lattice so that the monomer-monomer correlations decay at precisely the same rate, and a heuristic derivation of the special value $G(1/2)$ of the Barnes $G$-function.

%
%
%
%
%
%
%
%






\bigskip


%
%
%
%
%
%
%
%
%


\end{document}